\documentclass[12pt]{amsart}
\usepackage{amsmath,amssymb,amsthm,mathrsfs} %,constants} %,marvosym}
\usepackage[numbers,sort&compress,square,comma]{natbib}

\usepackage{scalerel}    %%%%  used first for scaling \bullet for missing argument in a function 

\usepackage{enumitem}  

\usepackage{mathabx}   %% made \widecheck work  but caused problems 2017-05-01 

\usepackage{graphicx,color,tikz}
\usetikzlibrary{decorations.pathreplacing,arrows,decorations.markings}
\usepackage{pgfplots}
\RequirePackage[colorlinks,urlcolor=my-blue,linkcolor=my-red,citecolor=my-green]{hyperref}
\definecolor{my-blue}{rgb}{0.0,0.0,0.6}
\definecolor{my-red}{rgb}{0.5,0.0,0.0}
\definecolor{my-green}{rgb}{0.0,0.5,0.0}
\definecolor{nicos-red}{rgb}{0.75,0.0,0.0}
\definecolor{nicos-red}{rgb}{0.75,0.0,0.0}
\definecolor{light-gray}{gray}{0.6} 
\definecolor{really-light-gray}{gray}{0.8}
\definecolor{darkgreen}{rgb}{0.0,0.5,0.0}
\definecolor{darkblue}{rgb}{0.0,0.0,0.3}
\definecolor{nicosred}{rgb}{0.65,0.1,0.1}
\numberwithin{figure}{section}

\usepackage[left=1in,top=1in,right=1in,bottom=1in]{geometry}
\newtheorem{theorem}{\sc Theorem}[section]
\newtheorem{lemma}[theorem]{\sc Lemma}
\newtheorem{proposition}[theorem]{\sc Proposition}
\newtheorem{corollary}[theorem]{\sc Corollary}

\numberwithin{equation}{section}
\theoremstyle{remark}
\newtheorem{remark}[theorem]{Remark}
\newtheorem{example}[theorem]{Example}

\newcommand{\be}{\begin{equation}}
\newcommand{\ee}{\end{equation}}
\newcommand{\nn}{\nonumber}

 \def\esssup{\mathop{\mathrm{ess\,sup}}}

\def\p{\omega}
\def\pch{\widecheck\p}

\def\doob{\kappa}

\def\wt{\widetilde}
\def\w{\omega}

\def\wplus{\bar\w}

\def\Uset{\mathcal U}

\def\pich{\widecheck\pi}
\providecommand{\abs}[1]{\vert#1\vert}
\newcommand{\fl}[1]{\lfloor{#1}\rfloor} 

\def\E{\mathbb{E}}
\def\P{\mathbb{P}}

\def\wcP{\widecheck\bfP}
\def\wcE{\widecheck\bfE}
\def\bfP{\mathbf P}
\def\bfE{\mathbf E}
\def\bC{\mathbb{C}}  \def\bH{\mathbb{H}}
\def\N{\mathbb{N}}
\def\Z{\mathbb{Z}}
\def\R{\mathbb{R}}
\def\S{\mathbb{S}}
\def\Q{\mathbb{Q}}
\def\C{\mathbb{C}}
\def\B{\mathbb{B}}
\font \mymathbb = bbold10 at 11pt
\newcommand{\one}{\mbox{\mymathbb{1}}}

\def\e{\varepsilon}
\def\Beta{W}

\def\cO{\mathcal O}

\def\ri{{\rm{ri}}\,}
\def\Bd{\rho}
\def\Bdla{\Bd^\lambda}

\def\Pplus{{\overline\P}}

\def\Eplus{{\overline\E}}
\def\kS{\mathfrak{S}}

\def\Varplus{\overline{\mathbb V}{\rm ar}}

 \def\Vvv{{\rm\mathbb{V}ar}}   

\def\Gvar{\Gamma}   %gamma r.v.  

\def\Xmin{X^{1,\rm min}}
\def\Xmax{X^{1,\rm max}}
\def\Ymin{X^{2,\rm min}}
\def\Ymax{X^{2,\rm max}}

\def\llnv{\xi^*} %{\bar\xi}

\def\OSP{(\Omega,\kS,\P)}

\newcommand\cbullet{{\scaleobj{0.6}{\bullet}}}

\def\ddd{\displaystyle}

\newcommand\HP[3]{F^{#1}_{#2,\,#3}}    %hitting probab with subscripts 
\newcommand\HPa[3]{F^{#1}(#2,#3)}    %hitting probab  F^\w(x,y)  

\usepackage{filemod} 

\begin{document}
%\usdate

%\title[Variational formulas for LPP]{Variational representations for last passage percolation\footnotetext{\today,\ \currenttime}}
%\title[Path fluctuations for the beta RWRE]
%{KPZ path fluctuation bounds %wandering exponent 
%for random walk\\ 
%in i.i.d.\ dynamic beta random environment} %\footnotetext{\today,\ \currenttime}}
\title[Beta RWRE]
{Large deviations and wandering exponent  
for random walk in a  dynamic beta environment} %\footnotetext{\today,\ \currenttime}}
%{Wandering exponent  for random walk in a  dynamic beta environment} %\footnotetext{\today,\ \currenttime}}
%{Localization and wandering exponent  for random walk in a  dynamic beta environment} %\footnotetext{\today,\ \currenttime}}

\author[M.~Bal\'azs]{M\'arton Bal\'azs}
\address{M\'arton Bal\'azs\\ School of Mathematics\\ University of Bristol\\ University Walk\\ Bristol, BS8 1TW\\ United Kingdom.}
\email{m.balazs@bristol.ac.uk}
\urladdr{http://www.maths.bris.ac.uk/~mb13434}
\thanks{M.\ Bal\'azs was partially supported by the Hungarian National Research, Development and Innovation Office, NKFIH grant K109684.}

\author[F.~Rassoul-Agha]{Firas Rassoul-Agha}
\address{Firas Rassoul-Agha\\ University of Utah\\  Mathematics Department\\ 155S 1400E\\   Salt Lake City, UT 84112\\ USA.}
\email{firas@math.utah.edu}
\urladdr{http://www.math.utah.edu/~firas}
\thanks{F.\ Rassoul-Agha was partially supported by National Science Foundation grant DMS-1407574 and Simons Foundation grant 306576.}

\author[T.~Sepp\"al\"ainen]{Timo Sepp\"al\"ainen}
\address{Timo Sepp\"al\"ainen\\ University of Wisconsin-Madison\\  Mathematics Department\\ Van Vleck Hall\\ 480 Lincoln Dr.\\   Madison WI 53706-1388\\ USA.}
\email{seppalai@math.wisc.edu}
\urladdr{http://www.math.wisc.edu/~seppalai}
\thanks{T.\ Sepp\"al\"ainen was partially supported by  National Science Foundation grant  DMS-1602486, by Simons Foundation grant 338287,  and by the Wisconsin Alumni Research Foundation.} 

\thanks{F.\ Rassoul-Agha is grateful for the hospitality of the School of Mathematics  at the University of Bristol, where this project started.
Part of this work was completed during the 2017 Park City Mathematics Institute, supported by NSF Grant DMS-1441467. }

%\date{Last modified on \today\ at\ \filemodprinttime{\jobname}}
\date{February 5, 2018}

\begin{abstract}  
Random walk in a dynamic  i.i.d.\ beta random environment, conditioned  to escape at an atypical velocity,  converges to  a Doob transform of the original walk.   The Doob-transformed environment is correlated in time,  i.i.d.\ in space, and its marginal density function is a product of a beta density and a hypergeometric function.   Under its averaged distribution  the transformed walk  obeys the wandering exponent 2/3  that agrees with Kardar-Parisi-Zhang universality.    The  harmonic function in the Doob transform comes from   a Busemann-type limit and  appears as an extremal  in  
a variational problem  for the quenched large deviation rate function.  
%Along the way we construct the stationary beta polymer and prove fluctuation bounds for it. The Doob conditioned RWRE is in duality with this polymer and as a result it has a KPZ wandering exponent of $2/3$.
\end{abstract}
\maketitle

\setcounter{tocdepth}{2}
\tableofcontents

\section{Introduction} 
We study an exactly solvable  version of  random walk in a random environment (RWRE)  in one space dimension.  The walk is nearest-neighbor  and the environment dynamical and  product-form.  Our main results (i) construct a Doob transform of the   RWRE   that  conditions the walk on an atypical velocity,  (ii) establish that the transformed walk   has path fluctuation exponent 2/3 of the KPZ (Kardar-Parisi-Zhang) class  instead of the diffusive 1/2, and (iii) describe the quenched large deviation rate function of the walk.   

 The three points above are closely tied together.    The  harmonic functions in the Doob transform furnish extremals of a variational formula for the quenched large deviation rate function.   Explicit distributional properties of these harmonic functions enable the derivation of the path exponent.  The logarithm of the harmonic function itself  obeys the KPZ  longitudinal exponent 1/3.   

This work  rests on the development of analogues of   percolation and  polymer ideas   for  RWRE.  The  harmonic functions in the Doob transform arise through limits that correspond to Busemann functions of  percolation and polymers.  The quenched large deviation rate function is strictly above the averaged one except at their common minimum.  For a deviation of small order $h$, the difference of  the quenched and averaged rate functions is of  order $h^4$.  These properties are  exactly as for the quenched and averaged free energy of 1+1 dimensional directed polymers  \cite{Com-Var-06, Lac-10}.  
% This agrees with  small-$\beta$ asymptotics for the polymer free energy  
  As is the case for the entire KPZ class,  proofs of fluctuation exponents are restricted to models with special features.  A natural  expectation is that   the picture that emerges here should be universal for 1+1 dimensional directed RWRE under some assumptions.

% Variational problems for polymers and percolation that involve Busemann functions now represent large deviation rate functions.  

 We turn to a detailed introduction of the model.  
  
A dynamical  environment is   refreshed at each  time step.   On the  two-dimensional space-time lattice $\Z^2$ we run time in  the diagonal direction $(\tfrac12,\tfrac12)$,   and   the admissible steps of the walk are $e_1=(1,0)$ and $e_2=(0,1)$.  The jump probabilities are independent and identically distributed at each lattice point of $\Z^2$.   When the walk starts at the origin, after  $n$ time steps  its location is among the points $(i,j)$ in the first quadrant (that is,  $i,j\ge 0$)   with $i+j=n$.  

%The precise definition of the model goes as follows.  
   The {\it environment}  $\p=(\p_{x,\,x+e_1}:x\in\Z^2)$ is a collection of   i.i.d.\ $[0,1]$-valued random variables $\p_{x,\,x+e_1}$ indexed by lattice points $x$. Set $\p_{x,\,x+e_2}=1-\p_{x,\,x+e_1}$.   $(\p_{x,\,x+e_1},  \p_{x,\,x+e_2})$ are the jump probabilities from point $x\in\Z^2$ to one of the neighbors $\{x+e_1,x+e_2\}$.   Transitions $\p$ do not allow backward jumps.     The distribution of the environment $\p$ is   $\P$ with expectation operator $\E$.
Given a realization $\p$ and a point $x\in\Z^2$,   $P_x^\p$ denotes the {\it quenched}  path measure of  the Markov chain $(X_n)_{n\ge 0}$  on $\Z^2$ that starts at $x$ and uses transition probabilities $\p$:
\be\label{12}	\begin{aligned}
	&P_x^\p(X_0=x)=1\quad\text{and, for } \ y\in\Z^2,\ n\ge0, \  \text{ and } \ i\in\{1,2\},   \\
	&P_x^\p(X_{n+1}=y+e_i\,|\,X_n=y)=\p_{y,y+e_i}.
	\end{aligned}\ee
Precisely speaking,  $P^\p_x$ is a probability measure on the path space $(\Z^2)^{\Z_+}$ of the walk and $X_\cbullet$ is the coordinate process.   
This is a special case of   {\it  random walk in a space-time random environment}. 
%Measure $P_x^\p$ is called the {\it  quenched} measure while $P_x=\int P_x^\p \P(d\p)$ is called the {\it  averaged} measure. The latter is sometimes also called the {\it  annealed} measure.

This paper focuses on the {\it beta RWRE} where $\p_{x,\,x+e_1}$ is beta-distributed.    Barraquand and Corwin  \cite{Bar-Cor-17}  discovered  that this special case is {\it exactly solvable}.   This means that  fortuitous coincidences of combinatorics and probability permit   derivation of   explicit formulas and precise results far deeper than anything presently available for the general case.   Some limit results  uncovered in an  exactly solvable case are expected to be universal.  These then form natural conjectures to be investigated in the general case.  

An earlier  case of exact calculations for  RWRE in a  static environment appeared in a series of papers by Sabot and coauthors (see \cite{Sab-Tou-17} and references therein).    They discovered and utilized  special features of the multidimensional Dirichlet RWRE to prove results currently not accessible for the general multidimensional RWRE.   Section 8 of  \cite{Sab-Tou-17} discusses  one-dimensional  RWRE in a static beta environment.

%The multidimensional generalization of the beta distribution is the Dirichlet distribution. In a series of papers, Sabot and coauthors have identified and used  special features of the multidimensional Dirichlet RWRE to prove results currently not accessible for the general multidimensional RWRE.  

  Before specializing to the dynamic beta environment, we  run through some known results for the general 1+1 dimensional  RWRE \eqref{12} in an i.i.d.\ environment.

\subsection{Nearest-neighbor space-time RWRE} \label{s:1.1} 
Under an i.i.d.\ environment for the quenched model in \eqref{12}, 
the {\it  averaged} path measure  $P_0(\cdot)=\int P_0^\p(\cdot) \,\P(d\p)$ is a classical random walk with admissible steps  $\{e_1,e_2\}$ and transition kernel  
$p(e_i)=\E(\p_{0,e_i})$, $i=1,2$. 
%$p(e_1)=\llnv_1=\E(\p_{0,e_1})$ and $p(e_2)=\llnv_2=\E(\p_{0,e_2})$. 
Hence there is a law of large numbers   $P_0\{X_N/N\to\llnv\}=1$ with limiting velocity $\llnv=(\llnv_1, \llnv_2)=(p(e_1), p(e_2))$.  
  Fubini's theorem   then gives the quenched law of large numbers  
	\begin{align}\label{lln}
	P_0^{\p}\Big\{\frac{X_N}N\to\llnv\Big\}=1\quad\text{for }\P\text{-a.e.\ }\p.
	\end{align}
	
By Donsker's invariance principle, under $P_0$ the centered and diffusively rescaled walk 
	\[\Big\{W_N(t)=\frac{X_{\fl{Nt}}-Nt\llnv}{\sqrt{\llnv_1\llnv_2N}}:t\ge0\Big\}\] 
 converges weakly to  $\{(W(t),-W(t)):t\ge0\}$, where $W(\cdot)$ is standard one-dimensional Brownian motion.

The same  invariance principle holds for the {\it  quenched} RWRE  if and only if $\P(\p_{0,e_1}\in\{0,1\})<1$. That is,  for $\P$-almost every $\p$
the distribution of $\{W_N(t):t\ge0\}$ under $P_0^\p$ converges weakly to that of $\{(W(t),-W(t)):t\ge0\}$ (Theorem 1 of \cite{Ras-Sep-05}). 
%Also, the corresponding point-to-line free energy is just the constant $\log P_0^\p\{\abs{X_n}_1=n\}=0$ and thus has 
%fluctuation exponent $0$.  %The KPZ relation between fluctuation exponents seems satisfied: $1/2$ for the path and $0$ for the energy. 
%\textcolor{blue}{(do we care about this last statement?)}
An invariance principle also holds for the quenched mean $E_0^\p[X_N]$, but with scaling $N^{1/4}$ 
(Corollary 3.5  of \cite{Bal-Ras-Sep-06}). 
%This places the RWRE in the Edwards-Wilkinson (EW) universality class. See \cite{Sep-10,Cor-12} for a review.
In summary, as $N\to\infty$,  the quenched mean of the walk has Gaussian fluctuations on a small scale of order $N^{1/4}$, while  under a typical   environment the walk  itself has Gaussian fluctuations on the larger scale of order $N^{1/2}$.
The fluctuations of the quenched walk  dominate and hence the averaged process has Gaussian fluctuations of order $N^{1/2}$.

%\textcolor{blue}{(The roles in the above summary are reversed in the degenerate case. There, the quenched walk has zero variance and all the fluctuations come from thequenched mean [though the quenched mean IS the process!]. But the fluctuations are still Gaussian and of order $N^{1/2}$. In KPZ, the path is localized but the quenched mean must be fluctuating on the order $N^{2/3}$ and giving the averaged process fluctuations of that order...)}

%For $x\in\R^2$ with $x\cdot(e_1+e_2)\in\Z_+$ let $[x]$ be the closest point to $x$ in $\{y\in\Z^2:(y-x)\cdot(e_1+e_2)=0\}$.
%It does not matter what norm we use, since $\{v\in\R^2:v\cdot(e_1+e_2)=0\}$ is a one-dimensional vector space.

Let  $\Uset=\{te_1+(1-t)e_2:0\le t\le1\}$ denote the simplex of possible limiting velocities.  For $\xi\in\Uset$ and $N\in\N$ let $[N\xi]$ denote a point closest to $N\xi$ on the antidiagonal $\{(x_1,x_2)\in\Z^2: x_1+x_2=N\}$.  
The averaged  large deviation principle (LDP) is the standard Cram\'er theorem and  tells us that for $\xi\in\Uset$ %=\{te_1+(1-t)e_2:0\le t\le1\}$
	\[\lim_{N\to\infty}N^{-1} \log P_0\{ X_N = [N\xi] \} =-I_a(\xi)\]
with rate function 
	\begin{align}\label{Ia}
	I_a(\xi)=\xi_1\log\frac{\xi_1}{\llnv_1}+\xi_2\log\frac{\xi_2}{\llnv_2} \qquad\text{for $\xi=(\xi_1,\xi_2)\in\Uset$.}
	\end{align}	

Under the assumption 
\be\label{qLDP3} \E[\abs{\log\p_{0,e_i}}^{2+\e}]<\infty\quad\text{ for $i\in\{1,2\}$ and some $\e>0$ }\ee
  a quenched   LDP  holds as well. 
By Theorems 2.2, 4.1, 2.6(b), and 3.2(a)  of \cite{Ras-Sep-14},  for all $\xi\in\Uset$, 
	\begin{align}\label{qLDP}
	\lim_{N\to\infty}N^{-1} \log P_0^\p\{ X_N = [N\xi] \} =-I_q(\xi)
	\end{align}
exists $\P$-almost surely. 
%Thm 2.2 says the p2p quenched probability limit is \La(g)-\La(g,\zeta).    (*)
%Thm 4.1 identifies the rate as \La(g) minus the usc regularization of \La(g,\zeta).
%This is the same as the above limit for \zeta in the relative interior. But a priori, this is not necessarily the same as the above
%limit at the boundary.
%Thm 2.6(b) says that this rate is the same as \La(g) minus the continuous extension (from \ri\Uset to \Uset) of \La(g,\zeta). 
%Still, a priori this is not necessarily the same as the limit (*), at the boundary.
%but then Thm 3.2 says \La(g,\zeta) is continuous (in the iid directed case) and hence the rate function is indeed the same as the limit (*), even on the boundary.
The rate function  $I_q$ does not depend on $\w$.  It is a nonnegative convex continuous function on $\Uset$ with a unique zero at $\llnv$.    By Fatou's lemma and Jensen's inequality,   $I_q(\xi)\ge I_a(\xi)$ for all $\xi\in\Uset$.    It is shown in \cite{Yil-Zei-10} that in fact $I_q(\xi)>I_a(\xi)$ for all $\xi\in\Uset\setminus\{\llnv\}$.   The proof in \cite{Yil-Zei-10} utilizes a uniform ellipticity assumption, namely that  $\P(\delta\le\p_{0,e_1}\le1-\delta)=1$   for some  $\delta>0$,  but their proof works more generally.  Theorem \ref{thm:Iq} below states the strict inequality in  the beta case.  
%However, they only use this assumption to ensure the finiteness of certain quantities and   having enough moments for $\log\p_{0,e_i}$, $i=1,2$, should be enough.)} \textcolor{red}{(how many moments?! $>2$, I presume?)} \textcolor{blue}{(The claim holds in the beta case; see the separate file named {\tt IqIa.tex})}

In general in RWRE closed formulas for $I_q$ have not been found.  Variational representations  exist, for example in \cite{Ros-06,Yil-09-cpam,Ras-Sep-14,Gre-Hol-94,Com-Gan-Zei-00+err}.  We state below one particular formula for the RWRE \eqref{12} on $\Z^2$.  In the beta case  extremals  for this formula   are  identified  in Section \ref{sec:ldp} below,  in terms of harmonic functions constructed for the beta RWRE.

  Let  $\mathcal K$ denote  the space of  {\it integrable stationary cocycles} defined on the probability space $\OSP$  of the environments.  By this  we mean   stochastic processes $\{B_{x,y}(\w):x,y\in\Z^2\}$  that satisfy these conditions   for all $x,y,z\in\Z^2$ and   $\P$-a.e.\ $\p$:  $\E\abs{B_{x,y}}<\infty$,   $B_{x,y}(\p)+B_{y,z}(\p)=B_{x,z}(\p)$, and   $B_{x,y}(T_z\p)=B_{x+z,y+z}(\p)$  where $T_z$ is the shift    $(T_z\p)_{x,\,x+e_i}=\p_{x+z,\,x+z+e_i}$.   The rate function in \eqref{qLDP} is then characterized as 
  	\begin{align}\label{K-var1}
	\begin{split}
	I_q(\xi)&=-\inf_{B\in\mathcal K}\Bigl\{\E[B_{0,e_1}]\xi_1+\E[B_{0,e_2}]\xi_2\\
	&\qquad+\P\text{-}\esssup_\w\log\bigl(\w_{0,e_1}e^{-B_{0,e_1}(\w)}+\w_{0,e_2}e^{-B_{0,e_2}(\w)}\bigr)\Bigr\}\qquad \text{for $\xi\in\ri\Uset$.}  
	\end{split}
	\end{align}
This formula for $I_q$  is valid  for an i.i.d.\ environment $\w$ under the same moment assumption \eqref{qLDP3} as the LDP.  	
	
For a nearest-neighbor RWRE on  $\Z^d$  for which all directions $\pm e_i$  satisfy \eqref{qLDP3}  formula \eqref{K-var1}  appeared in Theorem 2 on page 6 of \cite{Ros-06}.  In the directed case \eqref{K-var1} is a special case of variational formula   (4.7) in \cite{Geo-Ras-Sep-16-cmp}   for  the point-to-point limiting free energy of a directed polymer.  The  RWRE with transition probability $\w_{0,z}$ is obtained by taking the potential of the polymer to be  $V_0(\w,z)=\log \w_{0,z}-\log p(z)$ where $p(\cdot)$ is a fixed transition probability in the background. 
 
%	\be\label{B89a} B_{x,y}(T_z\p)=B_{x+z,y+z}(\p)\quad\text{and}\quad B_{x,y}(\p)+B_{y,z}(\p)=B_{x,z}(\p).\ee

%The variational formula in question is precisely stated as  There, we considered random polymer measures and the beta RWRE is covered as a special case: $\mathcal R=\{e_1,e_2\}$, inverse temperature equal to one,  $p(z)=1/2$, and $V_0(\w,z)=\log \w_{0,z}+\log2$. Rate function $I_q(\xi)$ is then equal to $-g_{\text{pp}}^1(\xi)$, where $g_{\text{pp}}^1$ is the point-to-point limiting free energy of the polymer.

%\textcolor{red}{Guillaume and Ivan find that 
%	\[c(\xi)=\lim_{n\to\infty}n^{-2/3}{\mathbb V}{\mathrm{ar}}\bigl(\log P_0^\p(X_n=n\xi)\bigr)\mathop{\longrightarrow}_{\xi\to\llnv}0.\]
%This prevents them from doing the expansions all the way to $\llnv$ and they believe the above is related to the matching of the rates.}
%\textcolor{blue}{(maybe mention that there is an exact formula in some special cases. reference \eqref{Iq}.
%what about the log gamma RWRE? rate is probably explicit there too! But it is not a dynamical environment..)}

\subsection{Beta RWRE} 
Let   $\alpha,\beta>0$ be positive real parameter values.  The standard 
 gamma  and beta  functions are given by 
	\be\label{BB} \Gamma(\alpha)=\int_0^\infty s^{\alpha-1} e^{-s}\,ds
	\quad\text{and}\quad 
	 B(\alpha,\beta)=\int_0^1 s^{\alpha-1}(1-s)^{\beta-1}\,ds=\frac{\Gamma(\alpha)\Gamma(\beta)}{\Gamma(\alpha+\beta)}.\ee
The c.d.f.\ of the Beta$(\alpha,\beta)$ distribution is 
	\be\label{be-cdf} F(t;\alpha,\beta)=B(\alpha,\beta)^{-1}\int_0^t s^{\alpha-1}(1-s)^{\beta-1}\,ds \qquad \text{for $0<t<1$.}  \ee  
The case $\alpha=\beta=1$ is the uniform distribution on $(0,1)$.  
	
 For the remainder of this paper,    the variables  $\{\p_{x,\,x+e_1}:x\in\Z^2\}$  in  the RWRE \eqref{12}  are  i.i.d.\ Beta$(\alpha,\beta)$ distributed.    

Barraquand and Corwin  \cite{Bar-Cor-17} showed that if $\alpha=\beta=1$ and $\xi_1-\xi_2>4/5$   then
\be\label{b-c}  \lim_{N\to\infty} \P\biggl\{  \frac{\log P_0^\p\{X_N\cdot(e_1-e_2)\ge N(\xi_1-\xi_2)\}+NI_q(\xi)}{c(\xi)N^{1/3}} \le y \biggr\}  = F_{\text{GUE}}(y)  \ee
where the limit  is the Tracy-Widom GUE  distribution.  
Later, in a less rigorous paper,Thiery and Le Doussal \cite{Thi-Dou-16-} did the same for $\log P_0^\p\{X_N=[N\xi]\}+NI_q(\xi)$ and all $\alpha,\beta>0$ and $\xi\ne\llnv$.

%Note that $\log P_0^\p\{X_N=x\}$ is the point-to-point free energy if one looks at the RWRE as a random polymer with a step-dependent potential $\log\p_{0,z}$, $z\in\{e_1,e_2\}$. 

These results revealed  that  this type of RWRE possesses features of the  1+1 dimensional Kardar-Parisi-Zhang (KPZ) universality class.   
A natural next question therefore  is, where in the model do we  find the KPZ  wandering exponent 2/3?   It is not in  the walk \eqref{12},  because,  as   pointed out in Section \ref{s:1.1}, the walk in an i.i.d.\ environment  satisfies a standard CLT under both its quenched and averaged distributions.   

%%%%  How far should we go in saying that once one takes the polymer view of the RWRE, this becomes sort of obvious? Or is it?  
 
 We answer  the question by  conditioning the walk on an  atypical velocity.  Under this conditioning the quenched process $X_{\cbullet}$   % $X_\bcdot$  
converges to a  new walk given by a Doob transform of the original walk.   The harmonic function in the transform is the exponential of an analogue of a Busemann function for RWRE.     The Doob transform  is a  random walk in a  correlated environment.  When the environment is averaged out,  at time $N$   this walk has  fluctuations of the order $N^{2/3}$ and thus has the KPZ wandering exponent.  
  This behavior deviates radically  from that of classical random walk:  standard  random walk conditioned on an atypical velocity converges to another random walk with transitions altered to produce the new mean.  

Conditioning on an atypical velocity is   intimately tied to large deviations. The  logarithm of  the  harmonic function in  the Doob transform turns out to be  an extremal in \eqref{K-var1} and its  expectation is  the  gradient 
of   $I_q$.

%\note{T: I propose elimination of the text below.} 
%As such, we will be interested in the asymptotic behavior of the RWRE, conditioned on $X_N=[N\xi]$ for a given $\xi\in\Uset$ and $N\to\infty$. This behavior may differ depending on what measure we use.
%
%There are three possible measures to consider:
%\begin{enumerate}[label=(\roman{*}), ref=\roman{*}] %[\ \ {\rm(}i{\rm)}]
%\item the {\it  quenched} measure: $P_0^\p(\,\cdot\,|\,X_N=[N\xi])$,
%\item the {\it  averaged} measure: $\int P_0^\p(\,\cdot \,|\,X_N=[N\xi])\,\P(d\p)$, and
%\item the {\it  annealed} measure: $P_0(\,\cdot\,|\,X_N=[N\xi])$.
%\end{enumerate}
%(See (1.3-1.5) in \cite{Hol-09} for the analogues of these definitions for random polymer measures.)
%
%
%In the annealed case, the theory of large deviations tells us that after conditioning on $X_N=[N\xi]$ and letting $N\to\infty$
%the process converges to another classical random walk with transitions $p'(e_i)=\xi_i$, $i\in\{1,2\}$. See for example Exercise 6.19 in \cite{Ras-Sep-15-ldp}. 
%Donsker's invariance principle applies and the new random walk is diffusive. In other words, there is nothing new to explore in this case. We therefore study the quenched and averaged measures.

\subsubsection*{Notation and conventions} We collect here some notation for easy reference.  
$\Z$ denotes the integers,   $\Q$  the rationals,  $\R$  the reals, and $\C$ the complex numbers.  
$\Z_+=\{0,1,2,3,\dotsc\}$, $\N=\{1,2,3,\dotsc\}$,  and $\R_+=[0,\infty)$.       For real  $a$,  $\fl{a}$ is the largest integer $\le a$.  

 For $x,y\in\R^2$  we use the following conventions.  Vector notation  is $x=(x_1,x_2)=x_1e_1+x_2e_2$, with canonical basis vectors $e_1=(1,0)$ and $e_2=(0,1)$.      The scalar product is  $x\cdot y$ and the $\ell^1$ norm 
$\abs{x}_1=\abs{x_1}+\abs{x_2}$.
Integer parts are  taken coordinatewise:  $\fl{x}=(\fl{x_1},\fl{x_2})$.
For   $x\cdot(e_1+e_2)\in\Z_+$,  $[x]$ is  a closest point to $x$ in $\{y\in\Z^2: y_1+y_2=x_1+x_2\}$.
%It does not matter what norm we use, since $\{v\in\R^2:v\cdot(e_1+e_2)=0\}$ is a one-dimensional vector space.
Inequality $y\ge x$ is interpreted  coordinatewise:  $y_1\ge x_1$ and $y_2\ge x_2$.

Shifts or translations $T_z$ act on environments $\p$ by $(T_z\p)_{x,\,x+e_i}=\p_{x+z,\,x+z+e_i}$ for $x,y\in\Z^2$.  %Shifts of sets are denoted by   $z\pm A=\{z\pm a: a\in A\}$.  \note{Do we have any?}  
When   subscripts are  inconvenient,   $\p_{x,y}$ becomes  $\p(x,y)$, with the  analogous convention for  other   quantities such as $\pi_{x,y}$, $B_{x,y}$, and $\Bd_{x,y}$.
% and so on) as  (respectively, $\pi(x,y)$, $B(x,y)$, and $\Bd(x,y)$).
A finite or infinite sequence is denoted by $x_{i,j}=(x_i,\dotsc,x_j)$, for $-\infty\le i<j\le\infty$.   The simplex of asymptotic velocities of walks is  $\Uset=\{te_1+(1-t)e_2:0\le t\le1\}$, with relative interior  
$\ri\Uset=\{te_1+(1-t)e_2:0<t<1\}$.

%\section{Introduction and results}
\section{Results for beta RWRE}  

In Section \ref{sec:doob}  below we construct the  Doob-transformed RWRE that is the limiting process of the quenched walk conditioned on an atypical velocity $\xi\ne\llnv$.  %Then we study the Doob-transformed walk: localization of the quenched walk   \note{Maybe not?} and 
Section \ref{sec:exp} states 
the KPZ fluctuation exponent of the averaged Doob-transformed  walk.   Finally in Section \ref{sec:ldp} we display the explicit quenched large deviation rate function and its connection  with the harmonic functions of the Doob transform.  

The standing assumptions for this section are that parameters $\alpha,\beta>0$ are fixed,  and  the environment  $\p=(\p_{x,\,x+e_1})_{x\in\Z^2}$ has the i.i.d.   Beta$(\alpha,\beta)$ distribution.   The probability space  of the environment  is    $\OSP$ where $\kS$ is the Borel $\sigma$-field on the product space $\Omega=[0,1]^{\Z^2}$.

\subsection{Doob transform of the  quenched walk} \label{sec:doob}

 The  first main result is   the existence of a family of increment-stationary harmonic functions, indexed by directions in $\ri\Uset=\{te_1+(1-t)e_2:0< t<1\}$.

\begin{theorem}\label{th:Bus}
On $\OSP$ there exists  a stochastic process $\{B^\xi_{x,y}(\p):x,y\in\Z^2,\xi\in\ri\Uset\}$ with  the following properties.

For each $\xi\in\ri\Uset$,  $e^{-B^\xi_{0,x}}$ is a harmonic function: for all $x\in\Z^2$
	\begin{align}\label{B:harm}
	\p_{x,\,x+e_1}e^{-B^\xi_{0,x+e_1}(\p)}+\p_{x,\,x+e_2}e^{-B^\xi_{0,x+e_2}(\p)}=e^{-B^\xi_{0,x}(\p)} \qquad \text{$\P$-a.s.} 
	\end{align}

For each   $\xi\in\ri\Uset$ there is an event $\Omega^{(\xi)}$ such that $\P(\Omega^{(\xi)}) =1$ and for   every $\p\in\Omega^{(\xi)}$, 
	\begin{align}\label{def:Bus}
	B^\xi_{x,y}(\p)=\lim_{N\to\infty}\big(\log P_x^{\p}\{X_{\abs{z_N-x}_1}=z_N\}-\log P_y^{\p}\{X_{\abs{z_N-y}_1}=z_N\}\big)
	\end{align}
for all $x,y\in\Z^2$,
and for any sequence $z_N\in\Z^2$  such that  $\abs{z_N}_1\to\infty$ and $z_N/\abs{z_N}_1\to\xi$.  

In the law of large numbers direction  $\llnv=(\frac{\alpha}{\alpha+\beta}, \frac{\beta}{\alpha+\beta})$ we have 
\be\label{B-llnv} B^{\llnv}_{x,y}(\p)=0. \ee

\end{theorem}

%As it is customary, %we have already done in the statement of the theorem
%we will often omit the $\p$-argument from $B$.

By analogy with limits of increments in percolation and polymers,  we could  call 
$B^\xi$  the {\it  Busemann function} in direction $\xi$.   For $\xi\ne\llnv$, the variables $B^\xi_{x,\,x+e_i}$ are marginally logarithms of  beta-variables. 
From limit  \eqref{def:Bus} we see that 
\be\label{B89} B^\xi_{x,y}(T_z\p)=B^\xi_{x+z,y+z}(\p)\quad\text{and}\quad B^\xi_{x,y}(\p)+B^\xi_{y,z}(\p)=B^\xi_{x,z}(\p)\ee
 for all $x,y,z\in\Z^2$ and $\P$-a.e.\ $\p$.  In other words,  $B^\xi$ is a member of the space $\mathcal K$ of integrable  stationary cocycles defined above \eqref{K-var1}. 
	Harmonicity  \eqref{B:harm} comes  from  limit  \eqref{def:Bus}  and  the  Markov property
\[P^\p_x\{X_{\abs{z_N-x}_1}= z_N\}= \p_{x,\,x+e_1} P^\p_{x+e_1}\{X_{\abs{z_N-x}_1-1}=z_N\} + \p_{x,\,x+e_2} P^\p_{x+e_2}\{X_{\abs{z_N-x}_1-1}=z_N\}.\]
Further continuity, monotonicity, and explicit distributional properties of the process   $B^\xi$ are given in Theorem \ref{th:Buse}.

Theorem \ref{th:Bus} is proved by constructing a family of harmonic functions on quadrants and by using these to control the convergence of the differences on the right of \eqref{def:Bus}.   This approach is the RWRE counterpart of the arguments used for an exactly solvable polymer model  in  \cite{Geo-etal-15} and for the corner  growth model with general i.i.d.\ weights  in  \cite{Geo-Ras-Sep-17-ptrf-1}.
% \textcolor{blue}{(should we comment on how the theorem should hold more generally, given some differentiability of $I_q$?)}   

By \eqref{B:harm} and \eqref{B89},  
\be\label{pi-doob} \doob^\xi_{x,\,x+e_i}(\p)=\p_{x,\,x+e_i}\frac{e^{-B^\xi_{0,x+e_i}(\p)}}{e^{-B^\xi_{0,x}(\p)}} =\p_{x,\,x+e_i}e^{-B^\xi_{x,\,x+e_i}(\p)},\quad i\in\{1,2\},\ee 
defines a new transition probability on $\Z^2$, as a Doob-transform of the original transition $\p$.   It is an    RWRE   transition as a function on $\Omega$ because,  by \eqref{B89}, it obeys shifts:   $\doob^\xi_{x,\,x+e_i}(T_z\p)=\doob^\xi_{x+z,\,x+z+e_i}(\p)$. 
The environment  $\doob^\xi(\p)=(\doob^\xi_{x,\,x+e_1}(\p))_{x\in\Z^2}$ is in general correlated over  locations $x$, except that its restriction on antidiagonals is i.i.d.\ as stated in the next theorem.

  Let 
 $P^{\doob^\xi}_x$ denote the quenched path measure of the Markov chain with transition probability $\doob^\xi$. In other words,     $P^{\doob^\xi}_x$ satisfies \eqref{12} with $\doob^\xi_{y,\,y+e_i}$ instead of $\p_{y,\,y+e_i}$.   $P^{\doob^\xi}_x=P^{\doob^\xi(\p)}_x$ is a function of $\p$ through its transition probability.  
 %  We  call $P^{\doob^\xi}_x$   the {\it $B^\xi$-tilted RWRE}.
 
 \begin{theorem}\label{th:pi-xi}  Fix $\xi\in\ri\Uset$.  Then for any $n\in\Z$, the random variables $\{ \doob^\xi_{x,\,x+e_1}(\p): x_1+x_2=n\}$ are i.i.d.    We have the law of large numbers: 
	\be\label{pi-lln} P_0^{\doob^\xi(\p)}\{N^{-1}X_N\to\xi\}=1\qquad\text{for $\P$-a.e.\ $\p$.}  \ee
%\be\label{pi-9}   \text{Then for any $n\in\Z$, the random variables $\{ \doob^\xi_{x,\,x+e_1}(\p): x_1+x_2=n\}$ are i.i.d.}   \ee
\end{theorem} 	

In \cite{Geo-etal-15} an RWRE in a correlated environment arose as a limit of the quenched log-gamma polymer.   The transition probability of the log-gamma RWRE is marginally beta-distributed.   The transition $\doob^\xi$ described above is not the same.  In particular, 
 the marginal distribution of the random variable $\doob^\xi_{x,\,x+e_i}$ is {\it not} beta.  Its density function is a product of a beta density and a hypergeometric function, given in Theorem \ref{th:pi-pdf} below.

%We prove this claim  after   Theorem \ref{th:Buse}.  
 
%Convergence \eqref{def:Bus} implies the following nice result.

The next theorem records the limits of quenched processes conditioned on particular velocities.  

\begin{theorem}\label{cond}
For each fixed $\xi\in\ri\Uset$ there is an event $\Omega^{(\xi)}$ such that $\P(\Omega^{(\xi)}) =1$ and the following holds for  every $\p\in\Omega^{(\xi)}$:  if $z_N\in\Z^2$ is any sequence such that $\abs{z_N}_1=N$ and $z_N/N\to\xi$,  	
then the  conditioned quenched path distribution $P^\p_0(\cdot\,\vert X_N=z_N)$ converges weakly on the path space $(\Z^2)^{\Z_+}$ 	to the Doob  transformed path measure $P^{\doob^\xi(\p)}_0$.

%Fix $\xi,\zeta\in\ri\Uset$.   The following statements hold for   $\P$-almost every $\p$.
%\begin{enumerate}[label={\rm(\roman{*})}, ref={\rm\roman{*}}] %[\ \ {\rm(}i{\rm)}]
%
%\item\label{weakcv1}  Let $z_N\in\Z^2$ be any sequence such that $\abs{z_N}_1=N$ and $z_N/N\to\xi$.  	
%Then the  conditioned quenched path distribution $P^\p_0(\cdot\,\vert X_N=z_N)$ converges weakly on the path space $(\Z^2)^{\Z_+}$ 	to the Doob  transformed path measure $P^{\doob^\xi(\p)}_0$.  
%% and for an admissible $x_{0,n}$, $z\in\{e_1,e_2\}$, $N>n$, and $z_N$ such that $\abs{z_N}_1=N$ and $z_N/N\to\xi$, we have 
%%	\begin{align}\label{cond:lim}
%%	\begin{split}
%%	\lim_{N\to\infty}P_0^{\p}\{X_{n+1}=x_n+z\,|\,X_{0,n}=x_{0,n},\ x_N=z_N \}
%%	&=\p_{x_n,x_n+z}e^{-B^\xi_{x_n,x_n+z}}\\
%%	&=P^{\doob^\xi}_0\{X_{n+1}=x_n+z\,|\,X_n=x_n\}.
%%	\end{split}
%%	\end{align}
%% 
%\item\label{weakcv2} \textcolor{red}{Fix $\zeta\in\ri\Uset$. Let $z_N\in\Z^2$ be any sequence such that $\abs{z_N}_1=N$ and $z_N/N\to\zeta$.  	
%Then the  conditioned quenched path distribution $P^{\doob^\xi(\p)}_0(\cdot\,\vert X_N=z_N)$ converges weakly on the path space $(\Z^2)^{\Z_+}$ 	to the Doob  transformed path measure $P^{\doob^\zeta(\p)}_0$.}  
%
%\end{enumerate}
\end{theorem}

The weak convergence claim  in the theorem amounts to checking that for any finite path $x_{0,m}$ with $x_0=0$, 
\begin{align*}
\lim_{N\to\infty}  P^\p_0(X_{0,m}=x_{0,m}\,\vert X_N=z_N)
= \prod_{k=0}^{m-1} \doob^\xi_{x_k, \,x_{k+1}}(\p)
\quad\text{for $\P$-a.e.\ $\p$.} 
\end{align*}
This is an immediate consequence of limit \eqref{def:Bus}.   Combining \eqref{B-llnv} with the theorem above tells us that if $z_N/N\to\llnv$, then 
$P^\p_0(\cdot\,\vert X_N=z_N)\to P^\p_0$.  In other words, conditioning on the typical velocity $\llnv$ introduces no new correlations in the limit  and leads back to the original path measure.   This behavior is consistent with classical random walk.  

Observe that $P^{\doob^\xi(\p)}_0(X_{0,m}=x_{0,m}\,\vert X_N=z_N)=P^\p_0(X_{0,m}=x_{0,m}\,\vert X_N=z_N)$ for $0\le m\le N$. Consequently the family $\{P^{\doob^\xi}_0\}$ is closed under taking  limits of path distributions conditioned on velocities.

%The first claim comes from 
%	\begin{align*}
%	P_0^{\p}\{X_{n+1}=x_n+z\,|\,X_{0,n}=x_{0,n},\ x_N=z_N\}
%%	&=\frac{\prod_{i=0}^{n-1} \p_{x_i,x_{i+1}}}\p_{x_n,x_n+1}P_{x_n+z}^{\p,\p}\{X_{N-n-1}=\fl{N\xi}\}}{\prod_{i=0}^{n-1} \p_{x_i,x_{i+1}}}P_{x_n}^{\pi,\p}\{X_{N-n}=\fl{N\xi}\}}
%	&=\p_{x_n,x_n+z}\frac{P_{x_n+z}^{\p}\{X_{N-n-1}=z_N\}}{P_{x_n}^{\p}\{X_{N-n}=z_N\}}.
%	\end{align*}

Theorems \ref{th:Bus}, \ref{th:pi-xi} and \ref{cond}   are proved after the statement of Theorem \ref{th:Buse}. %in Appendix \ref{sec:cond}. 

% Furthermore, for $i\in\{1,2\}$, $X_n=ne_i$ means the process is trivial in that it only used steps $e_i$.   Thus, we will focus on the case $\xi\in\ri\Uset\setminus\{\llnv\}$.    \note{T: How about $z_N/N\to e_i$ ? } \note{F: Monotonicity probably gives that we get the same thing. But maybe we should just omit this comment.}

%\subsection{Localization}
%(doesn't see like we will get to this yet.) 

\subsection{Fluctuation bounds}  \label{sec:exp} 

In 1+1 dimensional models in the KPZ class, the exponent $\frac13$ appears in fluctuations of heights of growing interfaces and free energies of polymer models, while  the exponent $\frac23$ appears in spatial correlations and path fluctuations.    The Barraquand-Corwin limit \eqref{b-c} indicated that  logarithms of quenched probabilities obey   $\frac13$-fluctuations.  The theorem below shows that the process $B^\xi$ has this same order of magnitude of fluctuations,  though only in the direction $\xi$,  as quantified by hypothesis \eqref{char-ass} below.   If the endpoint $(m,n)$ deviates from $N\xi$ by an amount of order  $N^\nu$ for $\nu>\tfrac23$, the fluctuations of $B^\xi_{0, (m,n)}$ become Gaussian.  (This follows similar observations for directed polymers in Corollary 1.4 of \cite{Cha-Noa-17-} and Corollary 2.2 of \cite{Sep-12-corr}.)
 
\begin{theorem}\label{th:B-var}  
Fix $\alpha,\beta>0$. Fix $\xi=(\xi_1, \xi_2) \in\ri\Uset\setminus\{\llnv\}$.  Given a constant  $0<\gamma<\infty$, there exist  positive finite constants $c$, $C$, and $N_0$, depending only on $\alpha$, $\beta$, $\gamma$, and $\xi$, 
such that 
\[   c   N^{2/3} \le \Vvv[B^\xi_{0, (m,n)}   ]\le C\, N^{2/3} \]
for all  $N\ge N_0$ and $(m,n)\in\N^2$ such that 
\be\label{char-ass} \abs{m-N\xi_1} \vee\abs{n-N\xi_2} \le \gamma N^{2/3}. \ee
The same constants can be taken for $(\alpha,\beta,\gamma,\xi)$ varying in a compact subset of $(0,\infty)^3\times\ri\Uset\setminus\{\llnv\}$.
\end{theorem}

Theorem \ref{th:B-var} was proved independently and concurrently in the present work and as  one case of a more general result for exactly solvable directed polymers by Chaumont and Noack (Theorem 1.2 of \cite{Cha-Noa-17-}).    A proof appears  in Section 4.1 of the first preprint version  \cite{Bal-Ras-Sep-18-arxiv} of this paper. In the present version we omit the proof and cite   \cite{Cha-Noa-17-} for   details.  The translation between the Doob-transformed walk and the beta polymer is explained in Section \ref{Beta-Polys}. 

%\textcolor{blue}{Theorem \ref{th:B-var} is proved in Section \ref{sec:var}. It follows essentially from Theorem 1.2 of \cite{Cha-Noa-17-}.}
%Theorems \ref{stat-upper} and \ref{stat-lower} stated in the beginning of Section \ref{sec:var}.  

The second fluctuation result quantifies the deviations of the walk from its limiting velocity,  under the averaged measure $\bfP^\xi(\cdot)=\int P^{\doob^\xi(\p)}(\cdot)\,\P(d\p)$ of the Doob-transformed RWRE.    Bounds \eqref{KPZrwre1} and  \eqref{KPZrwre3} indicate that  this walk is superdiffusive with  the KPZ  wandering exponent  $\frac23$ instead of the diffusive  $\frac12$ of classical random walk.
%Conditioning the RWRE $P^\p$ to have asymptotic velocity $\xi\ne\llnv$ turns it into another (Doob-transformed) RWRE whose asymptotic velocity is indeed $\xi$ and whose transitions $\doob^\xi$ are strongly correlated.   For this RWRE the averaged measure $\bfP^\xi(\cdot)=\int P^{\doob^\xi(\p)}(\cdot)\,\P(d\p)$ is no longer a classical random walk.   Our next  result is that under the averaged measure  the  conditioned walk is superdiffusive with  wandering exponent of $2/3$ instead of the diffusive  $1/2$.

\begin{theorem}\label{th:kpz}
Fix $\alpha,\beta>0$. % and let $\p_{0,e_1}$ have {\rm Beta}$\,(\alpha,\beta)$ distribution.  
Fix $\xi\in\ri\Uset\setminus\{\llnv\}$. There exist finite positive constants $C$, $c$, $r_0$, and $\delta_0$, depending only on $\alpha$, $\beta$, and $\xi$, such that for $r\ge r_0$, $\delta\in(0,\delta_0)$, and any $N\ge1$ we have
	\begin{align}\label{KPZrwre1}
	\bfP^{\xi}_0\{  \abs{X_N- N\xi}_1 \ge rN^{2/3} \}\le Cr^{-3}
	\end{align}
	and 
	\begin{align}\label{KPZrwre3}
	\bfP^{\xi}_0\{ \abs{X_N-N\xi}_1\ge \delta N^{2/3} \}\ge c.
	\end{align}
The same constants can be taken for $(\alpha,\beta,\xi)$ varying in a compact subset of $(0,\infty)^2\times\ri\Uset\setminus\{\llnv\}$.
\end{theorem}

Theorem \ref{th:kpz}  is proved in Section \ref{pf:main}. The bounds come from using harmonic functions to control the exit point of the walk from rectangles.  For this proof also  we can cite an estimate from  \cite{Cha-Noa-17-}.

%This result suggests that the RWRE, conditioned to escape at an atypical velocity, indeed belongs to the KPZ universality class.
 
%\textcolor{blue}{(NOTE: there is a strong similarity with percolation with a flat edge!  Gaussian inside the cone and KPZ outside it! though none of this is proved either... maybe a future project?!!!)}
%\textcolor{blue}{(Should we also mention the RAP, which also paces the [quenched mean of] RWRE in the EW class?)}
%\textcolor{red}{(Guillaume and Ivan said something about a vanishing second-derivative-like quantity. Check this out!)}

%On the other hand, we prove that, under the quenched measure, the process is localized  around its quenched mean $E^\p[X_n]$, while 
%the quenched mean itself has fluctuations of order $n^{??}$. \textcolor{red}{(hopefully we can show all this and get the order of the fluctuations of the quenched mean! it should be 2/3, no?!)}
%The localization of the RWRE comes then from .... \textcolor{red}{(fill out once done)}

\subsection{Large deviations}  \label{sec:ldp} 
 This section records explicit  large deviation rate functions and their link with the process $B^\xi$ of Theorem \ref{th:Bus}.   
 
We begin with a technical point that is needed for the rate function and for the entire remainder of the paper.  
 The next lemma establishes connections between three parameters:  $\xi\in\Uset$ is an asymptotic velocity of the walk, $t\in\R$ is the tilt  dual to $\xi$, and $0<\lambda<\infty$  parametrizes two families of increment-stationary harmonic functions that we construct in Section \ref{sec:bdry} and use as  tools to analyze the model.  A slight inconvenience is that as $\xi$ ranges across $\Uset$ from left to right (in the direction of $\xi_1$), $\lambda$ goes from $0$ to $\infty$ and back, with $\lambda=\infty$ corresponding to $\xi=\llnv$.  This is depicted in the  first two plots of Figure \ref{fig:lam}. 

Recall the polygamma functions $\psi_0(s)=\Gamma'(s)/\Gamma(s)$ and  $\psi_n(s)=\psi_{n-1}'(s)$ for $s>0$ and $n\in\N$.  
Some basic properties of these functions are given in Appendix \ref{psi-prop}.   Qualitatively speaking,  $\psi_0$ is strictly concave and  increasing from $\psi_0(0+)=-\infty$ to $\psi_0(\infty-)=\infty$, while $\psi_1$ is strictly convex  and decreasing from $\psi_1(0+)=\infty$ to $\psi_1(\infty-)=0$. 
 
 \begin{figure}[h]
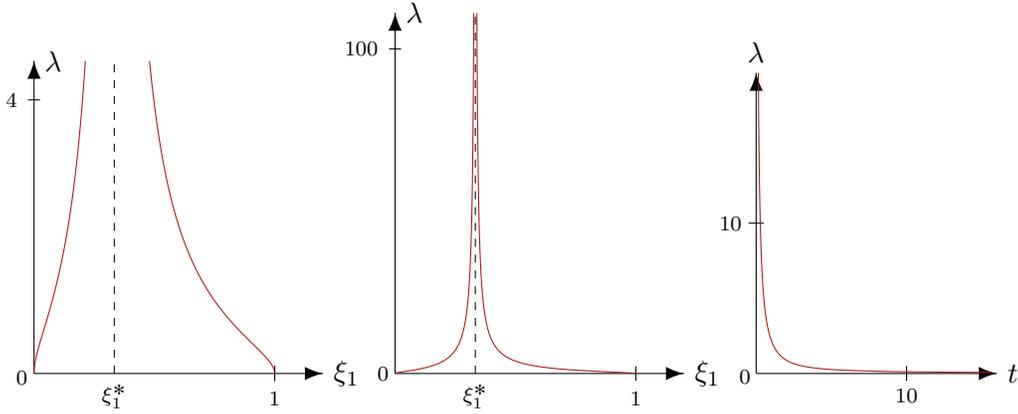

 	\begin{center}
		%alpha=1, beta=2
 		 \begin{tikzpicture}[scale=0.8,>=latex,decoration={markings, mark=at position 1 with {\arrow[scale=1.7,black]{latex}}}]
		 \begin{scope}[xscale=4, yscale=1.3]
		      	\draw[postaction={decorate}] (0,0) -- (1.2,0) node[right] {\small$\xi_1$};
			\draw(-0.05,-0.05) node{\tiny$0$};
      			\draw[postaction={decorate}] (0,0) -- (0,4) node[right] {\small$\lambda$}; 
			%\draw (0.58,2.95) node {\small$\lambda$};
			%\draw (0.12,3) node {\small$\bar\lambda$};
			\draw[dashed] (1/3,0)--(1/3,4);
      			\input{lamxi-max4}
			\draw (-0.025,3.5) -- (0.025,3.5);
			\draw (-0.025,3.5) node[left] {\tiny$4$};
			\draw (1,-0.1) -- (1,0.1);
			\draw (1,-0.1) node[below] {\tiny$1$};
			\draw (1/3,0) node[below] {\tiny$\llnv_1$};
		\end{scope}
 		 \begin{scope}[shift={(6,0)}, xscale=4, yscale=0.06]
		      	\draw[postaction={decorate}] (0,0) -- (1.2,0) node[right] {\small$\xi_1$};
      			\draw[postaction={decorate}] (0,0) -- (0,100) node[right] {\small$\lambda$};
			\draw(-0.05,-0.05) node{\tiny$0$};
      			%\draw (0.45,15) node {\small$\lambda$};
			%\draw (0.22,15.5) node {\small$\bar\lambda$};
			\input{lamxi-max100}
			\draw[dashed] (1/3,0)--(1/3,100);
			\draw (-0.025,90) -- (0.025,90);
			\draw (-0.025,90) node[left] {\tiny$100$};
			\draw (1,-2) -- (1,2);
			\draw (1,-2) node[below] {\tiny$1$};
			\draw (1/3,0) node[below] {\tiny$\llnv_1$};
		\end{scope}
 		 \begin{scope}[shift={(12,0)}, scale=0.25]
		      	\draw[postaction={decorate}] (0,0) -- (16,0) node[right] {\small$t$};
      			\draw(-0.75,-0.2) node{\tiny$0$};
      			\draw[postaction={decorate}] (0,0) -- (0,20) node[above]{\small$\lambda$};
			\draw(-0.5,10)--(0.5,10);
			\draw(-1.7,10) node{\tiny$10$};
			\draw(10,-0.5)--(10,0.5);
			\draw(10,-1.5) node{\tiny$10$};
			\input{lamt}
		\end{scope}
		\end{tikzpicture}
	\end{center}
 	\caption{\small Leftmost and middle plots are of $\lambda$ as a function of $\xi_1$. 
	The left plot stretches the $\lambda$-axis to reveal the behavior away from $\llnv_1$. The rightmost plot is of $\lambda$ as a function of $t$. These graphs are for  $(\alpha,\beta)=(1,2)$.}
 \label{fig:lam}
 \end{figure}

\begin{lemma}\label{lam-xi-t}
Fix $\alpha,\beta>0$. 
\begin{enumerate}[label={\rm(\alph*)}, ref={\rm\alph*}] %[\ \ \rm(a)]
\item\label{lam-xi} Given $\xi=(\xi_1, 1-\xi_1)\in\Uset$ there is a unique $\lambda=\lambda(\xi)\in[0,\infty]$ such that 
\begin{align}
&\xi_1=\frac{\psi_1(\lambda)-\psi_1(\alpha+\lambda)}{\psi_1(\lambda)-\psi_1(\alpha+\beta+\lambda)}\quad\text{for }\xi_1\in[\llnv_1, 1],\quad\text{and}\label{eq:lam-xi1}\\
&\xi_1=1-\frac{\psi_1(\lambda)-\psi_1(\beta+\lambda)}{\psi_1(\lambda)-\psi_1(\alpha+\beta+\lambda)}\quad\text{for }\xi_1\in[0,\llnv_1],\label{eq:lam-xi2}\\
&\text{with }\lambda=0\Longleftrightarrow \xi\in\{e_1,e_2\}\text{ and }\lambda=\infty\Longleftrightarrow \xi=\llnv=(\tfrac{\alpha}{\alpha+\beta}, \tfrac{\beta}{\alpha+\beta})\,. \notag  
\end{align}
Furthermore, $\lambda$ is strictly  increasing on $\xi_1\in[0,\llnv_1)$ and strictly decreasing on $\xi_1\in(\llnv_1,1]$.\\[-5pt] 
\item\label{lam-t} Given $t\in[0,\infty]$ there is a unique $\lambda=\lambda(t)\in[0,\infty]$ such that
 	\begin{align}\label{eq:lam-t}
	t=  \psi_0(\alpha+\beta+\lambda)-\psi_0(\lambda)
	\end{align}
where $\lambda=0\Longleftrightarrow t=\infty$ and $\lambda=\infty\Longleftrightarrow t=0$.   
\end{enumerate}
\end{lemma}

The proof of Lemma \ref{lam-xi-t} is given in Section \ref{Iq-pf}.
The formula for the quenched rate function  $I_q$ in \eqref{qLDP}  in  the beta environment  can now be given. See Figure \ref{fig:I} for an illustration.

\begin{theorem}\label{thm:Iq}
Fix $\alpha,\beta>0$ and let $\p$ have i.i.d.\  {\rm Beta}$\,(\alpha,\beta)$ distribution.  
Then  for $\xi=(\xi_1,\xi_2)\in\Uset$ we have $I_q(\llnv)=0$ and
\begin{align}\label{Iq} 
I_q(\xi)=
\begin{cases}
\xi_1\psi_0\bigl(\alpha+\beta+\lambda(\xi)\bigr)+\xi_2\psi_0\bigl(\lambda(\xi)\bigr)-\psi_0\bigl(\alpha+\lambda(\xi)\bigr) &\text{for }\xi_1\in(\llnv_1,1],\\[3pt] 
\xi_2\psi_0\bigl(\alpha+\beta+\lambda(\xi)\bigr)+\xi_1\psi_0\bigl(\lambda(\xi)\bigr)  -\psi_0\bigl(\beta+\lambda(\xi)\bigr) &\text{for }\xi_1\in[0,\llnv_1),
\end{cases}
\end{align}
where in both cases $\lambda$ and $\xi$ determine each other uniquely via \eqref{eq:lam-xi1} and \eqref{eq:lam-xi2}.     $I_q$ is a strictly convex function on $[0,1]$ and satisfies $I_q(\xi)>I_a(\xi)$ for all $\xi\in\Uset\setminus\{\llnv\}$.
\end{theorem}

\begin{example} [Case $\alpha=\beta=1$]
In the i.i.d.\ uniform environment  $\lambda$ and $I_q$ can be found in closed form with the help of the recurrence formulas \eqref{psi-rec}.  The rate function is 
\be\label{Iq-unif}  I_q(\xi)=1-2\sqrt{\xi_1\xi_2}  
=\sum_{n=1}^\infty \binom{\tfrac12}{n}(-1)^{n+1}4^n (\xi_1-\tfrac12)^{2n}\qquad \text{for } \; \xi\in\Uset.  \ee
The series   illustrates that this rate function is analytic on the entire open segment $\ri\Uset$, a property which is open for   general $(\alpha, \beta)$. 
\hfill$\triangle$\end{example}

\begin{remark}[Regularity of $I_q$]\label{rk:I-reg}
Lemma \ref{lm700} in Appendix \ref{app:exp} shows that $I_q$ is analytic away from $\llnv$.  We compute derivatives of $I_q$ up to the fourth one, to verify that across $\llnv$ we have at least four continuous derivatives.   We obtain the following expansion   around $\llnv$: 
	%\[I_q(\xi)=\frac{(\alpha+\beta)^2}{2\alpha\beta}(\xi_1-\llnv_1)^2+o(\abs{\xi-\llnv}_1^2).\] 
\be\label{Iq-exp}	\begin{aligned}
	I_q(\xi)&=\frac{(\alpha+\beta)^2}{2\alpha\beta}(\xi_1-\llnv_1)^2+\frac{(\alpha+\beta)^3(\alpha-\beta)}{6\alpha^2\beta^2}(\xi_1-\llnv_1)^3\\
	&\qquad+\frac{(\alpha+\beta)^4(2\alpha^2-2\alpha\beta+2\beta^2+1)}{24\alpha^3\beta^3}(\xi_1-\llnv_1)^4+{\scriptstyle\cO}(\abs{\xi_1-\llnv_1}^4).
	\end{aligned}\ee
The details appear at the end of Appendix \ref{app:exp}.

For the sake of comparison,  here is the   expansion    around    $\llnv$ of the averaged rate function  $I_a$ from \eqref{Ia}:
%	\[I_a(\xi)=\frac{(\alpha+\beta)^2}{2\alpha\beta}(\xi_1-\llnv_1)^2+o(\abs{\xi-\llnv}_1^2).\] 
	\begin{align}\label{Ia-exp}
	\begin{split}
	I_a(\xi)
	&=\frac{(\alpha+\beta)^2}{2\alpha\beta}(\xi_1-\llnv_1)^2+\frac{(\alpha+\beta)^3(\alpha-\beta)}{6\alpha^2\beta^2}(\xi_1-\llnv_1)^3\\
	&\qquad+\frac{(\alpha+\beta)^3(\alpha^3+\beta^3)}{12\alpha^3\beta^3}(\xi_1-\llnv_1)^4+{\mathcal O}(\abs{\xi_1-\llnv_1}^5).
	\end{split}
	\end{align}
%From this and the fact that both rate functions vanish at $\llnv$ it follows that $I_q$ is strictly concave at $\llnv$.
The  expansions  of $I_q$ and $I_a$  agree to third order.   This explains the minute difference between the two graphs in Figure \ref{fig:I}.
One can check  that 
%the difference of the coefficients of the quartic terms equals
	\begin{align*}
	&\frac{d^4}{d\xi_1^4}\Bigl[I_q(\xi_1,1-\xi_1)-I_a(\xi_1,1-\xi_1)\Bigr]_{\xi=\llnv}=\frac{(\alpha+\beta)^4}{\alpha^3\beta^3}>0.
	\end{align*}
%and is positive for all $\alpha,\beta>0$.
%The discriminent equals
%	\[-4(45000\beta^2+2700\beta+32019)<0\]
Thus  the fourth-order terms differ in the two expansions.
\hfill$\triangle$\end{remark} 

\begin{figure}[h]
 	\begin{center}
		%alpha=1, beta=2
 		 \begin{tikzpicture}[scale=1, >=latex,decoration={markings, mark=at position 1 with {\arrow[scale=1.7,black]{latex}}}]
		 \begin{scope}[xscale=5, yscale=4]
		      	\draw[postaction={decorate}] (0,0) -- (1.2,0) node[right] {\small$\xi_1$};
      			\draw[postaction={decorate}] (0,0) -- (0,1.7);
			\draw (0.94,1.3) node{$I_q$}; 
			\draw (0.94,0.67) node{$I_a$}; 
			\input{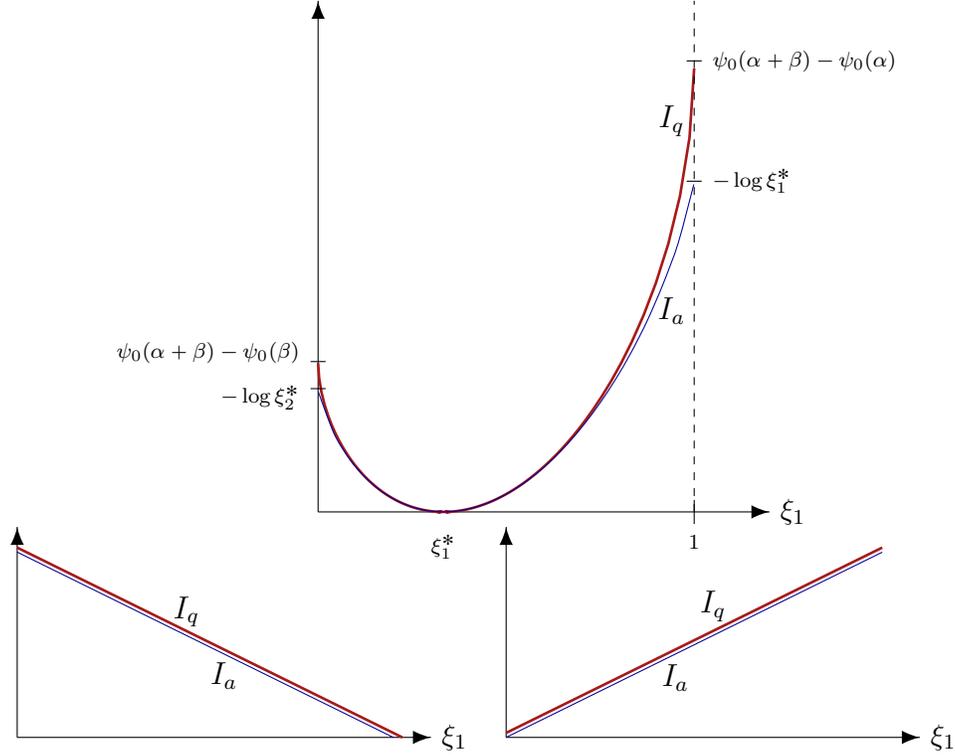}
			\draw (1,-0.04) -- (1,0.04);
			\draw (1,-0.04) node[below] {\tiny$1$};
			\draw (1/3,-0.04) node[below] {\tiny$\llnv_1$};
			\draw[domain=0.001:0.999,smooth,variable=\x,my-blue] plot ({\x},{\x*ln(3*\x)+(1-\x)*ln(3*(1-\x)/2)});
			\draw (-0.02,0.5) -- (0.02,0.5);
			\draw (-0.02,0.53) node[left] {\tiny$\psi_0(\alpha+\beta)-\psi_0(\beta)$};
			\draw (-0.02,0.41) -- (0.02,0.41);
			\draw (-0.02,0.38) node[left] {\tiny$-\log\llnv_2$};
			\draw[dashed] (1,0)--(1,1.7);
			\draw (1-.02,1.5) -- (1.02,1.5);
			\draw (1.02,1.5) node[right] {\tiny$\psi_0(\alpha+\beta)-\psi_0(\alpha)$};
			\draw (1-.02,1.1) -- (1.02,1.1);
			\draw (1.02,1.1) node[right] {\tiny$-\log\llnv_1$};
		 \end{scope}
 		 \begin{scope}[xscale=5, yscale=2.5, shift={(-0.8,-1.2)}]
		      	\draw[postaction={decorate}] (0,0) -- (1.1,0) node[right] {\small$\xi_1$};
      			\draw[postaction={decorate}] (0,0) -- (0,1.12);
		 	\draw [my-blue] (0,0.986013107782760)--(1,0);
			\draw [nicosred, line width=1pt] (0,1.01)--(1.013956843195735*1.01,0);
			\draw (0.45,0.67) node {$I_q$};
			\draw (0.55,0.33) node {$I_a$};
		   \end{scope}
 		 \begin{scope}[xscale=5, yscale=2.5, shift={(0.5,-1.2)}]
		      	\draw[postaction={decorate}] (0,0) -- (1.1,0) node[right] {\small$\xi_1$};
      			\draw[postaction={decorate}] (0,0) -- (0,1.12);
		 	\draw [my-blue] (0,0)--(1,0.985587899318051);
			\draw [nicosred, line width=1pt] (0,0.013910173484537+0.01)--(1,1.01);
			\draw (0.55,0.7) node {$I_q$};
			\draw (0.45,0.32) node {$I_a$};
		   \end{scope}
		\end{tikzpicture}
	\end{center}
 	\caption{\small Top plot shows $I_q$ (higher, thicker graph) and $I_a$ (lower, thinner graph) as functions of $\xi_1$. Bottom figures   zoom in on $I_q$ and $I_a$ over intervals 
	$[\llnv_1-4.440004\times10^{-4},\llnv_1-4.439998\times10^{-4}]$ and $[\llnv_1+4.442957038524\times 10^{-4},\llnv_1+4.442957038583\times10^{-4}]$. 
	%$[\llnv_1-4.440003948256721\times10^{-4},\llnv-4.439998030628112\times10^{-4}]$ and $[\llnv_1+4.442957038524042\times 10^{-4},\llnv_1+4.442957038583439\times10^{-4}]$. 
	On the left, the vertical axis goes from $4.436541\times10^{-7}$ to $4.436554\times10^{-7}$.
	On the right, the vertical axis goes from $4.440484\times10^{-7}$ to $4.440497\times10^{-7}$.
	%On the left, the vertical axis goes from $4.43654174206540\times10^{-7}$ to $4.43655373771890\times10^{-7}$.
	%On the right, the vertical axis goes from $4.44048445286842\times10^{-7}$ to $4.44049646830535\times10^{-7}$.
These graphs are for  $(\alpha,\beta)=(1,2)$.} 
%	Here, $\alpha=1$ and $\beta=2$ and thus $\llnv_1=1/3$. At $\xi_1=0$ and $\xi_1=1$, $I_q$ equals, respectively, $\psi_0(\alpha+\beta)-\psi_0(\alpha)=1.5$ and 
%	$\psi_0(\alpha+\beta)-\psi_0(\beta)=0.5$.  At these points, $I_a$ equals, respectively, $-\log\llnv_1 \approx1.1$ and $-\log\llnv_2\approx0.41$.}
 \label{fig:I}
 \end{figure}

We also record  the convex conjugate 
	\[I_q^*(h)
	=\sup_{\xi\in\Uset} \{ h\cdot\xi-I_q(\xi)\}
	=\lim_{n\to\infty} n^{-1} \log E^\p_0[e^{h\cdot X_n}] ,\quad h\in\R^2 .\]
The second equality above  is an instance of Varadhan's theorem \cite[page 28]{Ras-Sep-15-ldp}.   
Since  $(X_n-X_0) \cdot(e_1+e_2)=n$,  we have $I_q^*(te_1+se_2)=s+I_q^*((t-s)e_1)$  and it suffices to consider $h=te_1$ for real $t$. 

\begin{theorem}\label{thm:Iq2}
Fix $\alpha,\beta>0$ and let $\p$ have i.i.d.\  {\rm Beta}$\,(\alpha,\beta)$ distribution.  
For $t\ge0$
\begin{align}
I_q^*(te_1) & =  \psi_0(\alpha+\lambda(t))-\psi_0(\lambda(t)) \label{Iq*1}  \\
\text{and}\qquad I_q^*(-te_1)&=-t+\psi_0(\beta+\lambda(t))-\psi_0(\lambda(t)),\label{Iq*2} 
\end{align}
where $\lambda$ and $t$ determine each other via \eqref{eq:lam-t}. 
%More generally, for $t,s\in\R$
%	\begin{align}\label{Iq*3}
%	I_q^*(te_1+se_2)=s+I_q^*((t-s)e_1).
%	\end{align}
\end{theorem}

Formula \eqref{Iq} for $I_q$ appeared earlier   in equations~(8)--(9) of \cite{Bar-Cor-17} where it was derived by nontrivial  asymptotic analysis.   We derive $I_q$ through its  convex conjugate $I_q^*$, which in turn is calculated with the help of  harmonic functions to be constructed below.

%\textcolor{blue}{In a remark (above (8)) the authors of \cite{Bar-Cor-17} wondered about the connection between  \eqref{Iq} and solutions to our variational formulas for $I_q$.}  Note: they were asking about the entropy variational formula. 

Next we state the connections between $I_q$ and the processes $B^\xi$.  

\begin{theorem}\label{thm:Iq3} 
{\rm (a)} Fix $\xi\in\ri\Uset$.  Then the process $B^\xi$ is an extremal for variational formula \eqref{K-var1}.  In particular, we have 
	\be\label{Iq17} \begin{aligned}  
	I_q(\xi)&= -\E[B^\xi_{0,e_1}]\xi_1-\E[B^\xi_{0,e_2}]\xi_2\\
	&=-\inf_{\zeta\,\in\,\ri\Uset}\big\{\E[B^\zeta_{0,e_1}]\xi_1+\E[B^\zeta_{0,e_2}]\xi_2\big\},
\end{aligned} 	\ee
where the last infimum  is uniquely attained at $\zeta=\xi$.   

{\rm (b)}  Extend $I_q$  homogeneously to all of $\R_+^2$, that is, by $I_q(c\xi)=cI_q(\xi)$ for $c>0$ and $\xi\in\Uset$.   Then the gradient of $I_q$ satisfies 
 \be\label{Iq18}   \nabla I_q(\xi)=-\E[B^\xi_{0,e_1}]e_1-\E[B^\xi_{0,e_2}]e_2,\quad\xi\in\ri\Uset.\ee 

\end{theorem} 

Corollary 4.5 and Remark 5.7 in \cite{Geo-Ras-Sep-16-cmp}  put  equations \eqref{Iq17}--\eqref{Iq18} in the context of a general theory for directed walks in random potentials.   Theorems \ref{thm:Iq}, \ref{thm:Iq2} and \ref{thm:Iq3} are proved in Section \ref{Iq-pf}. 

% Extend $I_q$  homogeneously to all of $\R_+^2$, that is, by $I_q(c\xi)=cI_q(\xi)$ for $c>0$ and $\xi\in\Uset$.  Formula \eqref{Iq17} remains valid.  This, together with some calculus   and the continuity of $\E[B^\zeta(0,e_i)]$, 
%gives  a formula for the gradient of $I_q$: 
% \be\label{Iq18}   \nabla I_q(\xi)=-\E[B^\xi_{0,e_1}]e_1-\E[B^\xi_{0,e_1}]e_2,\quad\xi\in\ri\Uset.\ee 
%let a(\zeta)=\E[B^\zeta_{0,e_1}] and b(\zeta)=\E[B^\zeta_{0,e_2}]
%from the var formula we can bound -I_q(\xi+\e e_1) above by a(\xi)(\xi_1+\e)+b(\xi)\xi_2. so we get the upper bound -(I_q(\xi+\e e_1)-I_q(\xi))/\e \le a(\xi)
%similarly, we can bound -I_q(\xi) below by a(\xi+\e e_1)\xi_1+b(\xi+\e e_1)\xi_2 and get the lower bound -(I_q(\xi+\e e_1)-I_q(\xi))/\e \ge a(\xi+\e e_1)
%the claim now comes from the continuity of a.

%\textcolor{blue}{Convexity of $I_q$ implies the monotonicity of $e_i\cdot\nabla I_q(\xi)$, $i\in\{1,2\}$. 
%This and the monotonicity of $\E[B^\xi_{0,e_i}]$ (Theorem \ref{th:Buse}\eqref{B:mono}) imply that this formula continues to hold all the way to the limits $\xi\in\{e_1,e_2\}$.
%In particular, $e_i\cdot\nabla I_q(\xi)$ blows up at $\xi=e_{3-i}$, $i\in\{1,2\}$.} \textcolor{red}{(warning: we do not define $B^{e_i}$, but we can check from the formulas that the expected values
%blow up. so maybe this paragraph needs a bit of rewarding...)}

 Lastly, we record the  LDP for the Doob-transformed RWRE.  
 Definition \eqref{pi-doob} and the cocycle property in \eqref{B89} imply that 
	\[P_0^{\doob^\xi(\p)}(X_N=x)=P_0^\p(X_N=x)e^{-B^\xi_{0,x}(\p)}.\]
$B^\xi$ has i.i.d.\ increments along horizontal and vertical lines (Theorem \ref{th:Buse}\eqref{B:B=Bd}) and hence the   law of large numbers applies: $\P$-almost surely
	\[\lim_{N\to\infty}N^{-1} B^\xi_{0,[N\zeta]}=\E[B^\xi_{0,e_1}]\zeta_1+\E[B^\xi_{0,e_2}]\zeta_2=-\zeta\cdot\nabla I_q(\xi) \qquad \forall  \zeta\in\ri\Uset.\]
%(See \eqref{whatever} and Theorem \ref{th:Buse}\eqref{B:B=Bd}.)
The quenched LDP  \eqref{qLDP} of the beta walk then gives this theorem.
 \begin{theorem}
For any fixed $\xi\in\ri\Uset$,  the following holds $\P$-almost surely, simultaneously for all $\zeta\in\ri\Uset$, 
	%\begin{align}\label{qLDP-kappa}
	\[\lim_{N\to\infty}N^{-1} \log P_0^{\doob^\xi}\{ X_N = [N\zeta] \} =-I_q(\zeta)+\zeta\cdot\nabla I_q(\xi).\]
	%\end{align}
\end{theorem} 
 The rate function $I_q(\zeta)-\zeta\cdot\nabla I_q(\xi)$ is minimized at $\zeta=\xi$, by the convexity and homogeneity of $I_q$.    The main results have been stated and we turn to proofs.

%\subsection{Beta polymer}  Here a brief explanation of the polymer connection and  results. 

%\section{Stationary Beta polymers}
\section{Increment-stationary harmonic functions}

In this section we construct quenched harmonic functions  whose probability distributions  are suitably invariant under lattice translations.     This is done first on restricted subsets of the lattice by solving a boundary value problem, and then extended to the entire lattice by taking limits.    That this is possible with explicit distributions and useful independence properties  is a  feature of exact solvability. 

%Define hitting times
% \be\label{tau12}  \tau^{\ell}_v=\inf\{  n\ge 0:  X_n\in  v+\Z e_\ell\}.  \ee
 The boundaries   of the positive and negative quadrants $v+\Z_+^2$ and $v-\Z_+^2$ with a corner at   $v\in\Z^2$  are  denoted by 
 \be\label{Bpm}  \B^{+}_v=\{ v+(i,0), v+(0,j): i,j\ge 0\}
 \quad\text{and}\quad 
 \B^{-}_v=\{ v-(i,0), v-(0,j): i,j\ge 0\}.  \ee
% (Whether the corner $v$ itself is included in the boundary is immaterial because a walk that starts in the bulk $v+\N^2$ cannot hit $v$ before hitting $\B^{+}_v$, and similarly for a walk 
 Hitting times of the boundaries follow analogous  notation: 
 \be\label{taupm}  \tau^\pm_v=\inf\{  n\ge 0:  X_n\in  \B^{\pm}_v\}.  \ee
  The separate axes of these boundaries are distinguished by the notation
\be\label{B12}   \B^{(\pm1)}_v=\{ v\pm(i,0) : i\ge 0\}
\quad\text{and}\quad 
 \B^{(\pm2)}_v=\{ v\pm(0,j) : j\ge 0\}. \ee
 In particular,     $\B^{\pm}_v = \B^{(\pm1)}_v\cup  \B^{(\pm2)}_v$.  

%\subsection{Beta harmonic functions}
%\subsection{Harmonic functions on quadrants} \label{sec:bdry} 
\subsection{An involution for beta variables} \label{sec:invo}  

%boundary system}\label{sec:bdry}

%The  development begins with finding harmonic functions 

This section undertakes some technical preparation for the construction of harmonic functions on quadrants of the lattice.   A distribution-preserving  involution of triples of beta variables is defined and its properties recorded.  We begin by motivating this construction through   a Dirichlet problem. 

Consider backward nearest-neighbor transition probabilities $\pch_{x,\,x-e_i}$, $i\in\{1,2\}$, on the  lattice $\Z^2$.   These transition probabilities allow   two steps  $-e_1$ and $-e_2$ and satisfy $\pch_{x,\,x-e_1}+\pch_{x,\,x-e_2}=1$ at each $x\in\Z^2$.   Suppose   a function $f$ is given on the boundary $\B^+_0$ of the first quadrant $\Z_+^2$.    When the backward walk starts   in the first quadrant,  the hitting time  $\tau^+_0$
is obviously  finite. Then 
\be\label{u88}  H(x)=E^{\pch}_x[f(X(\tau^+_0))] \ee
 defines  an $\pch$-harmonic function on the positive  first quadrant.  That is, $H$ satisfies 
	\begin{align}\label{har7}
	H(x) = \pch_{x,\,x-e_1} H(x-e_1) +  \pch_{x,\,x-e_2} H(x-e_2)\quad\text{for $x\in\N^2$}.
	\end{align}   

We solve   \eqref{har7}   inductively, by beginning from the boundary values and then defining $H(x)$ once $H(x-e_1)$ and $H(x-e_2)$ have been defined.  We formulate this induction in terms of   ratios   $\Bd_{x,y}= {H(x)}/{H(y)}$.   The induction assumption is   that the nearest-neighbor ratios $\Bd_{x-e_2,\,x-e_1-e_2}$ and $\Bd_{x-e_1,\,x-e_1-e_2}$   have been defined on the south and west sides of a unit square with northeast corner at $x$. Then, by \eqref{har7},   the ratios on the north and east sides are obtained from the equations 
%\be\label{ind700} \begin{aligned}
\begin{align} \label{ind700}
\Bd_{x,\,x-e_1}&= \frac{ \pch_{x,\,x-e_1} \Bd_{x-e_1,\,x-e_1-e_2} +  (1-\pch_{x,\,x-e_1}) \Bd_{x-e_2,\,x-e_1-e_2}}{\Bd_{x-e_1,\,x-e_1-e_2}} , \\
\label{ind701} 
\Bd_{x,\,x-e_2}&= \frac{ \pch_{x,\,x-e_1} \Bd_{x-e_1,\,x-e_1-e_2} +  (1-\pch_{x,\,x-e_1})  \Bd_{x-e_2,\,x-e_1-e_2}}{\Bd_{x-e_2,\,x-e_1-e_2}}. 
\end{align} %\end{aligned}\ee
It is useful to augment this pair of equations with a third equation  
\begin{align} \label{ind702}
 \p_{x-e_1-e_2,\,x-e_2}=\frac{\Bd_{x-e_2,\,x-e_1-e_2}( \Bd_{x-e_1,\,x-e_1-e_2}-1)}{ \Bd_{x-e_1,\,x-e_1-e_2}-\Bd_{x-e_2,\,x-e_1-e_2}}
\end{align}
provided the denominator never vanishes.
 Together the three equations define an involution.  In the case we specialize to below $ \p_{x-e_1-e_2,\,x-e_2}$ is a forward   transition probability from $x-e_1-e_2$ to   $x-e_2$. The complementary  transition probability from $x-e_1-e_2$ to   $x-e_1$ is of course then 
\begin{align} \label{ind703}
\p_{x-e_1-e_2,\,x-e_1}=1-\p_{x-e_1-e_2,\,x-e_2}. 
\end{align}

Equations \eqref{ind700}--\eqref{ind703} are illustrated by Figure \ref{fig:ind}, with $x$ in the upper right corner of the unit square and with 
\begin{align*}   &(U,V,\Beta)=(\Bd_{x-e_2,\,x-e_1-e_2}, \, \Bd_{x-e_1,\,x-e_1-e_2},  \,\pch_{x,\,x-e_1}) \\[2pt]  
& \quad\text{and}\quad 
 (U',V',\Beta')=(\Bd_{x,\,x-e_1}, \,\Bd_{x,\,x-e_2}, \,\p_{x-e_1-e_2,\,x-e_2}). \end{align*} 

%  The stationary version of  each of the four known exactly solvable polymer  models is based on  an involution that preserves the distribution of weights.  This  is the version for the beta model. 

 \begin{figure}[h]
 	\begin{center}
 		 \begin{tikzpicture}[>=latex, scale=0.5]
		 	\draw(0,0)--(0,6)--(6,6)--(6,0)--(0,0);
			\draw [line width= 5pt, color=nicosred](0.3,0)--(5.7,0);
			\draw(3,-0.1)node[below]{$U$};
			\draw [line width= 5pt, color=nicosred](0,0.3)--(0,5.7);
			\draw(-0.1,3)node[left]{$V$};
			\draw (5,6.1)node[above]{$\Beta$};
			\draw (6.1,5)node[right]{$1-\Beta$};
			\draw[->, line width=1.5pt](6,6)--(4,6); 
			\draw[->, line width=1.5pt](6,6)--(6,4); 
			\draw[<->](8,3)--(11, 3);
			\draw[color=black] (9.5, 3)node[above]{\eqref{ind9}};
		 	\draw(13,0)--(13,6)--(19,6)--(19,0)--(13,0);
			\draw [line width= 5pt, color=nicosred](13.3,6)--(18.7,6);
			\draw(16,6.1)node[above]{$U'$};
			\draw [line width= 5pt, color=nicosred](19,0.3)--(19,5.7);
			\draw(19.1,3)node[right]{$V'$};
			\draw (14,-0.1)node[below]{$\Beta'$};
			\draw (12.9,1)node[left]{$1-\Beta'$};
			\draw[->, line width=1.5pt](13,0)--(15,0); 
			\draw[->, line width=1.5pt](13,0)--(13,2); 
			\draw[nicosred,fill=nicosred](13,0) circle(7pt);
			\draw[nicosred,fill=nicosred](6,6) circle(7pt); 
		\end{tikzpicture}
 	 \end{center}
 	\caption{\small Involution \eqref{ind9}: Respectively, weights $U$ and $V$ on the south and west edges 
	and west/south transition $(\Beta,1-\Beta)$
	become weights $U'$ and $V'$ on the north and east edges and east/north transition $(\Beta',1-\Beta')$, and vice-versa.}
 \label{fig:ind}
 \end{figure}
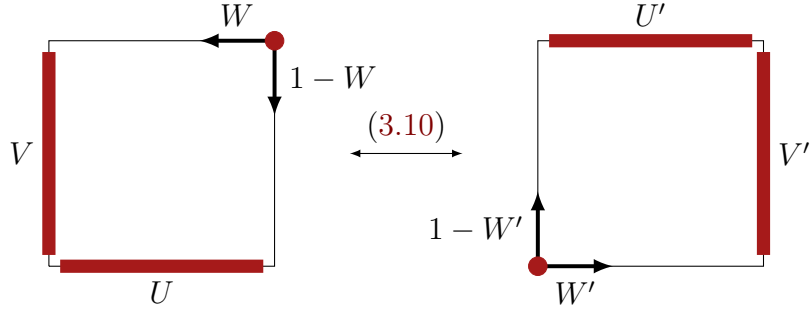

%We would like to find explicit distributions preserved by this ``flip the corner'' mapping of the basic square ($0,e_1,e_2,e_1+e_2$),   that would also give a ``reversibility'' which is captured by the involution property of the mapping.  The following lemma answers this question in the positive.

Now assume that  the transition probabilities $\pch$ come from a beta RWRE, in other words, that the variables $\{\pch_{x,x-e_1}\}_{x\in\Z^2}$ are i.i.d.\ Beta$(\alpha, \beta)$.  The next key lemma  indicates how to choose the distributions of the ratios of the boundary values $H(x)$ in order to get tractable harmonic functions.  We regard the parameters $\alpha, \beta$ of the environment fixed,  while $0<\lambda<\infty$ parametrizes two different boundary conditions in cases (a) and (b) in the lemma.

\begin{lemma}\label{ind-sol}
The equations 
	\begin{align}\label{ind9}
	U'=\frac{\Beta V+(1-\Beta)U}{V}\,,\quad V'=\frac{\Beta V+(1-\Beta)U}{U}\quad\text{and}\quad\Beta'=\frac{U(V-1)}{V-U}  	
	\end{align}
define an involution $(U,V,\Beta)\mapsto (U',V',\Beta')$ on the product space $(0,1)\times(1,\infty)\times(0,1)$.  

Let $0<\alpha,\beta,\lambda<\infty$.  

{\rm (a)}  Suppose that  $(U,V,\Beta)$ are  independent variables with distributions 
\begin{align}\label{uvw1}
U\sim\text{\rm Beta}(\alpha+\lambda,\beta), \quad V^{-1}\sim\text{\rm Beta}(\lambda,\alpha), \quad\text{and}\quad \Beta\sim\text{\rm Beta}(\alpha,\beta).
\end{align} 
	Then the triples  $(U',V',\Beta')$ and $(U,V,\Beta)$ have the same distribution.  
	
{\rm (b)}  Suppose that  $(U,V,\Beta)$ are  independent variables with distributions 	
\begin{align}\label{uvw2} U^{-1}\sim\text{\rm Beta}(\lambda,\beta), \quad V\sim\text{\rm Beta}(\beta+\lambda,\alpha), \quad\text{and}\quad \Beta\sim\text{\rm Beta}(\alpha,\beta).
\end{align} 
	Then again  the triples  $(U',V',\Beta')$ and $(U,V,\Beta)$ have the same distribution.

 \end{lemma} 

\begin{proof}
Algebra checks the involution property.  We prove part (a).  Part (b) follows by switching around  $\alpha$ and $\beta$ and by switching around  the axes.  

Let $(\Beta,\Gvar_\alpha,\Gvar_\beta,\Gvar_\lambda)$ be jointly independent with $\Beta\sim\text{Beta}(\alpha,\beta)$ and $\Gvar_\nu\sim\text{Gamma}(\nu,1)$.
Set 
	\be\label{UV8}  U=\frac{\Gvar_\alpha+\Gvar_\lambda}{\Gvar_\alpha+\Gvar_\beta+\Gvar_\lambda}\quad\text{and}\quad 
	V=\frac{\Gvar_\alpha+\Gvar_\lambda}{\Gvar_\lambda}.\ee
Then $(U,V,\Beta)$ have the desired distribution because $V$ is independent of $\Gvar_\alpha+\Gvar_\lambda$.

Compute
	\begin{align}
	\nn &U'=\Beta+(1-\Beta)\frac{U}{V}=\Beta+(1-\Beta)\frac{\Gvar_\lambda}{\Gvar_\alpha+\Gvar_\beta+\Gvar_\lambda},\\
	%\Beta+(1-\Beta)\frac{\Gvar_\lambda}{\Gvar_\alpha+\Gvar_\beta+\Gvar_\lambda},\\
	\nn &V'=\Beta\frac{V}{U}+1-\Beta=\Beta\frac{\Gvar_\alpha+\Gvar_\beta+\Gvar_\lambda}{\Gvar_\lambda}+1-\Beta,\\
\label{Be'9} 	&\Beta'=\frac{U(V-1)}{(V-U)}=\frac{\Gvar_\alpha}{\Gvar_\alpha+\Gvar_\beta}.
	\end{align}	

$\Beta'$ is independent of the pair $(U',V')$ because it is independent of $\Gvar_\alpha+\Gvar_\beta$. It also clearly has the same distribution as $\Beta$.

It remains to show that $(U',V')$ has the same distribution as $(U,V)$.  For this set
	\[Y=\frac{\Gvar_\lambda}{\Gvar_\alpha+\Gvar_\beta+\Gvar_\lambda}.\]
Observe that 
	\[U'=\Beta+(1-\Beta)Y\quad\text{and}\quad V'=\Beta Y^{-1}+1-\Beta.\]
Also
	\[\Beta'+(1-\Beta')Y=Y+\Beta'(1-Y)=\frac{\Gvar_\alpha+\Gvar_\lambda}{\Gvar_\alpha+\Gvar_\beta+\Gvar_\lambda}=U\]
and similarly
	\[\Beta' Y^{-1}+1-\Beta'=V.\]
Furthermore, $(Y,\Beta')$ are independent and so are $(Y,\Beta)$. Consequently, the two pairs have the same distribution
and then $(U',V')$ has the same distribution as $(U,V)$. The lemma is proved.
\end{proof}

  Observe from \eqref{ind9} that 
 \be\label{ind-83}   \frac{W'}{U}+ \frac{1-W'}{V}=1\qquad\text{and}\qquad 
  \frac{W}{U'}+ \frac{1-W}{V'}=1 .  \ee
 This is how the Doob transformed transition probabilities  arise from a  given forward transition $(W',1-W')$ or backward transition $(W,1-W)$.  
We derive the probability distribution of $W'/U$ (which is the same as that of $W/U'$).  $_{2}F_{1}$ below is the standard Gauss hypergeometric function
\be\label{hg3} _{2}F_{1}(a,b,c; z)=\sum_{k=0}^\infty \frac{(a)_k\,(b)_k}{(c)_k}\,\frac{z^k}{k!}  \ee
where $(c)_{k}=c(c+1)\dotsm(c+k-1)$ denotes the ascending factorial.   Other examples of rational functions of beta variables whose densities involve hypergeometric functions appear  in \cite{Duf-10, Pha-00}.

% The density function of the random variable   $\doob^\xi_{x,x+e_1}$ is given by 

\begin{proposition}\label{pr:hypgeom}  The random variables $W'/U$ and  $W/U'$ of Lemma \ref{ind-sol} have the following density function $g_\lambda$  on the interval $(0,1)$. 
 
 In case  {\rm(a)} under assumption \eqref{uvw1}, 
 \be\label{f_R2.1}
g_\lambda(x) =  
\frac{B(\alpha+\lambda, \alpha+\beta)}{B(\alpha+\lambda, \beta)} \cdot \frac{x^{\alpha-1}(1-x)^{\lambda-1}}{B(\lambda,\alpha)}   \cdot 
  {_{2}F_1}(\alpha+\lambda, \alpha+\lambda,2\alpha+\beta+\lambda;x).  
\ee

 In case  {\rm(b)} under assumption \eqref{uvw2}, 
 \be\label{f_R2.2}
\wt g_\lambda(x) = 
\frac{B(\beta+\lambda, \alpha+\beta)}{B(\beta+\lambda, \alpha)} \cdot \frac{x^{\lambda-1}(1-x)^{\beta-1}}{B( \lambda, \beta)}   \cdot 
  {_{2}F_1}(\beta+\lambda, \beta+\lambda,\alpha+2\beta+\lambda;1-x).  
 \ee
Neither $g_\lambda$ nor  $\wt g_\lambda$ is  the density function of any beta distribution.  
\end{proposition}

\begin{proof}  Consider case  (a).  Let $F_{V^{-1}}$ denote the  Beta$(\lambda, \alpha)$ c.d.f.\ of $V^{-1}$.    Fix $0<x<1$.   From  $W'/U= (1-V^{-1})/(1-UV^{-1})$, 
\begin{align}
\nn g_\lambda(x) &= \frac{d}{dx}\P\Bigl( \frac{W'}U \le x\Bigr)=  \frac{d}{dx} \P\Bigl( V^{-1}\ge \frac{1-x}{1-xU}\Bigr) \\[4pt] 
%\label{89} 
\nn &=  \frac1{B(\alpha+\lambda, \beta)} \int_0^1   \frac{\partial}{\partial x}\bigl( 1-F_{V^{-1}}(\tfrac{1-x}{1-xu})\bigr)  
u^{\alpha+\lambda-1}(1-u)^{\beta-1} \,du \\[4pt]
\label{888.2}   &=  \frac{x^{\alpha-1}(1-x)^{\lambda-1}}{B(\lambda, \alpha)B(\alpha+\lambda, \beta)} \int_0^1 
 (1-xu)^{-\alpha-\lambda}\,   u^{\alpha+\lambda-1}(1-u)^{\alpha+\beta-1} \,du.   
\end{align}
The last integral equals 
$B(\alpha+\lambda,\alpha+\beta) _2F_1(\alpha+\lambda, \alpha+\lambda, 2\alpha+\beta+\lambda; x)$ (equation (9.09) on page 161 in \cite{Olv-97}).  This verifies \eqref{f_R2.1}. 

In case (b) write $  {W'}/U  = 1- (1-U^{-1})/(1-U^{-1}V)$ where the last fraction has the distribution found in case (a) but with $\alpha$ and $\beta$ interchanged.  Hence we have \eqref{f_R2.2}. 

We verify that $g_\lambda$ is not a beta density, by deriving a contradiction from the assumption that for some $\gamma, \delta>0$, 
\be\label{f-be4} 
g_\lambda(x)=\frac1{B(\gamma, \delta)} \,x^{\gamma-1}(1-x)^{\delta-1}
\qquad\text{for}\ \ 0<x<1. 
\ee
Set this equal to line \eqref{888.2} and let $x\to 0$.  This forces 
  $\gamma=\alpha$  and gives  this identity:   
   \begin{align*}
1 &=  \frac{B(\gamma, \delta)}{B(\lambda, \alpha)B(\alpha+\lambda, \beta)}    \int_0^1 
   u^{\alpha+\lambda-1}(1-u)^{\alpha+\beta-1} \,du  %\\&
   =  \frac{B(\gamma, \delta)B(\alpha+\lambda, \alpha+\beta)}{B(\lambda, \alpha)B(\alpha+\lambda, \beta)}  .    
\end{align*} 
Using the identity  above, the equality of \eqref{888.2} and \eqref{f-be4} can be written as 
\be\label{f-be8} 
(1-x)^{\delta-\lambda}  = 
% \frac{B(\gamma, \delta)}{B(\lambda, \alpha)B(\alpha+\lambda, \beta)}  
\frac1{B(\alpha+\lambda,\alpha+ \beta)} 
 \int_0^1 
 (1-xu)^{-\alpha-\lambda}\,   u^{\alpha+\lambda-1}(1-u)^{\alpha+\beta-1} \,du .    
\ee 
As $x\to 1$ the right-hand side converges to a limit that is strictly larger than 1 and  possibly infinite.  This forces $\lambda>\delta$ on the left.  Expand both sides of \eqref{f-be8} in series:
\be\begin{aligned}\nn
\sum_{k=0}^\infty \frac{(\lambda-\delta)_k}{k!} x^k 
&=  \frac1{B(\alpha+\lambda,\alpha+ \beta)}  \sum_{k=0}^\infty  \frac{(\alpha+\lambda)_{k}}{k!}\,x^k \int_0^1 
    u^{\alpha+\lambda+k-1}(1-u)^{\alpha+\beta-1} \,du\\
%    &= \sum_{k=0}^\infty  \frac{B(\alpha+\lambda+k,  \alpha+\beta)}{B(\alpha+\lambda,\alpha+ \beta)} \, \frac{(\alpha+\lambda)_{k}}{k!}\,r^k \\
%        &= \sum_{k=0}^\infty  \frac{\Gamma(2\alpha+ \beta+\lambda)\Gamma(\alpha+\lambda+k) }{\Gamma(\alpha+\lambda) \Gamma(2\alpha+ \beta+\lambda+k)}    \, \frac{(\alpha+\lambda)_{k}}{k!}\,r^k \\
   &= \sum_{k=0}^\infty       \frac{(\alpha+\lambda)_{k}(\alpha+\lambda)_{k}}{(2\alpha+ \beta+\lambda)_k\,k!}\,x^k
\end{aligned}\ee
The equality of the coefficients for $k=1$ and $k=2$ gives a contradiction. 
\end{proof} 

\subsection{Harmonic functions on quadrants} \label{sec:bdry} 
 
  Lemma \ref{ind-sol} is applied  to construct two processes:  $(\p^\lambda, \Bd^\lambda)$ using   case (a) of the lemma  and $(\wt\p^\lambda, \wt\Bd^\lambda)$ using  case (b).  Parameters $(\alpha, \beta)$ are fixed while in both cases $0<\lambda<\infty$.  $\p^\lambda$ and  $\wt\p^\lambda$ are new i.i.d.\ Beta$(\alpha, \beta)$ environments.    $\Bd^\lambda$ and $\wt\Bd^\lambda$ are harmonic functions on $\Z_+^2$ that give rise to Doob transformed transition probabilities $\pi^\lambda$  and  $\wt\pi^\lambda$, respectively.  

 The need for two cases (a) and (b) arises from the two-to-one connection between parameters $\xi\in(\ri\Uset)\setminus\{\llnv\}$ and $0<\lambda<\infty$,  given  in Lemma \ref{lam-xi-t}\eqref{lam-xi}.    Then  to parametrize in terms of $\xi$,  let $\lambda(\xi)$ be given  by Lemma \ref{lam-xi-t}\eqref{lam-xi} and  define 
 \be\label{bar-def} 
 (\bar\p^\xi, \bar\Bd^\xi, \bar\pi^\xi)=
 \begin{cases} 
 (\p^{\lambda(\xi)}, \Bd^{\lambda(\xi)}, \pi^{\lambda(\xi)}), &\xi_1\in(\llnv_1,1)=(\frac\alpha{\alpha+\beta}, 1)\\[4pt]
  (\wt\p^{\lambda(\xi)}, \wt\Bd^{\lambda(\xi)}, \wt\pi^{\lambda(\xi)}), &\xi_1\in(0,\llnv_1)=(0,\frac\alpha{\alpha+\beta}). 
 \end{cases} 
\ee 
 This way we establish in Theorem \ref{thm:LLN} below  that $\xi\in(\ri\Uset)\setminus\{\llnv\}$ is  the    limiting velocity  of the Doob transformed RWRE with transition $\bar\pi^\xi$.

The law of large numbers velocity $\llnv=(\frac\alpha{\alpha+\beta},\frac\beta{\alpha+\beta})$ does not arise from any transition $\pi^\lambda$ or  $\wt\pi^\lambda$ for a finite $\lambda$.  As stated in   Lemma \ref{lam-xi-t}\eqref{lam-xi},   $\llnv$ corresponds to $\lambda=\infty$.   In the proof of Lemma \ref{ind-sol} above   letting $\lambda\to\infty$ in \eqref{UV8} yields $U=V=1$, which then gives also $U'=V'=1$.   We could define $\Bd^\infty=\wt\Bd^\infty=1$ which corresponds to the constant harmonic function.

We now perform construction \eqref{ind700}--\eqref{ind703}  of harmonic functions $\Bd^\lambda$ and forward transition probabilities $\p^\lambda$.   The distributional properties of the construction will come  from  part (a) of Lemma \ref{ind-sol}.    The  given inputs of the construction are   boundary variables and backward transition probabilities in the bulk.   We create  simultaneously infinitely many coupled systems indexed by the parameter   $0<\lambda<\infty$.   Remark \ref{rm-wt} below comments on the similar construction of $(\wt\p^\lambda, \wt\Bd^\lambda)$ based on case (b) of Lemma \ref{ind-sol}.

 Let  $\Pplus$ denote the joint distribution of mutually independent random variables   \label{Pbar} 
	\be\label{Del1} \{\Delta_{(i,0)},\Delta_{(0,j)}, \pch_{x,\,x-e_1}: i,j\in\N, x\in\N^2 \}\ee
  with marginal distributions  
$\Delta_{(i,0)},\Delta_{(0,j)}\sim$  Uniform(0,1)  and  $\pch_{x,\,x-e_1}\sim$ Beta$(\alpha,\beta)$.  
%Expectation under $\Pplus$ will be denoted by $\Eplus$.
Set \begin{align}\label{pch2}\pch_{x,\,x-e_2}=1-\pch_{x,\,x-e_1}.\end{align}

For fixed  positive $a$ and $b$,  let  $F^{-1}(\cdot\,;a,b) :[0,1]\to[0,1]$ denote the inverse function of the Beta$(a,b)$ c.d.f.\ \eqref{be-cdf}.  For   $0<\lambda<\infty$ define  coupled boundary variables on the coordinate axes: 
	\be \label{Bdla-def}\begin{aligned}
	\Bd^\lambda_{(i-1,0),(i,0)}&=F^{-1}(\Delta_{(i,0)};\alpha+\lambda,\beta)
	\quad\text{for $i\ge 1$}  \\
\text{and}\qquad 
	 \Bd^\lambda_{(0,j-1),(0,j)}&=\frac1{F^{-1}(\Delta_{(0,j)};\lambda,\alpha)} \quad\text{for $j\ge 1$} .\end{aligned}\ee
%(We only keep $\lambda$ in the notation on the left-hand side because we think of $\alpha$ and $\beta$ as being given at the outset.)
  $\{\Bd^\lambda_{(i-1,0),(i,0)}:i\ge 1\}$ are i.i.d.\ Beta$(\alpha+\lambda,\beta)$, $\{(\Bd^\lambda_{(0,j-1),(0,j)})^{-1} :j\ge 1\}$ are
i.i.d.\ Beta$(\lambda,\alpha)$, and the two collections are independent of each other and of $\{\pch_{x,\,x-e_1}:x\in\N^2\}$.    

For each $\lambda>0$,  apply equations \eqref{ind700}--\eqref{ind702} inductively  to define random variables
	\begin{align}\label{Beta-bdry}
	\{\Bd^\lambda_{x,\,x+e_1},\Bd^\lambda_{x,\,x+e_2},\p^\lambda_{x,\,x+e_1}:x\in\Z_+^2\}  
	\end{align}
indexed by the full quadrant.  	
For $x\in\Z_+^2$  define additionally 
\[  \p^\lambda_{x,\,x+e_2}=1-\p^\lambda_{x,\,x+e_1}. \]
%and 
%\[ \Bd^\lambda_{x+z,\,x}=\frac1{\Bd^\lambda_{x,\,x+z}} .   \]
Conservation equations 
	\begin{align}\label{cocycle3}
	\Bd^\lambda_{x,\,x+e_1}\Bd^\lambda_{x+e_1,x+e_1+e_2}=\Bd^\lambda_{x,\,x+e_2}\Bd^\lambda_{x+e_2,\,x+e_1+e_2}
	\end{align}
are satisfied around all unit squares.    
Consequently  we can extend the definition of $\Bd^\lambda_{x,\,x+e_i}$  from directed nearest-neighbor edges to  $\Bd^\lambda_{x,y}$ for all $x,y\in\Z_+^2$ so that $\Bd^\lambda_{x,\,x}=1$ and  	\begin{align}\label{cocycle}
	\Bd^\lambda_{x,y}\Bd^\lambda_{y,z}=\Bd^\lambda_{x,z} \qquad \text{ for all   $x,y,z\in\Z_+^2$.} 
	\end{align}
In the sequel we  write   $\Bd^\lambda(x,y)$ for $\Bd^\lambda_{x,y}$ when subscripts are not convenient.

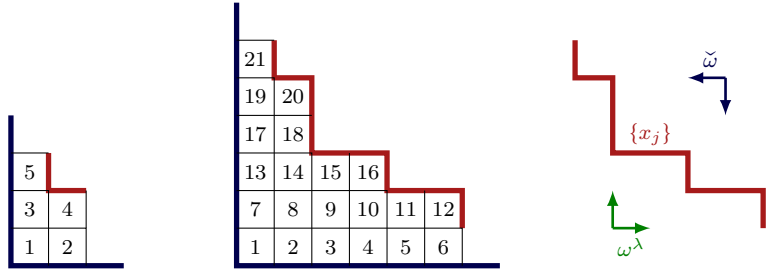
\begin{figure}[h]
 	\begin{center}
 		 \begin{tikzpicture}[>=latex, scale=0.5]
		 	\draw[line width=2pt, color=darkblue](3,0)--(0,0)--(0,4);
			\draw[line width=2pt, color=nicosred](1,3)--(1,2)--(2,2);
			\draw(0,3)--(1,3);
			\draw(0,2)--(1,2)--(1,0);
			\draw(0,1)--(2,1)--(2,0);
			\draw(2,2)--(2,1);
			\draw(0.5,0.5)node{\tiny 1};
			\draw(1.5,0.5)node{\tiny 2};
			\draw(0.5,1.5)node{\tiny 3};
			\draw(1.5,1.5)node{\tiny 4};
			\draw(0.5,2.5)node{\tiny 5};
		 \begin{scope}[shift={(6,0)}]
		 	\draw [line width= 2pt, color=nicosred](1,6)--(1,5)--(2,5)--(2,3)--(4,3)--(4,2)--(6,2)--(6,1);
	 		\draw [line width= 2pt, color=darkblue](0,7)--(0,0)--(7,0);
			\draw(0,6)--(1,6);
			\draw(0,5)--(1,5)--(1,0);
			\draw(0,4)--(2,4);
			\draw(0,3)--(2,3)--(2,0);
			\draw(0,2)--(4,2)--(4,0);
			\draw(3,3)--(3,0);
			\draw(5,2)--(5,0);
			\draw(0,1)--(6,1)--(6,0);
			\draw(0.5,0.5)node{\tiny 1};
			\draw(1.5,0.5)node{\tiny 2};
			\draw(2.5,0.5)node{\tiny 3};
			\draw(3.5,0.5)node{\tiny 4};
			\draw(4.5,0.5)node{\tiny 5};
			\draw(5.5,0.5)node{\tiny 6};
			\draw(0.5,1.5)node{\tiny 7};
			\draw(1.5,1.5)node{\tiny 8};
			\draw(2.5,1.5)node{\tiny 9};
			\draw(3.5,1.5)node{\tiny 10};
			\draw(4.5,1.5)node{\tiny 11};
			\draw(5.5,1.5)node{\tiny 12};
			\draw(0.5,2.5)node{\tiny 13};
			\draw(1.5,2.5)node{\tiny 14};
			\draw(2.5,2.5)node{\tiny 15};
			\draw(3.5,2.5)node{\tiny 16};
			\draw(0.5,3.5)node{\tiny 17};
			\draw(1.5,3.5)node{\tiny 18};
			\draw(0.5,4.5)node{\tiny 19};
			\draw(1.5,4.5)node{\tiny 20};
			\draw(0.5,5.5)node{\tiny 21};
		\end{scope}
		\begin{scope}[shift={(14,0)}]
			\draw [line width= 2pt, color=nicosred](1,6)--(1,5)--(2,5)--(2,3)--(4,3)--(4,2)--(6,2)--(6,1);
			\draw(3,3.5) node{\tiny\textcolor{nicosred}{$\{x_j\}$}};
	 		\draw[->, line width=1pt, color=darkblue](5,5)--(4,5);
			\draw[->, line width=1pt, color=darkblue](5,5)--(5,4);
			\draw(4.6,5.5) node{\tiny\textcolor{darkblue}{$\pch$}};
			\draw[->, line width=1pt, color=darkgreen](2,1)--(3,1);
			\draw[->, line width=1pt, color=darkgreen](2,1)--(2,2);
			\draw(2.5,0.5) node{\tiny\textcolor{darkgreen}{$\p^\lambda$}};
		\end{scope}
		\end{tikzpicture}
 	 \end{center}
 	\caption{\small Illustration of the corner-flipping procedure. Left and center: To obtain the $\Bd^\lambda$ values on the thick edges of the down-right path inside the quadrant start with the known values on the boundary edges and 
	consecutively flip the corners of the squares, for example in the indicated order. 
	%Center: To obtain the top red path from the bottom blue one flip the corners in the indicated order. 
	Right: ratios $\Bd^\lambda$ along the down-right path, transitions $\pch$ out of sites northeast of the  path, and transitions $\p^\lambda$ southwest of it are jointly independent.}
 \label{fig:flip}
 \end{figure}

  A {\it down-right} lattice path $\{x_j\}_{j\in\Z}$  is by definition a nearest-neighbor path with increments  $x_j-x_{j-1}\in \{e_1,-e_2\}$.     Note that any bounded portion of a down-right path in $\Z_+^2$ can be obtained by finitely many corner flips starting from the path   $x_j= (j^+, j^-)$ that lies on the coordinate  axes.   A single corner flip is the transformation of variables  $(U,V,W)$ into  $(U', V', W')$  in Figure \ref{fig:ind}.  Figure \ref{fig:flip} illustrates successive corners flips.  
  By 	Lemma \ref{ind-sol}(a), each iteration of \eqref{ind700}--\eqref{ind702} preserves the properties stated in   the next proposition.

\begin{proposition}\label{stat-pr}   Let random variables \eqref{Del1} and \eqref{Bdla-def} be given, and define the process \eqref{Beta-bdry} inductively through  \eqref{ind700}--\eqref{ind702}.   Then for each  $0<\lambda<\infty$ we have the following distributional properties. 

Random variables $\{\p^\lambda_{x,\,x+e_1}:x\in\Z_+^2\}$ are i.i.d.\ {\rm Beta}$(\alpha,\beta)$. 
For each  $x\in\Z_+^2$ we have the marginal distributions  
\be\label{Bd-8.9}   \text{$\Bd^\lambda_{x,\,x+e_1}\sim\text{\rm Beta}(\alpha+\lambda,\beta)$ \ \ and \ \ 
$\dfrac1{\Bd^\lambda_{x,\,x+e_2}}\sim\text{\rm Beta}(\lambda,\alpha)$. }\ee

For any down-right path $\{x_j\}_{j\in\Z}$ in $\Z_+^2$,  the following random variables are all mutually independent:  
 \[  \{\Bd^\lambda_{x_j,x_{j+1}}:  j\in\Z\}, \ \  
\bigcup_{j\in\Z}\{\pch_{z,\,z-e_1}: z\ge x_j+(1,1)\},\ \  \text{and} \ \   \bigcup_{j\in\Z}\{\p^\lambda_{x,\,x+e_1}:0\le x\le x_j-(1,1)\}.  \] 
%  $\{\Bd^\lambda_{x_j,x_{j+1}}:  j\in\Z\}$,  $\bigcup_j\{\pch_{x,\,x-e_1}: x\ge x_j+e_1+e_2\}$, and  $\bigcup_j\{\p^\lambda_{x,\,x+e_1}:0\le x\le x_j-e_1-e_2\}$. 

In particular, we have the translation invariance of the joint distribution: for any $a\in\Z_+^2$, 
\be\label{la-45.2} \begin{aligned} 
&(\p^\lambda_{x, x+e_1}, \,\Bd^\lambda_{u,v}, \,\pch_{z,z-e_1})_{x,u,v\in\Z_+^2, \, z\in\N^2} \\
&\qquad\qquad 
\ \overset{d}=\   
(\p^\lambda_{a+x,\, a+x+e_1}, \,\Bd^\lambda_{a+u,\,a+v}, \,\pch_{a+z,\,a+z-e_1})_{x,u,v\in\Z_+^2, \, z\in\N^2}
\end{aligned} \ee
\end{proposition}
 
Translation invariance \eqref{la-45.2} is a consequence of the down-right path  statement:  with a new origin at $a$,    the edge variables $\Bd^\lambda_{a+(i-1)e_k, \,a+ie_k}$ for $i\in\N$ and $k\in\{1,2\}$ and the bulk variables $(\pch_{z,\,z-e_1})_{z\in a+\N^2}$ have the same joint distribution as the original ones given in \eqref{Del1} and \eqref{Bdla-def}.

Equations  \eqref{ind-83}  give the  identities 
 	\begin{align}\label{aver2}
	\frac{\p^\lambda_{x,\,x+e_1}}{\Bd^\lambda_{x,\,x+e_1}}+\frac{\p^\lambda_{x,\,x+e_2}}{\Bd^\lambda_{x,\,x+e_2}}=1\qquad\text{for }x\in\Z_+^2
	\end{align}
and
 	\begin{align}\label{aver3}
	\frac{\pch_{x,\,x-e_1}}{\Bd^\lambda_{x-e_1,\,x}}+\frac{\pch_{x,\,x-e_2}}{\Bd^\lambda_{x-e_2,\,x}}=1\qquad\text{for }x\in\N^2.
	\end{align}
	
%Combined with \eqref{cocycle} this gives the following, for any fixed $v,w\in\Z_+^2$: 
%  	\begin{align}\label{aver6}
%	 \p^\lambda_{x,\,x+e_1} \Bd^\lambda_{x+e_1, w}+  \p^\lambda_{x,\,x+e_2}\Bd^\lambda_{x+e_2,w}= \Bd^\lambda_{x,w} \qquad\text{for }x\in\Z_+^2
%	\end{align}
%and
% 	\begin{align}\label{aver7}
%	 \pch_{x,\,x-e_1}\Bd^\lambda_{v, x-e_1}+ \pch_{x,\,x-e_2}\Bd^\lambda_{v, x-e_2}=\Bd^\lambda_{v,x} \qquad \text{for }x\in\N^2.
%	\end{align}

Consider   the RWRE $P^{\p^\lambda}$ that uses forward  transitions $\p^\lambda$.   
Combining \eqref{aver2}  with \eqref{cocycle}  gives the following for any fixed $y\in\Z_+^2$:  
  	\begin{align}\label{aver6}
	 \p^\lambda_{x,\,x+e_1} \Bd^\lambda_{x+e_1, y}+  \p^\lambda_{x,\,x+e_2}\Bd^\lambda_{x+e_2,y}= \Bd^\lambda_{x,y} \qquad\text{for }x\in\Z_+^2. 
	\end{align}
In other words,  for any fixed $y$,  $\Bd^\lambda_{x,y}$ is a harmonic function of $x$ for transition probabilities $\p^\lambda_{x,\,x+e_i}$ on $\Z_+^2$.  
%Similarly, for any fixed $v$,  $\Bd^\lambda(v,x)$ is a harmonic function for transition probabilities $\pch_{x,\,x-e_i}$ on $\N_+^2$. 
%For any such harmonic functions $H(x)$ we have
%\be\label{harm761} 
% H(u)= E^{\p^\lambda}_u[H(X_{\tau^-_y}, y)]. \ee	
In particular,  for two points $u\le y$ in $\Z_+^2$ we have
\be\label{harm76} 
 \Bd^\lambda_{u,y}= E^{\p^\lambda}_u[\Bd^\lambda(X_{\tau^-_y}, y)]. \ee	
By \eqref{aver3} the same harmonic function $\Bd^\lambda$ works for backward transitions $\pch$ and we have
\be\label{harm77} 
\Bd^\lambda_{u,y}= E^{\pch}_y[\Bd^\lambda(u,X_{\tau^+_u})].
\ee

We  perform a Doob transform on $P^{\p^\lambda}$ by introducing 
transition probabilities 
\be\label{pila2} \pi^{\lambda}_{x,\,x+e_i}=\frac{\p^\lambda_{x,\,x+e_i}}{\Bd^\lambda_{x,\,x+e_i}}\,,\quad i\in\{1,2\}.\ee
%Now that we have an explicit expression of the rate $I(\xi)$ we can give a nice interpretation of bijection \eqref{x-lam}. Consider the following Doob-transform of transitions $\pi^\lambda$: 
The RWRE  that uses   transitions $\pi^\lambda$  is the {\it $\Bd^\lambda$-tilted} RWRE and its quenched path measure is denoted by $P^{\pi^{\lambda}}$.   Let $x_{0,k}=(x_0,\dotsc,x_k)$ be an up-right path from $x_0=u$  that first enters the boundary $\B^-_y$ (recall definition \eqref{Bpm})  at  the endpoint $x_k$.  Then 
	\begin{align}\label{pila8}
	\begin{split}
	P^{\pi^\lambda}_u\{X_{0,k}=x_{0,k}\}
	&=\prod_{i=0}^{k-1}\frac{\p^\lambda_{x_i,x_{i+1}}}{\Bd^\lambda_{x_i,x_{i+1}}}
	=\frac{P^{\p^\lambda}_u\{X_{0,k}=x_{0,k}\}}{\Bd^\lambda_{u,x_k}}\\
	&=\frac{P^{\p^\lambda}_u\{X_{0,k}=x_{0,k}\}\Bd^\lambda_{x_k,y}}{\Bd^\lambda_{u,y}}
	=\frac{E^{\p^\lambda}_u[ \Bd^\lambda(X_{\tau^-_y},y) ,  X_{0,k}=x_{0,k}]  }{E^{\p^\lambda}_u[ \Bd^\lambda(X_{\tau^-_y},y)]}.
	\end{split}
	\end{align}
A particular consequence that we use in subsequent sections   is the following identity  for the probability of hitting one of the two parts  of the boundary.   For fixed  $u\le y$ in $\Z_+^2$ and $ i\in\{1,2\}$, summing \eqref{pila8} over all paths that enter $\B^-_y$  at a point of $ \B^{(-i)}_y$ gives 
\be\label{pila10}
P^{\pi^\lambda}_u\{X_{\tau^-_y} \in \B^{(-i)}_y\} =  \frac{E^{\p^\lambda}_u[ \Bd^\lambda(X_{\tau^-_y},y) ,  X_{\tau^-_y} \in \B^{(-i)}_y]  }{E^{\p^\lambda}_u[ \Bd^\lambda(X_{\tau^-_y},y)]}.
\ee

Equation \eqref{aver3}  says that the same harmonic function $\Bd^\lambda$ works for backward transitions $\pch$ as well. Hence we define also the backwards Doob transform
\be\label{pichla} 
\pich^{\lambda}_{x,\,x-e_i}=\frac{\pch_{x,\,x-e_i}}{\Bd^\lambda_{x-e_i,\,x}}\,,\quad i\in\{1,2\}.  
\ee
Then as  above  for fixed $u\le y$ in $\Z_+^2$ and a down-left path $x_{0,k}$ started  from $y$ that first enters $\B_u^+$ at $x_k$, 
	\begin{align}\label{pila88}
	\begin{split}
	P^{\pich^\lambda}_y\{X_{0,k}=x_{0,k}\}
%	&=\prod_{i=0}^{k-1}\frac{\p^\lambda_{x_i,x_{i+1}}}{\Bd^\lambda_{x_i,x_{i+1}}}
%	=\frac{P^{\p^\lambda}_u\{X_{0,k}=x_{0,k}\}}{\Bd^\lambda_{u,x_k}}\\
	&=\frac{P^{\pch}_y\{X_{0,k}=x_{0,k}\}\Bd^\lambda_{u,x_k}}{\Bd^\lambda_{u,y}}
	=\frac{E^{\pch}_y[ \Bd^\lambda(u,X(\tau^+_u)) ,  X_{0,k}=x_{0,k}]  }{E^{\pch}_y[ \Bd^\lambda(u,X(\tau^+_u))]}.
	\end{split}
	\end{align}
%and
%\be\label{pila11}
%\[P^{\pich^\lambda}_y\{X(\tau^+_u) \in \B^{(+i)}_u\} =  \frac{E^{\pch}_y[ \Bd^\lambda(u,X(\tau^+_u)) ,  X(\tau^+_u) \in \B^{(+i)}_u]  }{E^{\pch}_y[ \Bd^\lambda(u,X(\tau^+_u))]}\,, \quad i\in\{1,2\}.\]
%\ee

\begin{remark}\label{rm-wt}   Let us comment briefly on the    version of the construction above  that produces $(\wt\p^\lambda, \wt\Bd^\lambda)$ based on  case (b) of Lemma \ref{ind-sol}.  
   Instead of \eqref{Bdla-def},  begin with 
	\be \label{Bdla-def7}\begin{aligned}
	\wt\Bd^\lambda_{(i-1,0),(i,0)}&=\frac1{F^{-1}(\Delta_{(i,0)};\lambda,\beta)} 
	\quad\text{for $i\ge 1$}  \\
\text{and}\qquad 
	 \wt\Bd^\lambda_{(0,j-1),(0,j)}&=F^{-1}(\Delta_{(0,j)};\beta+\lambda,\alpha)\quad\text{for $j\ge 1$} .\end{aligned}\ee
Equations \eqref{ind700}--\eqref{ind703} are iterated exactly as before.  
Proposition \eqref{stat-pr} is valid word for word for $(\wt\p^\lambda, \wt\Bd^\lambda, \pch)$, except that \eqref{Bd-8.9} is replaced with 
\be\label{Bd-8.10}   \text{  $\dfrac1{\wt\Bd^\lambda_{x,\,x+e_1}}\sim\text{\rm Beta}(\lambda,\beta)$   \ \ and \ \ 
$\wt\Bd^\lambda_{x,\,x+e_2}\sim\text{\rm Beta}(\beta+\lambda,\alpha)$. }
\ee
The Doob-transformed transitions are defined again by 
\be\label{pila9} \wt\pi^{\lambda}_{x,\,x+e_i}=\frac{\wt\p^\lambda_{x,\,x+e_i}}{\wt\Bd^\lambda_{x,\,x+e_i}}\,,\quad i\in\{1,2\},\ee
with quenched path measure  $P^{\wt\pi^{\lambda}}$.  Equations \eqref{harm76} and \eqref{pila8} are valid for $ (\wt\p^{\lambda}, \wt\Bd^{\lambda}, \wt\pi^{\lambda})$. 
%Then  we have $P_0^{\pi^{\lambda}}\{n^{-1}X_n\to\xi\}=1$ with velocity $\xi\in\Uset$ such that $\xi_1\in(0,\frac{\alpha}{\alpha+\beta})$ is given
%by bijection \eqref{eq:lam-xi2}.
%
%No finite $\lambda$ produces velocity $\llnv=(\frac{\alpha}{\alpha+\beta}, \frac{\beta}{\alpha+\beta})$,  the law of large numbers velocity of the original RWRE with transitions $\p$.  This corresponds to the case $\lambda=\infty$.  
%
%of the theorem above that produces a limiting velocity $\xi$ with $\xi_1\in(0,\frac{\alpha}{\alpha+\beta})$ uses case (b) of Lemma \ref{ind-sol} to construct the process $\Bd^\lambda$. 
\hfill$\triangle$\end{remark}

Now let $\lambda(\xi)$ be given  by Lemma \ref{lam-xi-t}\eqref{lam-xi} for  $\xi=(\xi_1, 1-\xi_1)\in(\ri\Uset)\setminus\{\llnv\}$.   
 Combine the two constructions  $(\p^{\lambda}, \Bd^{\lambda}, \pi^{\lambda})$ and $ (\wt\p^{\lambda}, \wt\Bd^{\lambda}, \wt\pi^{\lambda})$ by defining $ (\bar\p^\xi, \bar\Bd^\xi, \bar\pi^\xi)$ by \eqref{bar-def} for all $\xi\in(\ri\Uset)\setminus\{\llnv\}$.   The quenched path measure of the RWRE that uses transition $\bar\pi^\xi$ is given by 
 \be\label{Pxi}   P^{\bar\pi^\xi}_x= 
 \begin{cases} 
 P^{\pi^{\lambda(\xi)}}_x, &\xi_1\in(\llnv_1, 1)\\[4pt]
  P^{\wt\pi^{\lambda(\xi)}}_x, &\xi_1\in(0,\llnv_1). 
 \end{cases} 
\ee 
%As explained below \eqref{bar-def}, we could include $\llnv$ in this family by setting by ?????   (Not sure. Would $\p^\infty$ be a function of $\pch$?)  

\begin{theorem}\label{thm:LLN}
We have this almost sure law of large numbers:   for all $\xi\in(\ri\Uset)\setminus\{\llnv\}$,  
\[  P^{\bar\pi^\xi}_0\{n^{-1}X_n\to\xi\}=1 \qquad   \text{$\Pplus$-almost surely. }\]
%Let $\lambda>0$.
%The $\Bd^\lambda$-tilted RWRE that uses  transition probability \eqref{pila2}  has an almost sure law of large numbers with velocity $\xi\in\Uset$ such that $\xi_1\in(\frac{\alpha}{\alpha+\beta},1)$ is given
%by bijection \eqref{eq:lam-xi1}. That is, $\Pplus$-almost surely
%	\[P_0^{\pi^{\lambda}}\{n^{-1}X_n\to\xi\}=1.\]
%A similar statement holds when $\xi_1\in(0,\frac{\alpha}{\alpha+\beta})$ and bijection \eqref{eq:lam-xi2} is used.   \note{State properly! Will need to define the reflected version of the boundary model..}  
\end{theorem}

%%%%%%%%%%%%   Not sure if remark below about \lambda=\infty is of any use.
%\begin{remark}\label{no tilt}
%To produce velocity $\xi=\llnv=(\frac{\alpha}{\alpha+\beta}, \frac{\beta}{\alpha+\beta})$, which is the law of large numbers velocity of the original RWRE with transitions $\p$,
%the required tilt is given by $\lambda=\infty$. Since a Beta($a,b$) random variable converges almost surely to 1 as $a\to\infty$, we see that
%$\Bd(\infty)\equiv1$.
%In other words, no tilt is required to produce the law of large numbers velocity. 
%\note{Say better.} 
%\end{remark}

\begin{proof}    %[Proof of Theorem \ref{thm:LLN}]
%By \eqref{cocycle} we see that 
%	\[P_0^{\pi^{\lambda}}\{X_n=x\}=P_0^{\p^\lambda}\{X_n=x\}/\Bd_{0,x}^\lambda.\]
%
%Applying  Proposition \ref{stat-pr} to path $x_j=x+(j^+,j^-)$ we see that the distribution of 
%	\begin{align}\label{rho-shift}
%	\{\Bd_{y,y+e_1}^\lambda,\Bd^\lambda_{y,y+e_2},\p^\lambda_{y,y+e_1}:y\ge x,\lambda>0\}\cup\{\pch_{y,y-e_1}:y\in x+\N^2\}
%	\end{align}
%does not depend on $x\in\Z_+^2$. This allows us to extend the definition of $\Bd^\lambda_{x,y}$ to all $x,y\in\Z^2$.
%
We give the details for the case $\xi_1\in(\llnv_1, 1)=(\frac{\alpha}{\alpha+\beta},1)$ with $\lambda=\lambda(\xi)$.   
By  translation invariance (Proposition \ref{stat-pr}) we can extend $\log\Bd^\lambda$ to a process $\{\log\Bd^\lambda_{x,y}:x,y\in\Z^2\}$ indexed by the entire lattice.  This process  has the shift-invariance and additivity properties of \eqref{B89},  in other words it is a stationary $L^1$ cocycle. 
%in the language of Definition 2.1 of \cite{Geo-Ras-Sep-17-ptrf-1}.  
Such processes satisfy a uniform ergodic theorem under certain regularity assumptions, as for example given in Theorem A.3 in the Appendix of \cite{Geo-etal-15}.  
Variable    $\log\Bd^\lambda_{0,\,e_i}$  is integrable and  %has strictly more than two moments and   
\eqref{aver2} gives the  lower bound  $\log\Bd^\lambda_{x,\,x+e_i}\ge \log\p^\lambda_{x,\,x+e_i}$   in terms of an i.i.d.\ process with strictly more than two moments.  This is sufficient for   Theorem A.3 of \cite{Geo-etal-15} which     gives the almost sure  limit 
   \label{erg-thm}
	\begin{align}\label{whatever}
	\lim_{n\to\infty}n^{-1}\max_{\abs{x}_1\le n}\abs{\log\Bd_{0,x}^\lambda-m(\lambda)\cdot x}=0,
	\end{align}
with    mean vector 
	\begin{align*}
	m(\lambda)
	&=\E[\log\Bd^\lambda_{0,e_1}]e_1+\E[\log\Bd_{0,e_2}^\lambda]e_2\\
	&=(\psi_0(\alpha+\lambda)-\psi_0(\alpha+\beta+\lambda))e_1+(\psi_0(\alpha+\lambda)-\psi_0(\lambda))e_2.
	\end{align*}

Proposition \ref{stat-pr} says that under $\Pplus$ transitions $\p^\lambda$ have the same distribution as $\p$ does under $\P$.
Then, by \eqref{qLDP} and \eqref{whatever} we have $\Pplus$-almost surely
	\begin{align*}
	\lim_{n\to\infty}n^{-1}\log P_0^{\pi^{\lambda}}\{X_n=[n\zeta]\}
&=	\lim_{n\to\infty}\bigl( \, n^{-1}\log P_0^{\p^{\lambda}}\{X_n=[n\zeta]\}- n^{-1}\log \Bd^\lambda_{0, \,[n\zeta]} \,\bigr) \\
	&=-I_q(\zeta)+\zeta_1\psi_0(\alpha+\beta+\lambda)+\zeta_2\psi_0(\lambda)-\psi_0(\alpha+\lambda).\end{align*} 
In other words, the distribution of $X_n/n$ under $P_0^{\pi^{\lambda}}$ satisfies a (quenched) large deviation principle with rate function 
	\[I_q^\lambda(\zeta)=I_q(\zeta)-\zeta_1\psi_0(\alpha+\beta+\lambda)-\zeta_2\psi_0(\lambda)+\psi_0(\alpha+\lambda).\]
By the strict convexity of $I_q$ and its expression \eqref{Iq},    $I^{\lambda}_q(\zeta)$ has a unique zero at $\zeta=\xi$ with $\xi_1$ given by the right-hand side of \eqref{eq:lam-xi1}. 
This proves Theorem \ref{thm:LLN}.
%
% 
%Using symmetry as in the proof of Theorem \ref{thm:Iq} in Appendix \ref{Iq-pf} gives the complementary picture for velocities $\xi$ with 
%$\xi_1\in(0,\frac{\alpha}{\alpha+\beta})$.
\end{proof}

As the last point, let us record monotonicity and continuity  that are valid   for the boundary variables by definition  \eqref{Bdla-def} and then extended by the construction.  For $x,y\in\Z_+^2$
	\begin{align}\label{monotone1}
	\gamma>\lambda>0\ \Longrightarrow \ \Bd^\gamma_{x,\,x+e_1}>\Bd^\lambda_{x,\,x+e_1}\quad\text{and}\quad\Bd^\gamma_{x,\,x+e_2}<\Bd^\lambda_{x,\,x+e_2}
	\end{align}
and
	\begin{align}\label{cont1}
	\Bd^\gamma_{x,\,y}\ \mathop{\longrightarrow}_{\gamma\to\lambda}\ \Bd^\lambda_{x,\,y} 
	\quad\text{and}\quad   \p^\gamma_x \ \mathop{\longrightarrow}_{\gamma\to\lambda}\ \p^\lambda_x\,. 
	\end{align}
The limits in  \eqref{cont1} are valid as stated also for  $(\wt\p, \wt\Bd)$, but the monotonicity is reversed:  
	\begin{align}\label{monotone2}
	\gamma>\lambda>0\ \Longrightarrow \ \wt\Bd^\gamma_{x,\,x+e_1}<\wt\Bd^\lambda_{x,\,x+e_1}\quad\text{and}\quad\wt\Bd^\gamma_{x,\,x+e_2}>\wt\Bd^\lambda_{x,\,x+e_2}
	\end{align}

\subsection{Global harmonic functions}

In this section we construct the process  $B^\xi_{x,y}$ discussed in the results of Section \ref{sec:doob}.  To have a single point of  reference, we summarize the construction and its properties in the next Theorem \ref{th:Buse},   then derive the claims made in  Section \ref{sec:doob}, and after that proceed to prove Theorem \ref{th:Buse} piece by piece.  

 The probability space $\OSP$  in the theorem  is the product space  $\Omega=[0,1]^{\Z^2}$ of beta environments $\p=(\p_{x,\,x+e_1}:x\in\Z^2)$ where the variables $\p_{x,\,x+e_1}$ are i.i.d.\ Beta$(\alpha,\beta)$-distributed.   Shift mappings    $T_z$ act by  $(T_z\p)_{x,\,x+e_i}=\p_{x+z,\,x+z+e_i}$ for $x,z\in\Z^2$.  Velocities $\xi\in\ri\Uset$ are denoted by $\xi=(\xi_1,\xi_2)=(\xi_1, 1-\xi_1)$, and the distinguished   velocity is $\llnv=(\frac{\alpha}{\alpha+\beta}, \frac{\beta}{\alpha+\beta})$.   A down-right path $\{x_i\}\subset\Z^2$ in part \eqref{B:B=Bd} below satisfies $x_{i+1}-x_{i}\in\{e_1, -e_2\}$.     Note that the increment distributions  in part (a.1) below are those of case (a) of Lemma \ref{ind-sol},  while part (a.2) corresponds to case (b) of Lemma \ref{ind-sol}.
 
\begin{theorem}\label{th:Buse}
Fix $0<\alpha,\beta<\infty$.   On the probability space $\OSP$ there exists a stochastic process $\{B^\xi_{x,y}(\p):x,y\in\Z^2,\xi\in\ri\Uset\}$ with  the following properties.

\noindent
{\rm I. Distribution and expectations. } 

\begin{enumerate}[label={\rm(\alph*)}, ref={\rm\alph*}] %[\ \ \rm(a)]  
\itemsep=2pt
\item\label{B:beta}   For $\xi=\llnv$ the process  $B^{\llnv}_{x,y}$ is identically zero. 
For $\xi\in(\ri\Uset)\setminus\{\llnv\}$,  the marginal distributions and  expectations are as follows, with $\lambda(\xi)$ given by \eqref{eq:lam-xi1}--\eqref{eq:lam-xi2}.   

 {\rm (a.1)}  	For  $\xi_1\in(\llnv_1,1)$, 
 \begin{align*}%\label{B-11.2} 
 e^{B^\xi_{x,\,x+e_1}}\sim {\rm Beta}(\alpha+\lambda(\xi), \beta)
 \quad\text{and}\quad 
  e^{-B^\xi_{x,\,x+e_2}}\sim {\rm Beta}(\lambda(\xi), \alpha), 
 \end{align*}
 and so 
	\begin{align}\label{EB1}
	\begin{split}
	&\E[B^\xi_{x,\,x+e_1}]=\psi_0(\alpha+\lambda(\xi))-\psi_0(\alpha+\beta+\lambda(\xi))\quad\text{and}\\
	&\E[B^\xi_{x,\,x+e_2}]=\psi_0(\alpha+\lambda(\xi))-\psi_0(\lambda(\xi)). 
	\end{split}
	\end{align}
	
 {\rm (a.2)}  	For  $\xi_1\in(0,\llnv_1)$, 
  \begin{align*}%\label{B-11.2} 
 e^{-B^\xi_{x,\,x+e_1}}\sim {\rm Beta}(\lambda(\xi), \beta)
 \quad\text{and}\quad 
  e^{B^\xi_{x,\,x+e_2}}\sim {\rm Beta}(\beta+\lambda(\xi), \alpha), 
 \end{align*}
 and so 
	\begin{align}\label{EB2}
	\begin{split}
	&\E[B^\xi_{x,\,x+e_1}]=\psi_0(\beta+\lambda(\xi))-\psi_0(\lambda(\xi))\quad\text{and}\\
	&\E[B^\xi_{x,\,x+e_2}]=\psi_0(\beta+\lambda(\xi))-\psi_0(\alpha+\beta+\lambda(\xi)).
	\end{split}
	\end{align}

\item\label{B:indep} For any $z\in\Z^2$, the variables  $\{B^\xi_{x,y}(\p):x,y\not\le z,\xi\in\ri\Uset\}$ are independent of  the variables $\{\p_{x+e_1}:x\le z\}$. 

\item\label{B:B=Bd}
For a fixed $\xi\in(\ri\Uset)\setminus\{\llnv\}$,  the joint distribution of $(\p, B^\xi)$  is the same as that of $(\bar\p^\xi, \log\bar\Bd^\xi)$ defined in \eqref{bar-def}.  This distribution is  described in Proposition \ref{stat-pr} and Remark \ref{rm-wt}. In particular,   on any down-right path $\{x_i\}_{i\in\Z}$ on $\Z^2$ the variables  $\{B^\xi_{x_i,\,x_{i+1}}\}_{i\in\Z}$ are independent.

\item\label{B:qldp} The quenched  large deviation rate function of \eqref{qLDP}  satisfies 
\[I_q(\xi)=-\inf_{\zeta\in\ri\Uset}\big\{\E[B^\zeta_{0,e_1}]\xi_1+\E[B^\zeta_{0,e_2}]\xi_2\big\} \qquad\text{for all $\xi\in\Uset$}.\] %\textcolor{red}{(or $\ri\Uset$ ??? Check.)} } \]
The infimum  is uniquely attained at $\zeta=\xi$.

\end{enumerate} 

\noindent 
{\rm II. Pointwise properties.}   There exists an event  $\Omega_0\subset\Omega$ such that  $\P(\Omega_0)=1$   and the following statements hold for each $\p\in\Omega_0$, $\xi, \zeta\in\ri\Uset$, and $x,y,z\in\Z^2$.  

\begin{enumerate}[resume,label={\rm(\alph*)}, ref={\rm\alph*}] %[\ \ \rm(a)]

\item\label{B:coc}   Cocycle properties:  stationarity 
\be\label{stat7} B^\xi_{x+z,y+z}(\p)=B^\xi_{x,y}(T_z\p) \ee
and additivity 
\be\label{add7}  B^\xi_{x,y}(\p)+B^\xi_{y,z}(\p)=B^\xi_{x,z}(\p). \ee
In particular,  $B^\xi_{x,x}(\p)=0$ and $B^\xi_{x,y}(\p)=-B^\xi_{y,x}(\p)$.

\item\label{B:rec} Harmonic increments: 
	\begin{align}\label{B:recovery}
	\p_{x,\,x+e_1}e^{-B^\xi_{x,\,x+e_1}(\p)}+\p_{x,\,x+e_2}e^{-B^\xi_{x,\,x+e_2}(\p)}=1.
	\end{align}

\item\label{B:mono} Monotonicity: If  $\xi\cdot e_1<\zeta\cdot e_1$ then 
	\[B^\xi_{x,\,x+e_1}\ge B^\zeta_{x,\,x+e_1}\quad\text{and}\quad B^\xi_{x,\,x+e_2}\le B^\zeta_{x,\,x+e_2}.\]
\item\label{B:cadlag}  $B^\xi_{x,y}(\p)$ is a cadlag function of $\xi_1\in(0,1)$.
\end{enumerate}

\noindent 
{\rm III. Limits. }  %The following statements hold for a fixed $\xi\in\ri\Uset$. 
For each fixed  $\xi\in\ri\Uset$ there exists an event  $\Omega^{(\xi)}_0\subset\Omega$ that can vary with $\xi$, has   $\P(\Omega^{(\xi)}_0)=1$,    and is such that  the following statements hold for each $\p\in\Omega^{(\xi)}_0$ and $x,y\in\Z^2$.  

% The event $\Omega^{(\xi)}_0$ can vary with $\xi$. 

\begin{enumerate}[resume,label={\rm(\alph*)}, ref={\rm\alph*}] % [\ \ \rm(a)]

\item\label{B:cont} %Fix $\xi\in\ri\Uset$ and a sequence $\xi^n\in\ri\Uset$ with $\xi^n\to\xi$. Then for $\P$-almost every $\p$ and for any $x,y\in\Z^2$
For any sequence $\xi^n\in\ri\Uset$ such that  $\xi^n\to\xi$, we have 
$\ddd \lim_{n\to\infty}B^{\xi^n}_{x,y}(\p)=B^\xi_{x,y}(\p).$
%	\[\lim_{n\to\infty}B^{\xi^n}_{x,y}=B^\xi_{x,y}.\]
	
\item\label{B:limit} %Fix $\xi\in\ri\Uset$. Then for $\P$-almost every $\p$, for any $x,y\in\Z^2$,   and for any sequence $z_N$ with $z_N/N\to\xi$ we have the limit 
For any sequence $z_N\in\Z^2$ with  $\abs{z_N}_1\to\infty$ and $z_N/N\to\xi$ we have the limit 
	\begin{align}\label{Busemann}
	B^\xi_{x,y}(\p)=\lim_{N\to\infty}\big(\log P_x^\p\{X_{\abs{z_N-x}_1}=z_N\}-\log P_y^\p\{X_{\abs{z_N-y}_1}=z_N\}\,\big).
	\end{align}
%\textcolor{red}{It should also be true that
%	\begin{align}\label{Busemann2}
%	B^\xi_{x,y}(\p)=\lim_{N\to\infty}\Big(\log P_{-z_n}^\p\{X_{\abs{z_N-x}_1}=x\}-\log P_{y-z_N}^\p\{X_{\abs{z_N-y}_1}=y\}\Big).
%	\end{align}}

%\item\label{B:limit1} For any $\p\in\Omega_0$, any $\xi\in\ri\Uset$, any $x,y\in\Z^2$,
%and any sequence $z_n$ with $\abs{z_n}_1=n$ and $z_n/n\to\xi$ we have
%	\begin{align*}%\label{Busemann2}
%	B^{\xi-}(x,y;\p)&\le\varliminf_{n\to\infty}\Big(\log P_x^{\pi,\p}\{X_{n-\abs{x}_1}=z_n\}-\log P_y^{\pi,\p}\{X_{n-\abs{y}_1}=z_n\}\Big)\\
%			&\le\varlimsup_{n\to\infty}\Big(\log P_x^{\pi,\p}\{X_{n-\abs{x}_1}=z_n\}-\log P_y^{\pi,\p}\{X_{n-\abs{y}_1}=z_n\}\Big)\le B^{\xi}(x,y;\p).
%	\end{align*}
%Furthermore, for almost every $\p$, $\{B^\xi(x,y):\xi\in\ri\Uset\cap\Q^2\}$ is uniformly continuous in $\xi$ and as such it admits a unique extension to a continuous process.
%\item\label{B-Bd} Process
%	\[\{B^\xi(-x,-y):\xi\in\Uset,\tfrac{\alpha}{\alpha+\beta}\le \xi_1\le1,x,y\ge0\}\]
%has the same distribution under $\P$ as process
%	\[\{\log\Bd^{\lambda(\xi_1)}(x,y):\xi\in\Uset,\tfrac{\alpha}{\alpha+\beta}\le \xi_1\le1,x,y\ge0\}\]
%does under $\Pplus$. {\rm(}Here, $\lambda(\xi_1)$ is given by \eqref{x-lam}.{\rm)}

\end{enumerate}

%\noindent 
%{\rm IV. Uniqueness. }
%
%\begin{enumerate}[resume,label={\rm(\alph*)}, ref={\rm\alph*}] %[\ \ \rm(a)]
%%\setcounter{enumi}{8}
%
%
%\item\label{Bd-B}  \note{Rethink how to state this.}  Fix $\xi\in\ri\Uset$ such that $\xi_1\in(\llnv_1,1)$ and let $\lambda=\lambda(\xi)$.
%If $\{\Bd_{x,y},\p_{x,\,x+e_1}:x,y\ge0\}$ is any beta boundary system corresponding to parameter $\lambda$, 
%then $\{\p_{x,\,x+e_1}:x\in\Z_+^2\}$ are i.i.d.\ Beta$(\alpha,\beta)$ random variables and the limit in \eqref{Busemann}, computed with these transitions $\p$, satisfies
%$B^\xi_{x,y}=\log\Bd_{x,y}$, $\P$-almost surely and for any $x,y\ge0$. A similar statement holds if $\xi_1\in[0,\llnv_1]$ 
%by switching the role of the two axes.
%\end{enumerate}

\end{theorem}

%Define
%	\[h(B^\xi)=-\E[B^\xi(0,e_1)]e_1-\E[B^\xi(0,e_2)]e_2.\]
%Now extend $I$ to a function on $\R_+^2$ by defining $I(0)=0$ and $I(\xi)=\abs{\xi_1}_1 I(\xi/\abs{\xi}_1)$. 
%The variational formula in claim \eqref{B:L1} gives 
%	\[I(\xi+\e v)-I(\xi)=\abs{\xi+\e v}_1 I(\tfrac{\xi+\e v}{\abs{\xi+\e v}_1})-I(\xi)\ge h(B^\xi)\cdot(\xi+\e v)-h(B^\xi)\cdot\xi=\e h(B^\xi)\cdot v.\]
%Since $v$ is arbitrary we get
%	\begin{align}\label{grad I}
%	\nabla I(\xi)=h(B^\xi).
%	\end{align}

A few comments about the theorem.   The translation-invariant process $(\bar\p^\xi, \log\bar\Bd^\xi)$ referred to in part \eqref{B:B=Bd} was constructed in \eqref{bar-def} in Section \ref{sec:bdry} only on the quadrant  $\Z_+^2$.   In order for part \eqref{B:B=Bd} above to make full sense,    $(\bar\p^\xi, \log\bar\Bd^\xi)$  must be extended from   $\Z_+^2$ to the full lattice $\Z^2$ by Kolmogorov's extension theorem.   It is also important to distinguish when $\xi$ is fixed and when it can vary.   The distributional equality of  $B^\xi$ and  $\log\bar\Bd^\xi$  is not valid jointly across different $\xi$ because the joint distribution of $\{B^\xi\}$ is not the one constructed in Section  \ref{sec:bdry} through a coupling with uniform random variables.  
Note also the distinction between \eqref{B:cadlag} and \eqref{B:cont}:  at a fixed $\xi$ there is continuity almost surely,  but globally over $\xi$ the path  is cadlag. 

We prove the results of Section \ref{sec:doob}.   As given in \eqref{pi-doob}, the transformed transition probability is defined by $\doob^\xi_{x,\,x+e_i}(\p)= \p_{x,\,x+e_i}e^{-B^\xi_{x,\,x+e_i}(\p)}$.

\begin{proof}[Proof of Theorem \ref{th:Bus}]
 Theorem \ref{th:Bus} is a subset of Theorem \ref{th:Buse}.  \end{proof}

\begin{proof}[Proof of Theorem \ref{th:pi-xi}]  
We can express $\doob^\xi_{x,\,x+e_1}$  in terms of the increments $(B^\xi_{x,\,x+e_1}, B^\xi_{x,\,x+e_2})$:   
\be\label{pixi67}   \doob^\xi_{x,\,x+e_1}= \frac{e^{B^\xi_{x,\,x+e_2}}-1}{e^{B^\xi_{x,\,x+e_2}}-e^{B^\xi_{x,\,x+e_1}}}. 
\ee 
To prove the formula, substitute in limits \eqref{Busemann} and use the Markov property.  Note also that this is the analogue of  $W'/U=(V-1)/(V-U)$ from \eqref{ind9}. 

Given $n\in\Z$, define the down-right path $\{x^j\}$ by 
\[  x^{2k}=(n+k,-k) \quad\text{and}\quad   x^{2k+1}=(n+k+1,-k)\quad\text{for } \ k\in\Z. \]
The antidiagonal  $\{x: x_1+x_2=n\}$ is the subsequence  $\{x^{2k}\}$, and $(B^\xi_{x^{2k},\,x^{2k}+e_1}, B^\xi_{x^{2k},\,x^{2k}+e_2})$
$=(B^\xi_{x^{2k},\,x^{2k+1}}, -B^\xi_{x^{2k-1},\,x^{2k}})$.  These pairs are i.i.d.\ by part \eqref{B:B=Bd} of Theorem \ref{th:Buse}. 

 For  $\xi\in\Uset\setminus\{\llnv\}$ the law of large numbers part \eqref{pi-lln} of Theorem \ref{th:pi-xi} follows from Theorem \ref{thm:LLN} and  the observation that $(\p, B^\xi)$ has the same distribution as $(\bar\p^\xi, \log\bar\Bd^\xi)$,  as stated in part \eqref{B:B=Bd} of Theorem \ref{th:Buse}.   
 
 For $\xi=\llnv$,   $P^{\doob^{\llnv}(\p)}_0=P^\p_0$, the original path measure in an i.i.d.\ environment, and the LLN is the one in \eqref{lln}.  
\end{proof}

\begin{proof}[Proof of Theorem \ref{cond}]
 %  As already mentioned, the weak convergence parts \eqref{weakcv1} and \eqref{weakcv2} of Theorem \ref{cond} are 
   Immediate from the limits \eqref{Busemann}.  
   %
%
%In Theorem \ref{cond}, \eqref{cond:lim} follows from limit \eqref{Busemann} and
%	\begin{align*}
%	P_0^\p\{X_{n+1}=x_n+z\,|\,X_{0,n}=x_{0,n},\ x_N=z_N\}
%%	&=\frac{\prod_{i=0}^{n-1} \pi_{x_i,x_{i+1}}}\pi_{x_n,x_n+1}P_{x_n+z}^{\pi,\p}\{X_{N-n-1}=\fl{N\xi}\}}{\prod_{i=0}^{n-1} \pi_{x_i,x_{i+1}}}P_{x_n}^{\pi,\p}\{X_{N-n}=\fl{N\xi}\}}
%	&=\p_{x_n,x_n+z}\frac{P_{x_n+z}^{\p}\{X_{N-n-1}=z_N\}}{P_{x_n}^{\p}\{X_{N-n}=z_N\}}.
%	\end{align*}
 \end{proof} 

\smallskip 

\begin{theorem}\label{th:pi-pdf} For  $\xi\in(\ri\Uset)\setminus\{\llnv\}$, 
random variable  $\doob^\xi_{0,e_1}$ is not beta distributed.  Let  $\lambda(\xi)$ be given by \eqref{eq:lam-xi1}--\eqref{eq:lam-xi2} and let    $g_\lambda$ and $\wt g_\lambda$ be  the functions defined in \eqref{f_R2.1}--\eqref{f_R2.2}.    Then the 
  density function $f^\xi(x)$  of $\doob^\xi_{0,e_1}$ for  $0<x<1$ is given by   
   \be\label{pi-pdf}   f^\xi(x)= 
 \begin{cases} 
g_{\lambda(\xi)}(x), &\xi_1\in(\llnv_1, 1)\\
 \wt g_{\lambda(\xi)}(x), &\xi_1\in(0,\llnv_1). 
 \end{cases} \ee

\end{theorem}

\begin{proof}  This comes from Proposition \ref{pr:hypgeom}. Formula \eqref{pixi67},   the independence of $B^\xi_{0,e_1}$ and  $B^\xi_{0,e_2}$, and their distributions given in part \eqref{B:beta} of Theorem \ref{th:Buse} imply that $\doob^\xi_{0,e_1}$ has exactly the distribution of $W'/U$ in Proposition \ref{pr:hypgeom}. 
\end{proof} 

We turn to prove Theorem \ref{th:Buse}.   
  In addition to the probability space $\OSP$ with its beta environment $\p$,  we use the coupled processes  $\{\bar\p^\xi_{x,\,x+e_1},\bar\Bd^\xi_{x,y}:x,y\in\Z_+^2\}$ under distribution $\Pplus$, constructed in Section \ref{sec:bdry} with properties given in Proposition \ref{stat-pr} and in the subsequent discussion.   The key point is that   each   environment $\bar\p^\xi$ has the same i.i.d.\ Beta$(\alpha,\beta)$ distribution as the original environment $\p$.   The construction of $B^\xi_{x,y}$ is based on  the limits   \eqref{Busemann}.  These limits are proved by bounding ratios of hitting probabilities with the ratio variables $\bar\Bd^\xi$ from \eqref{bar-def}   whose distributions we control. 
 
%For a fixed $\lambda\in[0,\infty]$ we say $\{\Bd_{x,y},\p_{x,\,x+e_1}:x,y\ge0\}$ is a beta boundary system corresponding to parameter $\lambda$ if its distribution is the same as that of   . 

%Given a configuration $\p=\{\p_{x,\,x+e_i}:i\in\{1,2\},x\in\Z^2\}$ and a $z\in\Z^2$ let $T_z\p$ be the configuration  with $(T_z\p)_{x,\,x+e_i}=\p_{x+z,\,x+z+e_i}$, $i\in\{1,2\}$, $x\in\Z^2$.

We begin with two lemmas that   do not use the  beta distributions.   The setting for Lemmas \ref{lm:order} and \ref{lm99} is the following:   $a\in\Z^2$ and   
on the  quadrant  $\S=a+\Z_+^2$ we have a   Markov transition probability $p$   such that 
	\begin{align}\label{t:p11}   0<p_{x,\,x+e_1}=1-p_{x,\,x+e_2}<1 \end{align}
for all $x\in\S$.   Let $P_x$ with expectation $E_x$ denote the Markov chain with transition $p$ starting at $x\in\S$.
Use the standard notation for hitting probabilities:  
\[  F(x,y) =P_x( \exists n\ge 0: X_n=y). \] 

\begin{lemma}\label{lm:order}   
The following inequalities hold for all $y\in\S=a+\Z_+^2$: 
\begin{align}\label{P98}
\frac{F(a+e_1,y+e_2)}{F(a,y+e_2)} \le \frac{F(a+e_1,y)}{F(a,y)} \le \frac{F(a+e_1,y+e_1)}{F(a,y+e_1)} . 
\end{align}
The first two numerators can vanish but the denominators are all positive.  
The same inequalities hold with  $e_1$ and $e_2$  switched around.    
\end{lemma} 

\begin{proof}  The second statement follows by applying \eqref{P98} to the  transition probability $\wt p$ obtained by reflecting $p$  across the diagonal passing through $a$: for $x=(x_1,x_2)\in\Z_+^2$ set  $\wt p_{a+x,a+x+e_i}=p_{a+\wt x, a+\wt x+e_{3-i}}$,  where  $\wt x=(x_2,x_1)$.  

We prove  claim  \eqref{P98} by induction on $y$.  It is convenient to use the ratios 
\[  U_{x,y}= \frac{F(x,y)}{F(x,y-e_1)}
\quad\text{and}\quad 
V_{x,y}= \frac{F(x,y)}{F(x,y-e_2)} \qquad \text{for $x\le y$ in $\S$. } \]
The numerator does not vanish but the denominator can vanish and then the ratio has value $\infty$.   

   \eqref{P98} holds trivially for $y=a+\ell e_2$ with $\ell\ge 0$  because the first two numerators vanish while the other probabilities are positive.   Hence we may assume $y\ge a+e_1$.   By a shift of $y$  \eqref{P98} is equivalent to having 
\begin{align}
\label{P101}    U_{a,y}\le U_{a+e_1,y}  \ \text{ for $y\ge a+2e_1$} 
\quad\text{and}\quad  
V_{a,y}\ge V_{a+e_1,y}  \ \text{for $y\ge a+e_1+e_2$.}  
\end{align}
%(Note that if $y=a+ke_1$ then the first two numerators in   \eqref{P98} vanish and so   \eqref{P98} is trivially true.  Hence only $y\ge a+e_1$ needs to be considered.) 

We check  the boundaries first.  For $y=a+ke_1$ for $k\ge 2$,  $U_{a,y}=U_{a+e_1,y}=p_{y-e_1,y}$.    For $y=a+e_1+\ell e_2$ for $\ell\ge 1$, 
\begin{align*}
V_{a,y}= \frac{F(a,y-e_2) p_{y-e_2,y}+  F(a,y-e_1)p_{y-e_1,y}}{F(a,y-e_2)}   > p_{y-e_2,y} =  V_{a+e_1,y}.  
\end{align*}

It remains the check  \eqref{P101} for $y=a+ke_1+\ell e_2$ for $k\ge 2$ and $\ell \ge 1$.  
For $y\ge x+e_1+e_2$ the Markov property and assumption  \eqref{t:p11}  give 
\[  F(x,y)= F(x, y-e_1)p_{y-e_1,y}+ F(x, y-e_2)p_{y-e_2,y}   \]
from which we derive the identities 
\[  U_{x,y} =  p_{y-e_1,y}+  p_{y-e_2,y} \frac{U_{x,y-e_2}}{V_{x,y-e_1}} 
%\qquad \text{for $y\ge x+e_1$} 
 \]
and 
\[  V_{x,y} =  p_{y-e_1,y} \frac{V_{x,y-e_1}} {U_{x,y-e_2}}+  p_{y-e_2,y} 
%\qquad \text{for $y\ge a+e_2$}. 
 \]
also for  %$y\ge (x+e_1)\vee(x+e_2)\Longleftrightarrow 
$y\ge x+e_1+e_2$.   

Now proceed by induction on $y\ge a+2e_1+e_2$, beginning with $y=a+2e_1+e_2$, and then taking $e_1$ and $e_2$ steps.   The boundary cases checked above together with the  induction assumption give $U_{a,y-e_2}\le U_{a+e_1,y-e_2}$  and $V_{a,y-e_1}\ge V_{a+e_1,y-e_1}$.  
Then the identities above give $U_{a,y}\le U_{a+e_1,y}$  and $V_{a,y}\ge V_{a+e_1,y}$.   
\end{proof}

\begin{lemma} \label{lm99}   Let $v\ge a$  on $\Z_+^2$ and set $y=v+e_1+e_2$.   Suppose $f(x)>0$ for $x$ on the boundary $\B^{-}_y$.  
We have the following inequalities. 

   For $a+e_1\le v$:  
	\begin{align}\label{mono}
	\begin{split}
	\frac{E_a[f(X_{\tau^-_y}),  X_{\tau^-_y}\in \B^{(-2)}_y]}{E_{a+e_1}[f(X_{\tau^-_y}),  X_{\tau^-_y}\in \B^{(-2)}_y]}
	&\le \frac{F(a,v)}{F(a+e_1,v)} % \\ 	&
	\le\frac{E_a[f(X_{\tau^-_y}),  X_{\tau^-_y}\in \B^{(-1)}_y]}{E_{a+e_1}[f(X_{\tau^-_y}),  X_{\tau^-_y}\in \B^{(-1)}_y]} . 
		\end{split}
	\end{align}
	
For $a+e_2\le v$: 	
 	\begin{align}\label{mono2}
	\begin{split}
	\frac{E_a[f(X_{\tau^-_y}),  X_{\tau^-_y}\in \B^{(-1)}_y]}{E_{a+e_2}[f(X_{\tau^-_y}),  X_{\tau^-_y}\in \B^{(-1)}_y]}
	&\le \frac{F(a,v)}{F(a+e_2,v)} %\\&
	\le\frac{E_a[f(X_{\tau^-_y}),  X_{\tau^-_y}\in \B^{(-2)}_y]}{E_{a+e_2}[f(X_{\tau^-_y}),  X_{\tau^-_y}\in \B^{(-2)}_y]} .
	\end{split}
	\end{align}
\end{lemma}

\begin{proof}   For the proof, fix $v$ and do   induction on $\abs{v-a}_1\ge1$. Consider the case $a=v-ke_1$ for  $k\ge1$. Then
%	\[E_a[f(X_{\tau^-_y}),  X_{\tau^-_y}\in \B^{(-2)}_y]=\p_{v+e_1,v+e_1+e_2}\prod_{i=0}^{\abs{v-a}_1} \p_{a+ie_1,u+(i+1)e_1}\]
%	and
%	\[P^{p}_a(X_{\abs{u-v}_1+1}=v)=\prod_{i=0}^{\abs{v-a}_1-1}\p_{a+ie_1,u+(i+1)e_1}.\]
%Consequently,
	\[\frac{E_a[f(X_{\tau^-_y}),  X_{\tau^-_y}\in \B^{(-2)}_y]}{E_{a+e_1}[f(X_{\tau^-_y}),  X_{\tau^-_y}\in \B^{(-2)}_y]}=p_{a,a+e_1}=\frac{F(a,v)}{F(a+e_1,v)}\,.\]
On the other hand,   when the walk is required to hit $\B^{(-1)}_y$, both steps $e_1$ and $e_2$ are feasible from $a$, and so 
	\[E_a[f(X_{\tau^-_y}),  X_{\tau^-_y}\in \B^{(-1)}_y]\ge p_{a,a+e_1}\,E_{a+e_1}[f(X_{\tau^-_y}),  X_{\tau^-_y}\in \B^{(-1)}_y].  \]
This establishes   \eqref{mono}.   \eqref{mono2}  for  $a=v-ke_2$ for $k\ge1$ follows in a  symmetric manner. In particular,  we have the full conclusion for $\abs{v-a}_1=1$.

Suppose \eqref{mono}--\eqref{mono2} hold for all pairs $a\le v$ with $\abs{v-a}_1=\ell\ge1$.  Consider $a\le v$ with $\abs{v-a}_1=\ell+1$.
We already know the result when  $a\in\{v-(\ell+1)e_1,v-(\ell+1)e_2\}$.
Thus we may assume   that $a<v$ coordinatewise.
For $i\in\{1,2\}$ take the identity 
	\[p_{a,a+e_1}\frac{E_{a+e_1}[f(X_{\tau^-_y}),  X_{\tau^-_y}\in \B^{(-i)}_y]}{E_{a}[f(X_{\tau^-_y}),  X_{\tau^-_y}\in \B^{(-i)}_y]}+p_{a,a+e_2}\frac{E_{a+e_2}[f(X_{\tau^-_y}),  X_{\tau^-_y}\in \B^{(-i)}_y]}{E_{a}[f(X_{\tau^-_y}),  X_{\tau^-_y}\in \B^{(-i)}_y]}=1\]
%and
%	\begin{align*}
%	&\frac{E_{a+e_1}[f(X_{\tau^-_y}),  X_{\tau^-_y}\in \B^{(-i)}_y]}{E_{a}[f(X_{\tau^-_y}),  X_{\tau^-_y}\in \B^{(-i)}_y]}\cdot \frac{E_{a+e_1+e_2}[f(X_{\tau^-_y}),  X_{\tau^-_y}\in \B^{(-i)}_y]}{E_{a+e_1}[f(X_{\tau^-_y}),  X_{\tau^-_y}\in \B^{(-i)}_y]}\\
%	&\qquad 
%		=\frac{E_{a+e_2}[f(X_{\tau^-_y}),  X_{\tau^-_y}\in \B^{(-i)}_y]}{E_{a}[f(X_{\tau^-_y}),  X_{\tau^-_y}\in \B^{(-i)}_y]}\cdot\frac{E_{a+e_1+e_2}[f(X_{\tau^-_y}),  X_{\tau^-_y}\in \B^{(-i)}_y]}{E_{a+e_2}[f(X_{\tau^-_y}),  X_{\tau^-_y}\in \B^{(-i)}_y]}\,.\end{align*}  
and rearrange it to yield the two identities 
	\begin{align}\label{ind}
	\begin{split}
	&\qquad\qquad\qquad \frac{{E_{a}[f(X_{\tau^-_y}),  X_{\tau^-_y}\in \B^{(-i)}_y]}}{{E_{a+e_1}[f(X_{\tau^-_y}),  X_{\tau^-_y}\in \B^{(-i)}_y]}}
	=p_{a,a+e_1} \\
	&+p_{a,a+e_2}\frac{{E_{a+e_2}[f(X_{\tau^-_y}),  X_{\tau^-_y}\in \B^{(-i)}_y]}}{{E_{a+e_1+e_2}[f(X_{\tau^-_y}),  X_{\tau^-_y}\in \B^{(-i)}_y]}} \biggl( \frac{{E_{a+e_1}[f(X_{\tau^-_y}),  X_{\tau^-_y}\in \B^{(-i)}_y]}}{{E_{a+e_1+e_2}[f(X_{\tau^-_y}),  X_{\tau^-_y}\in \B^{(-i)}_y]}} \biggr)^{-1} 
	\end{split}
	\end{align}
 and 
	\begin{align}\label{ind2}
	\begin{split}	&\qquad\qquad\qquad \frac{{E_{a}[f(X_{\tau^-_y}),  X_{\tau^-_y}\in \B^{(-i)}_y]}}{{E_{a+e_2}[f(X_{\tau^-_y}),  X_{\tau^-_y}\in \B^{(-i)}_y]}}
	=p_{a,a+e_2}\\ 
	&+p_{a,a+e_1}\frac{{E_{a+e_1}[f(X_{\tau^-_y}),  X_{\tau^-_y}\in \B^{(-i)}_y]}}{{E_{a+e_1+e_2}[f(X_{\tau^-_y}),  X_{\tau^-_y}\in \B^{(-i)}_y]}}\biggl(  \frac{{E_{a+e_2}[f(X_{\tau^-_y}),  X_{\tau^-_y}\in \B^{(-i)}_y]}}{{E_{a+e_1+e_2}[f(X_{\tau^-_y}),  X_{\tau^-_y}\in \B^{(-i)}_y]}}\biggr)^{-1}.
	\end{split}
	\end{align}
Derive  the analogous  equations for the ratios of hitting probabilities  from   the identity 	
	\[p_{a,a+e_1}\frac{F(a+e_1,v)}{F(a.v)}+p_{a,a+e_2}\frac{F(a+e_2,v)}{F(a,v)}=1. \]	
Apply the induction assumption on the right-hand sides of \eqref{ind} and \eqref{ind2} and their counterparts for the ratios of hitting probabilities.  This verifies \eqref{mono} and \eqref{mono2} for $u$. 
\end{proof}

The remainder of the proof relies on special properties of the beta environment.  The next proposition gives  control over limits of hitting probability ratios through the harmonic functions constructed for Proposition \ref{stat-pr}.    Hitting probabilities under environment $\p$ are denoted by 
\be\label{HTdef}  \HPa{\p}{x}{y}=\HP{\p}{x}{y}=P^\p_x( \exists n\ge 0: X_n=y). \ee
When $x\le y$ this is of course $\HP{\p}{x}{y}=P^\p_x(  X_{\abs{y-x}_1}=y)$ which we also use occasionally when the notation is not too heavy.   

\begin{proposition}\label{prop:comp}
Fix $\xi\in(\ri\Uset)\setminus\{\llnv\}$.   
If $\eta,\zeta\in\ri\Uset$ are such that 
\be\label{56.7}  % \llnv_1\le 
\eta_1<\xi_1<\zeta_1 \ee
 then 
 %for almost every choice of $\{\Bd_{x,y},\p_{x,\,x+e_1}:x,y\ge0\}$ and
  for all $x\in\Z^2_+$ and $z\in\Z^2$ we have almost surely 
	\[\varlimsup_{N\to\infty}\frac{\HP{\bar\p^\xi}{x}{[N\zeta]+z}}{\HP{\bar\p^\xi}{x+e_1}{[N\zeta]+z}}\le\bar\Bd^\xi_{x,\,x+e_1}
	\le\varliminf_{N\to\infty}\frac{\HP{\bar\p^\xi}{x}{[N\eta]+z}}{\HP{\bar\p^\xi}{x+e_1}{[N\eta]+z}}\]
and
	\[\varliminf_{N\to\infty}\frac{\HP{\bar\p^\xi}{x}{[N\zeta]+z}}{\HP{\bar\p^\xi}{x+e_2}{[N\zeta]+z}}\ge\bar\Bd^\xi_{x,\,x+e_2}
	\ge\varlimsup_{N\to\infty}\frac{\HP{\bar\p^\xi}{x}{[N\eta]+z}}{\HP{\bar\p^\xi}{x+e_2}{[N\eta]+z}}\,.\]
%%%%%   Old version with explicit probabiltiies 
%	\[\varlimsup_{N\to\infty}\frac{\HP{\p^\lambda}{x}{[N\zeta]+z}}{\HP{\p^\lambda}{x+e_1}{[N\zeta]+z}}\le\Bd^\lambda_{x,\,x+e_1}
%	\le\varliminf_{N\to\infty}\frac{\HP{\p^\lambda}{x}{[N\eta]+z}}{\HP{\p^\lambda}{x+e_1}{[N\eta]+z}}\]
%and
%	\[\varliminf_{N\to\infty}\frac{\HP{\p^\lambda}{x}{[N\zeta]+z}}{P_{x+e_2}^{\p^\lambda}\{X_{\abs{N\zeta]+z-x}_1-1}=[N\zeta]+z\}}\ge\Bd^\lambda_{x,\,x+e_2}
%	\ge\varlimsup_{N\to\infty}\frac{\HP{\p^\lambda}{x}{[N\eta]+z}}{\HP{\p^\lambda}{x+e_2}{[N\eta]+z}}\,.\]
%	
%Similar inequalities hold when $\xi_1\in(0,\llnv_1)$. \textcolor{blue}{(maybe actually the same equalities hold?)}
Recall from \eqref{bar-def} that these inequalities split into two separate results: one for $(\p^{\lambda(\xi)}, \Bd^{\lambda(\xi)})$ when $\xi_1\in(\llnv_1,1)$, and the other for $(\wt\p^{\lambda(\xi)}, \wt\Bd^{\lambda(\xi)})$ when $\xi_1\in(0,\llnv_1)$. 
\end{proposition}
 
\begin{proof}  
The inequalities claimed are all proved the same way.  We illustrate with the first one.  Let $y_N=[N\zeta]+z+e_1+e_2$.  
By \eqref{mono},   then \eqref{pila10} and \eqref{harm76},
\begin{align*}
&\frac{\HP{\bar\p^\xi}{x}{[N\zeta]+z}}{\HP{\bar\p^\xi}{x+e_1}{[N\zeta]+z}}
\le 
\frac{E^{\bar\p^\xi}_x[\bar\Bd^\xi(X_{\tau^-_{y_N}}, y_N), X_{\tau^-_{y_N}}\in \B^{(-1)}_{y_N}] }{E^{\bar\p^\xi}_{x+e_1}[\bar\Bd^\xi(X_{\tau^-_{y_N}}, y_N), X_{\tau^-_{y_N}}\in \B^{(-1)}_{y_N}]}\\
&=\frac{P_x^{\bar\pi^\xi}\{X_{\tau^-_{y_N}}\in \B^{(-1)}_{y_N}\}\,E^{\bar\p^\xi}_{x}[\bar\Bd^\xi(X_{\tau^-_{y_N}}, y_N)]}{P_{x+e_1}^{\bar\pi^\xi}\{X_{\tau^-_{y_N}}\in \B^{(-1)}_{y_N}\}\,E^{\bar\p^\xi}_{x+e_1}[\bar\Bd^\xi(X_{\tau^-_{y_N}}, y_N)]}
= \frac{P_x^{\bar\pi^\xi}\{X_{\tau^-_{y_N}}\in \B^{(-1)}_{y_N}\}}{P_{x+e_1}^{\bar\pi^\xi}\{X_{\tau^-_{y_N}}\in \B^{(-1)}_{y_N}\}} \cdot \bar\Bd^\xi_{x,\,x+e_1}. 
\end{align*}
	The probabilities in the last expression converge to one  by the law of large numbers of Theorem \ref{thm:LLN} because by  \eqref{56.7} the $\xi$-ray passes $\zeta$ on the left.  
\end{proof}

\begin{corollary}\label{cor:Bus}
Fix $\xi\in\ri\Uset$. Let $\p$ be an i.i.d.\ {\rm Beta$(\alpha,\beta)$} environment.
Then almost surely, for all $z\in\Z^2$, the limits
	\be\label{89.90}
	\lim_{N\to\infty}\frac{\HP{\p}{0}{[N\xi]+z}}{\HP{\p}{e_1}{[N\xi]+z}}
	\quad\text{and}\quad
	\lim_{N\to\infty}\frac{\HP{\p}{e_2}{[N\xi]+z}}{\HP{\p}{0}{[N\xi]+z}}
%%%%%%%%   Old version with explicit probabilities 
%	\lim_{N\to\infty}\frac{\HP{\p}{0}{[N\xi]+z}}{\HP{\p}{e_1}{[N\xi]+z}}
%	\quad\text{and}\quad
%	\lim_{N\to\infty}\frac{\HP{\p}{e_2}{[N\xi]+z}}{\HP{\p}{0}{[N\xi]+z}}	
%	
\ee
%	\quad\text{and}\quad
%	\varlimsup_{N\to\infty}\frac{P_0^{\p}\{X_{N-\abs{x}_1}=[N\xi]-x\}}{P_{e_1}^{\p}\{X_{N-\abs{x}_1-1}=[N\xi]-x\}}\]
%	\quad\text{and}\quad
%	\varlimsup_{N\to\infty}\frac{P_{e_2}^{\p}\{X_{N-\abs{x}_1-1}=[N\xi]-x\}}{P_0^{\p}\{X_{N-\abs{x}_1}=[N\xi]-x\}}\]
exist and are independent of $z$. 

%\noindent
\begin{enumerate}[label={\rm(\alph*)}, ref={\rm\alph*},leftmargin=*]
%{\rm (a)}  
\item When $\xi=\llnv$, the limits equal 1.  

%\noindent 
%{\rm (b)}  
\item\label{cor:Bus.b} For $\xi\ne\llnv$, let 
  $\lambda=\lambda(\xi)$ be determined by Lemma \ref{lam-xi-t}\eqref{lam-xi}.  
  \begin{enumerate}[label={\rm(\alph{enumi}.\roman*)}, ref={\rm\alph{enumi}.\roman*}] %\itemsep=6pt  
 \item If  $\xi_1\in(\llnv_1,1)$ the two limits in \eqref{89.90}  are, respectively, {\rm Beta}$(\alpha+\lambda,\beta)$ 
and {\rm Beta}$(\lambda,\alpha)$ distributed. 

\item If  $\xi_1\in(0,\llnv_1)$ the reciprocals of  the two limits in \eqref{89.90}  are, respectively, {\rm Beta}$(\lambda,\beta)$ 
and {\rm Beta}$(\beta+\lambda,\alpha)$ distributed.
\end{enumerate} 
\end{enumerate}
\end{corollary}

%In the above, the case $\lambda=\infty$ means the beta random variables equal the constant $1$.

\begin{proof}
Consider the case  $\xi_1\in(\llnv_1,1)$ and the first limit of \eqref{89.90}.    
Let $\gamma<\lambda<\delta$.    By  Lemma \ref{lam-xi-t}\eqref{lam-xi}    the velocities associated with these parameters  in the range $(\llnv_1,1)$ satisfy $\xi_1(\gamma)>\xi_1=\xi_1(\lambda)>\xi_1(\delta)$.   Hence by 
Proposition \ref{prop:comp} 
	\be\label{89.88}  \begin{aligned}
	&\Bd^\gamma_{0,e_1}\le\varliminf_{N\to\infty}\frac{\HP{\p^\gamma}{0}{[N\xi]+z}}{\HP{\p^\gamma}{e_1}{[N\xi]+z}}  %\\
\qquad \text{and}\qquad  
	\varlimsup_{N\to\infty}\frac{\HP{\p^\delta}{0}{[N\xi]+z}} {\HP{\p^\delta}{e_1}{[N\xi]+z}}  \le\Bd^\delta_{0,e_1}.
	\end{aligned}\ee
%Proposition \ref{stat-pr} implies that under $\Pplus$ environments  $\p^\gamma$ and $\p^\delta$ individually have  distribution $\P$. 
Since $\Bd^\gamma, \Bd^\delta\to\Bd^\lambda$ % and $\Bd^\delta\to\Bd^\lambda$ 
 as $\gamma, \delta\to\lambda$ % and $\delta\to\lambda$,
 by \eqref{cont1},    and since 
$\Bd^\lambda_{0,e_1}$ is Beta($\alpha+\lambda,\beta$)-distributed,  we have that 
	\[\varliminf_{N\to\infty}\frac{\HP{\p}{0}{[N\xi]+z}}{\HP{\p}{e_1}{[N\xi]+z}}\quad\text{and}\quad
	\varlimsup_{N\to\infty}\frac{\HP{\p}{0}{[N\xi]+z}}{\HP{\p}{e_1}{[N\xi]+z}}\]
are both Beta($\alpha+\lambda,\beta$) random variables. 
Since  $\varliminf\le \varlimsup$,  their equality in distribution   implies their $\P$-almost sure equality.   Same reasoning works for the second  limit of \eqref{89.90}. This proves the existence of the limits and claim (b.i), for any fixed $z$.  

To argue that the limit with $z=0$  equals  the limit with an arbitrary $z=(z_1,z_2)$,  pick integers $k_N$ so that $[k_N\xi]_2=[N\xi]_2+z_2$.   Then, depending on the relative locations of   $[k_N\xi]_1$ and $ [N\xi]_1+z_1$,  by  \eqref{P98} either  the inequality 
\be\label{89.93} \begin{aligned}
\frac{\HP{\p}{0}{[k_N\xi]}}{\HP{\p}{e_1}{[k_N\xi]}}
\ge \frac{\HP{\p}{0}{[N\xi]+z}}{\HP{\p}{e_1}{[N\xi]+z}}
\end{aligned}\ee
or its opposite is valid for infinitely many $N$.    In the limit we get again an almost sure inequality between two Beta($\alpha+\lambda,\beta$) random variables, which therefore must coincide almost surely. 

For $\xi_1\in(0,\llnv_1)$ one can repeat the same steps but use instead the processes $(\wt\p, \wt\Bd)$  that follow part (b) of Lemma \ref{ind-sol}. 
%mentioned at the end of Section \ref{Iq-pf}. 

For the case  $\xi=\llnv$,  pick   $\eta,\zeta\in\ri\Uset$  such that $\eta_1<\llnv_1<\zeta_1$.  By \eqref{P98}, 
\begin{align*}
\frac{\HP{\p}{0}{[N\zeta]+z}}{\HP{\p}{e_1}{[N\zeta]+z}}
\le
\frac{\HP{\p}{0}{[N\llnv]+z}}{\HP{\p}{e_1}{[N\llnv]+z}}
\le
\frac{\HP{\p}{0}{[N\eta]+z}}{\HP{\p}{e_1}{[N\eta]+z}}. 
\end{align*}
By the cases already proved,  the left and right ratios converge to  random variables with distributions 
Beta$(\alpha+\lambda(\zeta),\beta)$ and Beta$(\lambda(\eta),\beta)^{-1}$, respectively.   These random variables converge to 1 as we let $\zeta,\eta\to\llnv$ which sends $\lambda(\zeta), \lambda(\eta)\to\infty$ (see Lemma \ref{lam-xi-t}\eqref{lam-xi} and the middle plot of Fig.~\ref{fig:lam}).  The second ratio in \eqref{89.90} for  $\xi=\llnv$ is handled similarly.  
\end{proof}

We can now organize all the derived properties  to establish  Theorem \ref{th:Buse}.  

\begin{proof}[Proof of Theorem \ref{th:Buse}]
We begin by constructing the process for a fixed $\xi\in\ri\Uset$, then do it simultaneously for a dense countable subset of $\ri\Uset$, and finally capture all of $\ri\Uset$ with limits.  
%Start by fixing a $\xi\in\ri\Uset$. 
%We will work with the case $\xi_1\in(\llnv_1,1)$, the other interval being similar.
%Let $\lambda=\lambda(\xi)$ from Lemma \ref{lam-xi-t}.

Fix $\xi\in\ri\Uset$.   Using  Corollary \ref{cor:Bus} and shifts $P_{x+a}^{\p}(X_n=v)=P_a^{T_x\p}(X_n=v-x)$,  we  define 
\be\label{auxlim7}	\begin{aligned}
	B^\xi_{x,y}(\p)&=\lim_{N\to\infty}\big\{\log P_x^{\p}(X_{\abs{z_N-x}_1}=z_N) - \log P_{y}^{\p}(X_{\abs{z_N-y}_1}=z_N)\big\}\\
	&=\lim_{N\to\infty}\big(\log\HP{\p}{x}{z_N} - \log\HP{\p}{y}{z_N}\big). 
	\end{aligned} \ee
as an almost sure limit,  for all $x,y\in\Z^2$, and for any sequence $z_N=[N\xi]+z$ with an arbitrary fixed $z$.  (The second line is the same as the first, stated to illustrate the alternative notations we use.)   The limit is independent of the choice of $z$.    The marginal distributional properties \eqref{B:beta}, stationary cocycle  properties \eqref{B:coc}, and harmonicity \eqref{B:rec}  stated in Theorem \ref{th:Buse} follow from Corollary \ref{cor:Bus} and the structure of the limits.  

%gives the $\P$-almost sure existence and independence of $z$ of the limits
%	\begin{align}\label{auxlim}
%	B^\xi_{0,e_i}(\p)=\lim_{N\to\infty}\Big\{\log P_0^{\p}(X_{\abs{[N\xi]+z}_1}=[N\xi]+z) - \log P_{e_i}^{\p}(X_{\abs{[N\xi]+z}_1-1}=[N\xi]+z)\Big\}.
%	\end{align}
%Corollary \ref{cor:Bus}  also provides the explicit distribution of $B^\xi_{0,e_i}$ and formulas \eqref{EB1} and \eqref{EB2} follow. 
%
% Observing that $P_{x+a}^{\p}(X_n=v)=P_a^{T_x\p}(X_n=v-x)$,  we get almost sure existence of the limit in \eqref{Busemann} for this $\xi$, all $x\in\Z^2$, $y=x+e_i$, $i\in\{1,2\}$, and 
%for any  sequence of the form $z_N=[N\xi]+z$ for an arbitrary fixed $z$. Let $B^\xi_{x,\,x+e_i}$ be defined by this limit.  It follows immediately that
%	\begin{align}\label{initial1}
%      B^\xi_{x,\,x+e_1}+B^\xi_{x+e_1,x+e_1+e_2}=B^\xi_{x,\,x+e_2}+B^\xi_{x+e_2,\,x+e_1+e_2}.
%      \end{align}
%Set $B^\xi_{x+e_i,x}=-B^\xi_{x,\,x+e_i}$ and define $B^\xi_{x,y}$ for all $x,y\in\Z^2$ by
%	\begin{align}\label{initial2}
%	B^\xi_{x,y}=\sum_{i=1}^{n-1} B^\xi_{x_i,x_{i+1}}
%	\end{align}
%where $x_{0,n}$ is any nearest-neighbor path from $x$ to $y$. Thanks to \eqref{initial1} the above sum does not depend on the choice of path $x_{0,n}$.
%Now \eqref{Busemann} holds almost surely for the fixed $\xi$, all $x,y\in\Z^2$, and 
%sequences $z_N=[N\xi]+z$, and it equals $B^\xi_{x,y}$.

{\it Proof of part \eqref{B:indep}.}   The independence of the weights $\{\p_x\}$ and construction \eqref{auxlim7}  imply directly the first independence claim of  part \eqref{B:indep}. 
%that for any $z\in\Z^2$ the collection of random variables $\{B^\xi_{x,y}:x,y\not\le z\}$ is independent of $\{\p_{x+e_1}:x\le z\}$. This proves part \eqref{B:indep} for the fixed $\xi$.
%%% $\p\in\Omega_1$,  $\xi\in\Uset_0$,

{\it Proof of part \eqref{B:B=Bd} for  fixed $\xi$.} 
We write the details for the case $\xi_1\in(\llnv_1,1)$.    Consider the joint distribution of $m$ weights $\p_{z_h}$ for $1\le h\le m$ and $k+\ell$ nearest-neighbor increments  $B^\xi_{x_i,x_i+e_1}$ and $B^\xi_{y_j,y_j+e_2}$ for $1\le i\le k$ and $1\le j\le \ell$.  By a shift, we may assume that $z_h, x_i, y_j$ all lie in $\Z_+^2$.  Let $\gamma<\lambda(\xi)<\delta$ as in the proof of Corollary \ref{cor:Bus}.  Limit \eqref{auxlim7} works also in environments $\p^\gamma$ and $\p^\delta$ since they have the same i.i.d.\ beta distribution as $\p$.  Let $r_h, s_i, t_j\in\R$.  
Then inequalities \eqref{89.88} and their counterparts for $e_2$ give us these bounds:
\begin{align*}
&\P\{ \p_{z_h}\le r_h, \, e^{B^\xi_{x_i,x_i+e_1}}\le s_i,  \, e^{B^\xi_{y_j,y_j+e_2}}\ge t_j \;\forall \,h,i,j\} \\
&= \Pplus\{   \p^\gamma_{z_h}\le r_h, \, e^{B^\xi_{x_i,x_i+e_1}}(\p^\gamma)\le s_i,  \, e^{B^\xi_{y_j,y_j+e_2}}(\p^\gamma)\ge t_j \;\forall \,h,i,j\} \\
&\le\Pplus\{   \p^\gamma_{z_h}\le r_h, \, \Bd^\gamma_{x_i,x_i+e_1} \le s_i,  \, \Bd^\gamma_{y_j,y_j+e_2} \ge t_j \;\forall \,h, i,j\} 
\end{align*}
and from the other side 
\begin{align*}
&\P\{ \p_{z_h}\le r_h, \,  e^{B^\xi_{x_i,x_i+e_1}}\le s_i,  \, e^{B^\xi_{y_j,y_j+e_2}}\ge t_j \;\forall \,h,i,j\}  \\
&= \Pplus\{  \p^\delta_{z_h}\le r_h, \, e^{B^\xi_{x_i,x_i+e_1}}(\p^\delta)\le s_i,  \, e^{B^\xi_{y_j,y_j+e_2}}(\p^\delta)\ge t_j \;\forall  \,h,i,j\} \\
&\ge\Pplus\{  \p^\delta_{z_h}\le r_h, \, \Bd^\delta_{x_i,x_i+e_1} \le s_i,  \, \Bd^\delta_{y_j,y_j+e_2} \ge t_j \;\forall \,h,i,j\} . 
\end{align*}
Letting $\gamma,\delta\to\lambda(\xi)$  brings the bounds together by \eqref{cont1}:
\be\label{B:98.0}\begin{aligned}
&\P\{ \p_{z_h}\le r_h, \,  e^{B^\xi_{x_i,x_i+e_1}}\le s_i,  \, e^{B^\xi_{y_j,y_j+e_2}}\ge t_j \;\forall \,h,i,j\} \\
&= 
\Pplus\{  \p^{\lambda(\xi)}_{z_h}\le r_h, \, \Bd^{\lambda(\xi)}_{x_i,x_i+e_1} \le s_i,  \, \Bd^{\lambda(\xi)}_{y_j,y_j+e_2} \ge t_j \;\forall \,h, i,j\}.  
\end{aligned}\ee
Thus  the joint distribution of $(\p, B^\xi)$  is the same as that of $(\p^{\lambda(\xi)},  \log\Bd^{\lambda(\xi)})$  described in Proposition \ref{stat-pr}.     The independence of nearest-neighbor $B^\xi$-increments along a down-right path follows.  

Let $\Omega_0$ be the event of full $\P$-probability on which the process $B^\xi_{x,y}$ is defined by \eqref{auxlim7}  for all $\xi$ in the countable set $\Uset_0=(\ri\Uset)\cap\Q^2$.  

%Let $\Omega_0$ be the intersection of the $T_x$ shifts, over all $x\in\Z^2$, of the full $\P$-probability event on which limit \eqref{auxlim} exists for all $z\in\Z^2$ and for  $\xi$ in the countable set $\Uset_0=(\ri\Uset)\cap\Q^2$. Then $\P(\Omega_0)=1$.
%Similarly to the above, we can define $B^\xi_{x,y}$ for all $\xi\in\Uset_0$, $x,y\in\Z^2$, and $\p\in\Omega_0$. Below, we  take right limits to extend this definition to all $\xi\in\ri\Uset$. Part \eqref{B:indep} will then follow.
%
%Stationarity  \eqref{stat7} and additivity \eqref{add7}  follow directly 
%for $\p\in\Omega_0$ and  $\xi\in\Uset_0$.  
% The harmonic increments property \eqref{B:recovery}  comes from 
%\[P^\p_x\{X_{\abs{z_N-x}_1}= z_N\}= \p_{x,\,x+e_1} P^\p_{x+e_1}\{X_{\abs{z_N-x}_1-1}=z_N\} + \p_{x,\,x+e_2} P^\p_{x+e_2}\{X_{\abs{z_N-x}_1-1}=z_N\},\]
%  for $\p\in\Omega_0$ and  $\xi\in\Uset_0$.   Again, these properties will extend to all $\xi\in\ri\Uset$ once we extend our definition of $B^\xi_{x,y}$ by taking right limits.
%  

%This definition now satisfies
%	\begin{align}\label{auxi}
%	%B^{T_z\p}(x,y)=B^\p(x+z,y+z)\quad\text{and}\quad 
%	B^\xi_{x,y}+B^\xi_{y,z}=B^\xi_{x,z},
%	\end{align}
%for $\p\in\Omega_0$, $\xi\in\Uset_0$, and $x,y,z\in\Z^2$. Also, limit \eqref{Busemann} now holds for $\p\in\Omega_0$, $\xi\in\Uset_0$, $x,y\in\Z^2$, and sequence $z_N=[N\xi]$. 

Consider $\xi\in\ri\Uset$ and $\zeta,\eta\in\Uset_0$ with $\eta_1<\xi_1<\zeta_1$. Take a (possibly random) sequence $z_N$ with $z_N/N\to\xi$. Let $M=\abs{z_N}_1$.
For large enough $N$ we have $[M\eta]\cdot e_1<z_N\cdot e_1<[M\zeta]\cdot e_1$ and $[M\eta]\cdot e_2>z_N\cdot e_2>[M\zeta]\cdot e_2$. By Lemma \ref{lm:order} we have for such $N$
	\[\frac{\HP{\p}{x}{[M\zeta]}}{\HP{\p}{x+e_1}{[M\zeta]}} \le \frac{\HP{\p}{x}{z_N}}{\HP{\p}{x+e_1}{z_N}} \le 
	\frac{\HP{\p}{x}{[M\eta]}}{\HP{\p}{x+e_1}{[M\eta]}}\,.\]
The already established limit \eqref{auxlim7} with  sequences $[M\zeta]$ and $[M\eta]$ gives
	\begin{align}\label{sandwich}
	\begin{split}
	&B^\zeta_{x,\,x+e_1}
	\le \varliminf_{N\to\infty}\Big\{\log P^\p_x(X_{\abs{z_N-x}_1}=z_N) - \log P^\p_{x+e_1}(X_{\abs{z_N-x}_1-1}=z_N)\Big\}\\
	&\le \varlimsup_{N\to\infty}\Big\{\log P^\p_x(X_{\abs{z_N-x}_1}=z_N) - \log P^\p_{x+e_1}(X_{\abs{z_N-x}_1-1}=z_N)\Big\}\le B^\eta_{x,\,x+e_1}.
	\end{split}
	\end{align}

The reverse inequalities hold when $e_1$ is replaced by $e_2$. 

  \eqref{sandwich} proves  that the monotonicity in part \eqref{B:mono} holds for  $\xi, \zeta\in\Uset_0$, $\p\in\Omega_0$, and $x\in\Z^2$.   Consequently, for any $\xi\in(\ri\Uset)\setminus\Uset_0$ we can {\it define}  $B^\xi_{x,\,x+e_i}(\p) $ for $\p\in\Omega_0$ by the monotone limit 
  \be\label{auxlim9} 
 B^\xi_{x,\,x+e_i}(\p) = \lim_{\Uset_0\ni\zeta\to\xi,\, \zeta_1>\xi_1 }  B^\zeta_{x,\,x+e_i}(\p) 
 \ee
 as 
$\zeta_1$ decreases to $\xi_1$.   By shrinking $\Omega_0$ we can assume that \eqref{auxlim9} holds also when $\xi\in\Uset_0$.  (This is because the monotonicity gives an inequality in \eqref{auxlim9},  but the two sides agree in distribution and hence agree almost surely.)   By additivity on the right-hand side we can extend \eqref{auxlim9}  to define  \[  B^\xi_{x,\,y}(\p)  = \lim_{\Uset_0\ni\zeta\to\xi,\, \zeta_1>\xi_1 }  B^\zeta_{x,\,y}(\p) \qquad \text{for all $x,y\in\Z^2$ and $\p\in\Omega_0$.}   \]
This definition in terms of right limits extends  the properties proved thus far to all $\xi$.  Furthermore,  cadlag paths (part \eqref{B:cadlag}) have also been established now.  

%This monotonicity implies that for $\p\in\Omega_0$, $\xi\in\ri\Uset\setminus\Uset_0$, $x\in\Z^2$, and $i\in\{1,2\}$, the limit of $B^\zeta_{x,\,x+e_i}$, $\zeta\in\Uset_0$, exists as
%$\zeta_1$ decreases to $\xi_1$. Define $B^\xi_{x,\,x+e_i}$ as this limit. Then \eqref{initial1} holds and allows extension of  $B^\xi_{x,\,x+e_i}$    to $B^\xi_{x,y}$ for all  $x,y\in\Z^2$, using \eqref{initial2}. 
%Thus property \eqref{B:cadlag} follows and all the properties we have proved so far (namely, \eqref{B:indep}-\eqref{B:mono}) extend to $\xi\in\ri\Uset$.

Fix $\xi\in\ri\Uset$ and $i\in\{1,2\}$. The almost sure continuity of $\zeta\mapsto B^\zeta_{0,e_i}$ at $\zeta=\xi$ follows from   monotonicity \eqref{B:mono}  and from the continuity of $\zeta\mapsto\E[B^\zeta_{0,e_i}]$,  which itself is a consequence of
continuity of the polygamma functions in \eqref{EB1} and \eqref{EB2}. Claim \eqref{B:cont} follows then from the cocycle property in part \eqref{B:coc}.

Continue with a fixed $\xi\in\ri\Uset$.  Let $\zeta,\eta\to\xi$ in \eqref{sandwich} and use the almost sure continuity we just proved.  This shows   that  limit \eqref{Busemann} holds   $\P$-almost surely, simultaneously  for 
all $x\in\Z^2$, $y\in\{x+e_1,x+e_2\}$, and any sequence $z_N$ with $z_N/N\to\xi$. 
The case of a general $y\in\Z^2$  follows   from additivity.  Part \eqref{B:limit} is done.  

%Now let $\Omega_0\subset\Omega_0$ be the event
%on which \eqref{Busemann} holds for all $\xi\in\Uset_0$, $x,y\in\Z^2$, and all sequences $z_N\in\Z^2$ with $z_N/N\to\xi$. We have shown that $\P(\Omega_0)=1$.
%
%For $\p\in\Omega_0$, $\xi\in\Uset_0$, $x,y,z\in\Z^2$, and the choice $z_M=[M\xi]-z$, $N=\abs{z_M}_1$, we have
%	\begin{align*}
%	&B^\xi_{x,y}(T_z\p)\\
%	%&=\lim_{M\to\infty}\Big(\log P_x^{T_z\p}\{X_{\abs{z_M-x}_1}=z_M\}-\log P_y^{T_z\p}\{X_{\abs{z_M-y}_1}=z_M\}\Big)\\
%	&\qquad=\lim_{M\to\infty}\Big(\log P_x^{T_z\p}\{X_{\abs{[M\xi]-z-x}_1}=[M\xi]-z\}-\log P_y^{T_z\p}\{X_{\abs{[M\xi]-z-y}_1}=[M\xi]-z\}\Big)\\
%	&\qquad=\lim_{M\to\infty}\Big(\log P_{x+z}^\p\{X_{\abs{[M\xi]-x-z}_1}=[M\xi]\}-\log P_{y+z}^\p\{X_{\abs{[M\xi]-y-z}_1}=[M\xi]\}\Big)\\
%	&\qquad=B^\xi_{x+z,y+z}(\p).
%	\end{align*}
%Property \eqref{B:stat} is proved. 

We turn to part \eqref{B:qldp}. 
When  $\xi_1\in[\llnv_1,1]$, the variational formula for $I_q$ comes from  \eqref{EB1} and the explicit calculations in \eqref{I-var} in Section \ref{Iq-pf}. To minimize the  formula take the derivative of 
	\[\big(\psi_0(\alpha+\lambda(\zeta))-\psi_0(\alpha+\beta+\lambda(\zeta))\big)\xi_1+\big(\psi_0(\alpha+\lambda(\zeta))-\psi_0(\lambda(\zeta))\big)\xi_2\]
in $\zeta_1$ and set it to $0$. This gives the equation 
	\[\xi_1=\frac{\psi_1(\lambda(\zeta))-\psi_1(\alpha+\lambda(\zeta))}{\psi_1(\lambda(\zeta))-\psi_1(\alpha+\beta+\lambda(\zeta))}\]
which by Lemma \ref{lam-xi-t}\eqref{lam-xi} has a unique solution at $\zeta=\xi$. 
The case $\xi_1\in[0,\llnv_1]$ works similarly. 
%
%
%%%  Proof of the uniqueness part that has been covered up in the theorem.
%%
%It remains to prove part \eqref{Bd-B}.  
%For this fix $\xi$ and $\lambda$ as in the claim.
%%Recall measure $\Pplus$ from Section \ref{sec:bdry}. We can assume that $\Bd=\Bd^\lambda$ and $\p=\p^\lambda$. 
%The fact that $\{\p_{x,\,x+e_1}:x\in\Z^2_+\}$ are i.i.d.\ Beta$(\alpha,\beta)$ follows from Proposition \ref{stat-pr}.
%Since $\Bd$ satisfies \eqref{cocycle} and  we have already shown that $B^\xi$ satisfies \eqref{add7}, it suffices to prove that 
%$B^\xi_{x,\,x+e_i}=\log\Bd_{x,\,x+e_i}$ almost surely and for all $x\in\Z^2_+$ and $i\in\{1,2\}$. In fact, it is enough to work with $e_1$ since the argument for $e_2$ is similar.
%
%Combine Proposition \ref{prop:comp} with \eqref{Busemann} to get that almost surely and for $\zeta,\eta\in\Uset_0$ with $\eta_1<\xi_1<\zeta_1$
%	\[B^\zeta_{x,\,x+e_1}\le\log\Bd_{x,\,x+e_1}\le B^\eta_{x,\,x+e_1}.\]
%Take $\eta$ and $\zeta$ to $\xi$ and use the continuity from part \eqref{B:cont} to get that $B^\xi_{x,\,x+e_1}=\log\Bd_{x,\,x+e_1}$. The proof of the theorem is complete.
\end{proof}

 \section{Stationary beta polymer}
%\section{Stationary forward and backward Beta Polymers}
\label{Beta-Polys}
 
By looking at the random walk paths under the Doob-transformed RWRE in reverse direction, we can view  this model as a stationary  directed polymer model, called the beta polymer.  We establish this connection in the present section, and then use it in the next two sections to rely on recently published estimates in \cite{Cha-Noa-17-} for   the technical work behind our Theorems  \ref{th:B-var} and \ref{th:kpz}.  The polymer model described here is  case (1.4) on p.~4 of \cite{Cha-Noa-17-},  with their parameter triple $(\mu, \beta, \theta)$ corresponding to our $(\alpha, \beta, \lambda)$.  The notation $Z_{m,n}$ and $Q_{m,n}$ used below  matches the notation of  \cite{Cha-Noa-17-}.

Recall the backward  transition probabilities $\pch$, introduced in \eqref{Del1} and \eqref{pch2}, and 
random variables $(\Bd^\lambda,\p^\lambda)$ from \eqref{Beta-bdry}.
The quenched stationary  {\sl beta polymer} is a  polymer distribution  on up-right paths on the nonnegative first quadrant $\Z_+^2$  that start at the origin.  In our  notation this model  uses potential $V(x-e_j,e_j)=\log\pch_{x,\,x-e_j}$ across edges  $(x-e_j,x)$ for  $x\in\N$, and potential $V(x-e_j,e_j)=\log\Bd^\lambda_{x-e_j,x}$ across boundary edges  $(x-e_j,x)$ for $x\in\B^{(+j)}_0\smallsetminus\{0\}$, $j\in\{1,2\}$.
Fix a point $v=(m,n)\in\N^2$.  The point-to-point partition function for paths from $0$ to $v$ is 
\[Z_{m,n}=\sum_{y_{0,m+n}}e^{\sum_{i=0}^{m+n-1}V(y_i,y_{i+1}-y_i)}\,,\]
%	\[\Zsw_{0,v}=\sum_{y_{0,n}}e^{\sum_{i=0}^{n-1}V(y_i,y_{i+1}-y_i)}\,,\]
where the sum is over up-right  paths $y_{0,m+n}=(y_0,\dotsc,y_{m+n})$ from $0$ to $v=(m,n)$.  

If $x_i=y_{m+n-i}$ denotes the reversed path and $\ell=\min\{i: x_i\in \B^+_0\}$ is  the time  of its first entry   into the boundary $\B^+_0$, then 
\begin{align*}
e^{\sum_{i=0}^{m+n-1}V(y_i,y_{i+1}-y_i)} = \Bd^\lambda_{0, x_\ell} \prod_{i=0}^{\ell-1}\pch_{x_i,x_{i+1}}
= \Bd^\lambda_{0, x_\ell}  P^{\pch}_v\{X_{0,\ell}=x_{0,\ell}\}= \Bd^\lambda_{0, x_\ell}  P^{\pch}_v\{X_{0,\tau^+_0}=x_{0,\ell}\}. 
\end{align*} 
Summing up over the paths gives the first equality below, and the second comes from \eqref{harm77}:  
	\be\label{Z-Bd} Z_{m,n}=E^{\pch}_v[\Bd^\lambda(0,X_{\tau^+_0})]=\Bd^\lambda_{0,v}
	\qquad\text{for $v=(m,n)$.} \ee
%	\be\label{Z-Bd} \Zsw_{0,v}=E^{\pch}_v[\Bd^\lambda(0,X_{\tau^+_0})]=\Bd^\lambda_{0,v}.\ee
% The superscript SW refers to the fact that there is a different potential used on the southwest boundary  $\B^+_0$ than in the bulk.

The quenched polymer distribution  on up-right  paths $y_{0,m+n}=(y_0,\dotsc,y_{m+n})$ from $0$ to $v=(m,n)$ is
\[Q_{m,n}(y_{0,n})=\frac{e^{\sum_{i=0}^{m+n-1}V(y_i,y_{i+1}-x_i)}}{Z_{m,n}}\,.\]
%	\[\Qsw_{0,v}(y_{0,n})=\frac{e^{\sum_{i=0}^{n-1}V(y_i,y_{i+1}-x_i)}}{\Zsw_{0,v}}\,.\]
Letting again $x_i=y_{m+n-i}$  and $\ell=\min\{i: x_i\in \B^+_0\}$ and using \eqref{pila88}, 
\be\label{Q-P}
%\Qsw_{0,v}(y_{0,n})
Q_{m,n}(y_{0,n})=\frac{ \Bd^\lambda_{0, x_\ell}  P^{\pch}_v\{X_{0,\ell}=x_{0,\ell}\}}{\Bd^\lambda_{0,v}}
=P^{\pich^\lambda}_v\{X_{0,\ell}=x_{0,\ell}\} \qquad\text{for $v=(m,n)$.} 
\ee
Thus (the reverse of)  the polymer path under  $Q_{m,n}$  %$\Qsw_{0,v}$ 
is obtained by running the Doob-transformed RWRE under   $P^{\pich^\lambda}_v$ until it hits the boundary $\B^+_0$, and then following the boundary to the origin.

%Equation \eqref{pila88} shows how the quenched polymer measure $\Qsw_{0,v}$ couples with the Doob-transformed RWRE $P_v^{\pich^\lambda}$.

%%Recall $ (\bar\p^\xi, \bar\Bd^\xi, \bar\pi^\xi)$ defined by \eqref{bar-def}. 
%The {\sl stationary forward beta polymer} is given by the collection of polymer measures on up-right paths from $0$ to $v\in\N^2$ that use potential $\log\bar\p^\xi_{x,x+e_j}$ when going between 
%sites $x$ and $x+e_j$, $x<v$ (coordinatewise), and potential $\log\bar\Bd^\xi_{x,x+e_j}$, when going between sites $x\in\B^{(-j)}_v\smallsetminus\{v\}$ and $x+e_j$, $j\in\{1,2\}$.
%The point-to-point partition function for this polymer is given by
%%	\[\Zne_{0,v}=\sum_{x_{0,n}}e^{\sum_{i=0}^{n-1}V(x_i,x_{i+1}-x_i)}\,,\]
%%where the sum is over up-right  paths $x_{0,n}=(x_0,\dotsc,x_n)$ from $0$ to $v$ (so $n=\abs{v}_1$).
%%The quenched polymer measure is
%%	\[\Qne_{0,v}(x_{0,n})=\frac{e^{\sum_{i=0}^{n-1}V(x_i,x_{i+1}-x_i)}}{\Zne_{0,v}}.\]
%%The partition funtion can be rewritten as
%	\[\Zne_{0,v}=E^{\bar\p^\xi}_0[\bar\Bd^\xi(X_{\tau^-_v},v)]\]
%The superscript NE refers to the fact that there is a different potential used on the northeast boundary  $\B^-_v$ than in the bulk.
%By \eqref{harm76}, we have $\Zne_{0,v}=\bar\Bd^\xi_{0,v}$.  Equation \eqref{pila8} shows how quenched polymer measure couples with the RWRE $P^{\bar\pi^\xi}_0$.

%\section{Variance bounds for stationary harmonic functions}
\section{The variance of the increment-stationary harmonic functions}\label{sec:var}

%\textcolor{blue}{About constants convention: statements saying ``there exist a constant'' will use constants $C$ for large ones and $c$ for small ones. Statements saying ``fix a constant'' will use constants $A$ for large ones and constants $a$ for small ones.} \textcolor{blue}{(Is the previous useful?  Should we say that  $A, C, C_i$ denote typically large positive constants and $a,c$ small positive constants? F: one needs to track down what constants are fixed in statements and what constants come out in claims. 
%So I chose to give them different letters. And then thought I would alert the reader to this convention...)}
%$\Cgen$ denotes a constant that may change from term to term. $\Cgen$ and $\Cgen_i$, $i\ge1$, depend on $(\alpha,\beta,\lambda_1,\lambda_2)$, unless otherwise indicated.   

The method for bounding  the fluctuations of the walk for Theorem \ref{th:kpz} is to control the exit point of the walk from rectangles.   This  is achieved with the help of the harmonic functions $\Bd^\lambda$ and $\wt\Bd^\lambda$ constructed in Section \ref{sec:bdry}.    We work exclusively with $\Bd^\lambda$ and omit the analogous statements and proofs for $\wt\Bd^\lambda$.  Equivalently, we are treating explicitly only the case $\xi_1\in(\llnv_1,1)$ and omitting the details for $\xi_1\in(0,\llnv_1)$.  

This section gives the connection between  the fluctuations of $\log\Bd^\lambda$ and the entry  point on the boundary,  and bounds on the variance of $\log\Bd^\lambda$. Theorem \ref{th:B-var} is proved at the end of the section.

Recall the beta integral $B(a,b)$  and the c.d.f.\  $F(s;a,b)$  of the Beta$(a,b)$ distribution from  \eqref{BB} and \eqref{be-cdf}.   
Define	
\[\widetilde L(s,a,b)=-\,\frac1s\cdot \frac{\frac{\partial}{\partial a}F(s;a,b)}{\frac{\partial}{\partial s}F(s;a,b)}\,.\]

Note that $\frac{\partial}{\partial a}B(a,b)=(\psi_0(a)-\psi_0(a+b))B(a,b)$. 
A computation then gives
	\begin{align}\label{L-def}
	\widetilde L(s,a,b)
	%&=-\frac{B(a,b)^{-1}\int_0^s t^{a-1}(1-t)^{b-1}\log t\,dt-(\psi_0(a)-\psi_0(a+b))B(a,b)^{-1}\int_0^s t^{a-1}(1-t)^{b-1}\,dt}{B(a,b)^{-1}s^a(1-s)^{b-1}}\\
	&=-s^{-a}(1-s)^{1-b}\int_0^s  t^{a-1}(1-t)^{b-1}[\log t-(\psi_0(a)-\psi_0(a+b))]\,dt.
	\end{align}
Observe that \[\widetilde L(s,a,b)=s^{-a}(1-s)^{1-b}B(a,b){\rm Cov}(-\log\Beta,\one\{\Beta\le s\}),\] 
where $\Beta\sim\text{Beta}(a,b)$. Since $-\log t$ and $\one\{t\le s\}$ are 
decreasing functions of $t$ we see that $\widetilde L(s,a,b)>0$ for all $s\in(0,1)$ and $a,b>0$.    Let $L(s,\lambda)=\widetilde L(s,\alpha+\lambda,\beta)$.

 Recall hitting  times $\tau^-_v$ and $\tau^+_0$ defined in \eqref{taupm}. 
%For an up-right path $x_{0,m+n}$ from $u$ to $v$ let $\exitx\in[0,m]$ denote the last time the path was on the south boundary $u+\Z_+e_1$ and let 
%$\exity\in[0,n]$ denote the last time the path was on the west boundary $u+\Z_+e_2$. Exactly one of the two equals $0$. 
Let $\Varplus$ %and $\Covplus$  
denote the variance %and covariance 
under the coupling $\Pplus$ of Section \ref{sec:bdry}.
Let $\wcP_v^\lambda(\cdot)=\Eplus P_v^{\pich^\lambda}(\cdot)$ denote the averaged measure of the RWRE that utilizes the backward Doob-transformed transition probability $\pich^\lambda$ of \eqref{pichla}.  Its  expectation operation is  $\wcE_v^\lambda[\cdot]=\Eplus E_v^{\pich^\lambda}[\cdot]$.

\begin{theorem}\label{var(Bd)}
The following holds for all $\alpha,\beta,\lambda>0$ and all $v=(m,n)\in\N^2$:
	\begin{align}
	\begin{split}
	\Varplus(\log\Bd^\lambda_{0,v})
%	&=n(\psi_1(\lambda)-\psi_1(\alpha+\lambda))-m(\psi_1(\alpha+\lambda)-\psi_1(\alpha+\beta+\lambda))\\
%	 %&\qquad+2\Eplus\Big[\frac{E_0^{\p^\lambda}\big[\one\{X_{\tau^-_{v}}\in\B^{(-1)}_v\}\Bd^\lambda(X_{\tau^-_v},v)\sum_{i=\tau^-_{v}-m}^{n-1}L(\Bd^\lambda_{ie_1+me_2,(i+1)e_1+me_2},\lambda)\big]}{E_0^{\p^\lambda}[\Bd^\lambda(X_{\tau^-_v},v)]}\Big]
%	 &\qquad+2\bfE_0^\lambda\Big[%\one\{X_{\tau^-_{v}}\in\B^{(-1)}_v\}\!\!
%	 \sum_{i=\tau^1_{v}-m}^{n-1}\!\!L(\Bd^\lambda_{ie_1+me_2,(i+1)e_1+me_2},\lambda)\Big]
%	\end{split}
%	\label{var-id}\\
%	\begin{split}
	&=n(\psi_1(\lambda)-\psi_1(\alpha+\lambda))-m(\psi_1(\alpha+\lambda)-\psi_1(\alpha+\beta+\lambda))\\
	 %&\qquad+2\Eplus\Big[\frac{E_0^{\p^\lambda}\big[\one\{X_{\tau^-_{v}}\in\B^{(-1)}_v\}\Bd^\lambda(X_{\tau^-_v},v)\sum_{i=\tau^-_{v}-m}^{n-1}L(\Bd^\lambda_{ie_1+me_2,(i+1)e_1+me_2},\lambda)\big]}{E_0^{\p^\lambda}[\Bd^\lambda(X_{\tau^-_v},v)]}\Big]
	 &\qquad+2\wcE_v^\lambda\Big[%\one\{X_{\tau^+_0}\in\B^{(+1)}_0\} 
	 \sum_{i=0}^{X(\tau_0^+)\cdot e_1-1} \!\!\!\! L(\Bd^\lambda_{ie_1,(i+1)e_1},\lambda)\Big]
	\end{split}
	\label{var-id}\\
	\begin{split}
	&=m(\psi_1(\alpha+\lambda)-\psi_1(\alpha+\beta+\lambda))-n(\psi_1(\lambda)-\psi_1(\alpha+\lambda))\\
	&\qquad\qquad+2\wcE_v^{\lambda}\Big[\sum_{i=0}^{X(\tau_0^+)\cdot e_2-1}\widetilde L(1/\Bdla_{ie_2,(i+1)e_2},\lambda,\alpha)\Big]. %\textcolor{blue}{\ \ \ (CHECK!)}
	\end{split}\label{var-id-e2}
	\end{align}
\end{theorem}
%pdf of $\cWla$ is 
%	\[f_\lambda(s)=B(\lambda,\alpha)^{-n} e^{-\lambda s}\int_{\R_+^{n-1}}\Bigl((1-e^{-t_1})\dotsc(1-e^{-t_{n-1}})(1-e^{-s+\sum_{k=1}^{n-1}t_k})\Bigr)^{\alpha-1}\,dt_1\dotsc dt_{n-1}.\]
%Thus
%\begin{align*}
%\frac{\partial}{\partial\lambda}\Eplus[\cE_{\nu,lambda}]
%&=\frac{\partial}{\partial\lambda}\int_0^\infty\Eplus[\cE_{\nu,\lambda}\,|\,\cWla=s]\,f_\lambda(s)\,ds,\\
%&=\int_0^\infty\Eplus[\cE_{\nu,\lambda}\,|\,\cWla=s]\,\frac{\partial}{\partial\lambda}f_\lambda(s)\,ds,\\
%&=\int_0^\infty\Eplus[\cE_{\nu,\lambda}\,|\,\cWla=s](-s-n(\psi_0(\lambda)-\psi_0(\alpha+\lambda)))f_\lambda(s)\,ds,\\
%&=-\Eplus[\cE_{\nu,\lambda}\cWla]+\Eplus[\cWla]\Eplus[\cE_{\nu,\lambda}]\\
%&=-\cov(\cE_{\nu,\lambda},\cWla).
%\end{align*}
%But
%\begin{align*}
%\Eplus[\frac{\partial}{\partial\lambda}\cE_{\nu,lambda}]
%&=\Eplus\Bigl[\frac{\partial}{\partial\lambda}\log\Bd_{0,(m,n)}^{\nu,\lambda}\Bigr]\\
%&=\Eplus\Bigl[\frac{\sum_{k=1}^n\sum_{i=0}^{m+n-k-1}\frac{\frac{\partial}{\partial\lambda}\Bdla(ie_2,(i+1)e_2)}{\Bdla(ie_2,(i+1)e_2)}\Bdla(0,(n+m-k)e_2)P_v^\wch(\cdots)}{\Bd_{0,(m,n)}^{\nu,\lambda}}\Bigr]\\
%&=-\Eplus\Bigl[\frac{\sum_{k=1}^n\sum_{i=0}^{m+n-k-1}\widetilde L(1/\Bdla(ie_2,(i+1)e_2),\lambda,\alpha)\Bdla(0,(n+m-k)e_2)P_v^\wch(\cdots)}{\Bd_{0,(m,n)}^{\nu,\lambda}}\Bigr]\\
%&=-\wcE_v^\lambda[....].
%\end{align*}

An empty sum (e.g.\ $\sum_{i=0}^{-1}$) equals $0$. Thus, the  $\wcE_v^\lambda$  expectation on the right-hand side of  \eqref{var-id} is in fact
over the event $\{X(\tau^+_0)\in\B^{(+1)}_0\}$.    When $v$ is chosen (approximately) in the direction $\xi(\lambda)$ so that the first two terms on the right-hand side of \eqref{var-id} (approximately) cancel,  the equation expresses  the KPZ relation that in $1+1$ dimension the wandering exponent is twice the free energy exponent. 

Theorem \ref{var(Bd)} is the same as  Proposition 1.1 in \cite{Cha-Noa-17-},  via the connections \eqref{Z-Bd} and \eqref{Q-P} between the RWRE and the polymer.   Theorem \ref{var(Bd)} is also proved in Section 4.1 of the first preprint version  \cite{Bal-Ras-Sep-18-arxiv} of this paper.

Starting from  the  identity in Theorem \ref{var(Bd)}, a series of coupling arguments  and estimates leads to upper and lower  bounds on the fluctuations of $\log\Bd^\lambda$.   Theorem \ref{th:Bd-var} below follows directly from Theorem 1.2 of \cite{Cha-Noa-17-}. It is also proved in Sections 4.2 and 4.3 of the first preprint version  \cite{Bal-Ras-Sep-18-arxiv} of this paper.
%Theorems \ref{stat-upper} and \ref{stat-lower} stated in the beginning of Section \ref{sec:var}.  
Here, $\xi(\lambda)$ is given by \eqref{eq:lam-xi1}.  

\begin{theorem}\label{th:Bd-var}  Fix $\alpha,\beta>0$. Fix $\lambda>0$. 
Given a constant  $0<\gamma<\infty$, there exist  positive finite constants $c$, $C$, and $N_0$, depending only on $\alpha$, $\beta$, $\gamma$, and $\lambda$, 
such that 
	\begin{align}\label{var-Bd}  
	c   N^{2/3} \le \Vvv[\Bd^\lambda_{0,me_1+ne_2} ]\le C\, N^{2/3} 
	\end{align}
for all  $N\ge N_0$ and $(m,n)\in\N^2$ such that 
%\be\label{char-ass} 
\[\abs{m-N\xi_1(\lambda)} \vee\abs{n-N\xi_2(\lambda)} \le \gamma N^{2/3}.\]
The same constants $c$, $C$, and $N_0$ can be taken for $(\alpha,\beta,\gamma,\lambda)$ varying in a compact subset of $(0,\infty)^4$.
\end{theorem}

\begin{proof}[Proof of  Theorem \ref{th:B-var}]
By virtue of Theorem \ref{th:Buse}\eqref{B:B=Bd},  Theorem \ref{th:Bd-var} implies Theorem \ref{th:B-var} for the case $\xi_1\in(\llnv_1,1)$. The remaining case $\xi_1\in(0,\llnv_1)$ follows from the (omitted) 
version of Theorem \ref{th:Bd-var} for $\wt\Bd^\lambda$. 
\end{proof}

\section{Path fluctuations}\label{pf:main}

In this section we prove results about path fluctuations, from which Theorem \ref{th:kpz} will follow. 
For an up-right path $X_{0,\infty}$ started at the origin  and an integer $n\ge0$ let 
	\begin{align*}
	\Xmin_n=\min\{m\ge0:X_{m+n}\cdot e_2=n\}\quad\text{and}\quad
	\Xmax_n=\max\{m\ge0:X_{m+n}\cdot e_2=n\}.
	\end{align*}
Then $\Xmin_n e_1+ne_2$ and $\Xmax_n e_1+ne_2$ are, respectively, 
 the leftmost and rightmost points of the path on the horizontal line $ne_2+\Z_+e_1$. See the left panel in Figure \ref{fig:path}.
The vertical counterparts are given by
	\begin{align*}
	\Ymin_m=\min\{n\ge0:X_{m+n}\cdot e_1=m\}\quad\text{and}\quad
	\Ymax_m=\max\{n\ge0:X_{m+n}\cdot e_1=m\}.
	\end{align*}

 \begin{figure}[h]
 	\begin{center}
 		 \begin{tikzpicture}[>=latex, scale=0.7]
		 	\draw(0,0)--(0,6);
			\draw(0,0)--(8,0);
			%\draw[dashed](0,6)--(8,6)--(8,0);
			\draw[dashed](4,0)--(4,3)--(0,3);
			\draw(5,0)--(5,0.2);
			\draw(4,0)--(4,0.2);
			\draw(2,0)--(2,0.2);
			\filldraw[fill=light-gray,draw=black] (15.7,2.7) rectangle (16.3,3.3);
			\draw[line width=1pt, color=darkblue](0,0)--(8,6);
			\draw [line width= 2pt, color=nicosred](0,0)--(0,1)--(2,1)--(2,3)--(5,3)--(5,5)--(8,5);
			\draw[nicosred,fill=nicosred](2,3) circle(6pt);
			\draw[nicosred,fill=nicosred](5,3) circle(6pt);
			\draw[darkblue,fill=darkblue](4,3) circle(6pt);
			\draw (2,0)node[below]{\tiny$\Xmin_{n}$};
			\draw (5.5,0)node[below]{\tiny$\Xmax_{n}$};
			\draw (4,0)node[below]{\tiny$m$};
			%\draw (8,0)node[below]{\tiny$m$};
			\draw (0,3)node[left]{\tiny$n$};
			%\draw (0,6)node[left]{\tiny$n$};
			%\draw[<->](2,3.5)--(4,3.5);
			%\draw (3.3,3.5)node[above]{\tiny$\le rN^{2/3}$};
			%\draw[<->](4,2.5)--(5,2.5);
			%\draw (5.3,2.5)node[below]{\tiny$\le rN^{2/3}$};

		 	\draw(12,0)--(12,6);
			\draw(12,0)--(20,0);
			\draw[dotted](16,0)--(16,6);
			\draw[dotted](12,3)--(20,3);
			\draw[line width=1pt, color=darkblue](12,0)--(20,6);
			\draw (16,0)node[below]{\tiny$m$};
			\draw (12,3)node[left]{\tiny$n$};
			\draw[line width=1.5pt](16-1,3)--(16+1,3);
			\draw[line width=1.5pt](16,3-1)--(16,3+1);

			\draw [line width=1pt,color=nicosred](12,0)--(12,0.5)--(13,0.5)--(13,1)--(14,1)--(14,1.5)--(14.5,1.5)--(14.5,2)--(15.3,2)--(15.3,3)--(15.5,3)--(15.5,3.4)--(16,3.4)--(16,3.8)--(16.5,3.8)--(16.5,4.2)--(17.5,4.2)--(17.5,4.5)--(18,4.5)--(18,5)--(19.5,5)--(19.5,5.5)--(20,5.5);
			
		\end{tikzpicture}
 	 \end{center}
 	\caption{\small In both plots the diagonal line points in direction $\xi(\lambda)$. Left: the definition of $\Xmin_n$ and $\Xmax_n$. 
	Right: illustration of \eqref{path+bd:1} and \eqref{path+bd:4}. 
	The four arms of the cross centered at $(m,n)=\fl{N\xi}$ are of length $rN^{2/3}$ each. The shaded box, also centered at $(m,n)$, has sides of length $2\delta N^{2/3}$. For large $r$, the path has a high probability of 
	entering and exiting through the cross and never touching the dotted lines. For small $\delta$, there is a positive probability, uniformly in $N$, that the path stays left of the top edge of the shaded box, completely avoiding the  box.}
 \label{fig:path}
 \end{figure}
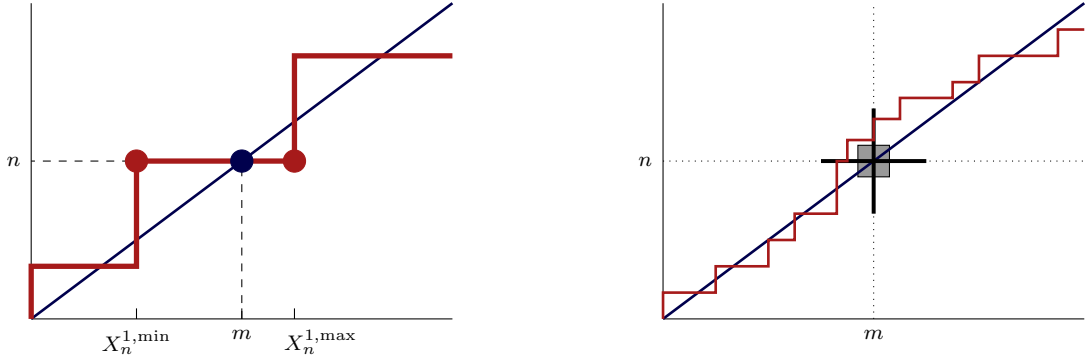

Again, the next result is stated and proved for the case $\xi_1\in(\llnv_1,1)$ only. The other case works similarly. 
Recall  $\xi(\lambda)$ from \eqref{eq:lam-xi1}.
Let $\bfP_0^\lambda=\Eplus P_0^{\pi^\lambda}$, with expectation $\bfE_0^\lambda=\Eplus E_0^{\pi^\lambda}$.
By Theorem \ref{thm:LLN}, $\xi(\lambda)$ is the LLN direction for $\bfP^\lambda_0$. 
 
\begin{theorem}\label{th:mainaux}
Fix $\alpha,\beta,\lambda>0$. %Let $\lambda_2>\lambda_1>0$. 
\begin{enumerate}[label={\rm(\alph*)}, ref={\rm\alph*}] %[\ \ \rm(a)]
\item {\rm Upper bound.} There exist finite positive constants $r_0$ and $C$ depending on $\alpha$, $\beta$, and $\lambda$, such that 
for all  %$\lambda\in[\lambda_1,\lambda_2]$, 
$r\ge r_0$,   integer $N\ge1$,  and $(m,n)=\fl{N\xi(\lambda)}$, we have
	\begin{align}
	\bfP^\lambda_0\big\{\Xmin_{n}<m-rN^{2/3}\big\}\le Cr^{-3}\quad\text{and}\quad\bfP^\lambda_0\big\{\Xmax_{n}>m+rN^{2/3}\big\}\le Cr^{-3}.\label{path+bd:1}
	\end{align}
From this it follows that
	\begin{align}
	\begin{split}
	&\bfE^\lambda_0[\abs{(m-\Xmin_n)^+}^p]^{1/p}\le\Big(1+\frac{Cp}{3-p}\Big)^{1/p}N^{2/3}\quad\text{and}\\
	&\bfE^\lambda_0[\abs{(\Xmax_n-m)^+}^p]^{1/p}\le\Big(1+\frac{Cp}{3-p}\Big)^{1/p}N^{2/3}.
	\end{split}
	\label{path+bd:5}
	\end{align}
\item {\rm Lower bound.}  There exist finite positive constants $\delta$ and  $c$ depending on  $\alpha$, $\beta$, and $\lambda$, such that 
for any %$\lambda\in[\lambda_1,\lambda_2]$ and  
integer $N\ge1$ such that $(m,n)=\fl{N\xi(\lambda)}\in\N^2$ 
we have
	\begin{align}
	&\bfE^\lambda_0[(m-\Xmin_n)^+]\ge c N^{2/3}\quad\text{and}\label{path+bd:3a}\\
	&\bfP^\lambda_0\big\{\Xmin_n\le\Xmax_n<m-\delta N^{2/3}\big\}\ge c.\label{path+bd:4}
	\end{align}
\end{enumerate}
Similar bounds hold for the vertical counterparts $\Ymin_m$ and $\Ymax_m$. The same constants can be taken for all $(\alpha,\beta,\lambda)$ varying in a compact subset of $(0,\infty)^3$.
\end{theorem}

%The following proof refers to Lemmas 4.2 and 4.7 of \cite{Cha-Noa-17-} at some point. A self-contained proof appears in \cite{Bal-Ras-Sep-18-arxiv}.

\begin{proof}%[Proof of Theorem \ref{th:mainaux}]
Abbreviate  $u=(m,n)=\fl{N\xi(\lambda)}$. 
Inequality \eqref{path+bd:1} is trivial if $rN^{1/3}\ge m$. We hence assume that $rN^{2/3}<m$. 

Note that 
	\begin{align}\label{Xmin=exit}
	(m-\Xmin_n)^+=m-X(\tau^-_{(m,n)})\cdot e_1. 
	\end{align}
Thus, the first probability in \eqref{path+bd:1} equals
	\begin{align}\label{abc}
	\bfP^\lambda_0\big\{X(\tau_u^-)\cdot e_1<m-rN^{2/3}\big\}=\wcP^\lambda_u\big\{X(\tau_0^+)\cdot e_1>rN^{2/3}\big\}.
	\end{align}
%
%
%Recall constant $\Cref{c2}$ from Theorem \ref{th4.7}. Restrictions $r\ge1\vee\Cref{c2}$ and $N\ge1$ guarantee that $1\vee\Cref{c2}\le rN^{2/3}\le m\le 2N$.
%Note that $(m,n)$ satisfy \eqref{mn-cond} with $K_N=1$.
%Apply \eqref{exit-tail}  with $\Nref{N0}=\aref{a:mn}=1$ and $\Aref{a1}=2$ to get
%	\begin{align*}
%	\wcP^\lambda_u\big\{X(\tau_0^+)\cdot e_1>rN^{2/3}\big\}
%	&\le\Cref{c3}(r^{-3}+2r^{-4})+e^{-\cref{cr}r^2N^{1/3}}\\
%	&\le\Cref{c3}(r^{-3}+2r^{-4})+e^{-\cref{cr}r^2}\le\Cref{C:path}r^{-3}.
%	\end{align*}
Applying Lemma 4.7 of \cite{Cha-Noa-17-} and the connection \eqref{Q-P}  gives 
	\[\wcP^\lambda_u\big\{X(\tau_0^+)\cdot e_1>rN^{2/3}\big\}\le Cr^{-3}.\]
(This is also (4.24) in \cite{Bal-Ras-Sep-18-arxiv}.) 
This proves the first inequality in \eqref{path+bd:1}. 

To prove the second inequality set $N_0=\fl{\frac{m+rN^{2/3}}{\xi_1(\lambda)}}$ and $(m_0,n_0)=\fl{N_0\xi(\lambda)}$.
Then $m_0\le m+rN^{2/3}$ and therefore if $\Xmax_n>m+rN^{2/3}$, then $\Ymin_{m_0}\le n$.
But we also have 
	\begin{align*}
	n_0
	&>N_0\xi_2(\lambda)-1\ge\frac{m\xi_2(\lambda)}{\xi_1(\lambda)}+\frac{\xi_2(\lambda)}{\xi_1(\lambda)}rN^{2/3}-1-\xi_2(\lambda)\\
	&\ge n+\frac{\xi_2(\lambda)}{\xi_1(\lambda)}rN^{2/3}-1-\xi_2(\lambda)-\frac{\xi_2(\lambda)}{\xi_1(\lambda)}\\
	&\ge n+\frac{\xi_2(\lambda)}{2\xi_1(\lambda)}rN^{2/3}\ge n+\frac{\xi_2(\lambda_1)}{2\xi_1(\lambda_2)}rN^{2/3},
	\end{align*}
provided $r\ge 2(1+\xi_1(\lambda_1)+\frac{\xi_1(\lambda_1)}{\xi_2(\lambda_2)})$.
The upshot is that if $\Xmax_n>m+rN^{2/3}$ then $\Ymin_{m_0}<n_0-\frac{\xi_2(\lambda_1)}{2\xi_1(\lambda_2)}rN_0^{2/3}$.
The second inequality in \eqref{path+bd:1} thus follows from the vertical version of the first inequality, but with $N_0$ and $r_0=\frac{\xi_2(\lambda_1)}{2\xi_1(\lambda_2)}r$
playing the roles of $N$ and $r$, respectively.

Bounds \eqref{path+bd:5} follow from \eqref{path+bd:1}. For example, for the first bound abbreviate $Y=(m-\Xmin_n)^+$ and write
	\begin{align*}
	\bfE^\lambda_0[(N^{-2/3}Y)^p]
	&=\int_0^\infty p\,r^{p-1}\,\bfP^\lambda_0(Y>rN^{2/3})\,dr\\
	&\le\int_0^1 p\,r^{p-1}\,dr+C\int_1^\infty p\,r^{p-4}\,dp=1+\frac{Cp}{3-p}\,.
	\end{align*}

Next, applying Lemma 4.2 of \cite{Cha-Noa-17-} we have
	\begin{align}\label{var-id2}
	\wcE_v^\lambda\Big[\sum_{i=0}^{X(\tau_0^+)\cdot e_1-1} \!\!\!\! L(\Bd^\lambda_{ie_1,(i+1)e_1},\lambda)\Big]\le C\big(\wcE^{\lambda}_v[X(\tau_0^+)\cdot e_1]+1\big).
	\end{align}
(This is also (4.15) in \cite{Bal-Ras-Sep-18-arxiv}.)
Now, bound \eqref{path+bd:3a} follows from stringing together \eqref{var-id2}, \eqref{var-id}, and the lower bound in \eqref{var-Bd}, then reversing the picture in \eqref{Xmin=exit}.
To get \eqref{path+bd:4} first write
	\begin{align*}
	c N^{2/3} 
	&\le  \bfE^\lambda_0[Y]  =  \bfE^\lambda_0[Y\one\{Y\le \delta N^{2/3}\}]  +   \bfE^\lambda_0[Y\one\{Y> \delta N^{2/3}\}]\\
	&\le      \delta N^{2/3}  +    \bfE^\lambda_0[Y^2]^{1/2}\,  \bfP^\lambda_0( Y> \delta N^{2/3} )^{1/2}.
	\end{align*}
Applying \eqref{path+bd:5} with say $p=2$ and taking $\delta\le c/2$ we get
	\begin{align}\label{bd4-prelim}
	\bfP^\lambda_0\big\{\Xmin_n<m-\delta N^{2/3}\big\}\ge \frac{c}{2\sqrt{1+2C}}\,.
	\end{align}

Now take $\delta_0>2\delta$, $N_0=N+\fl{\delta N^{2/3}}$, and $(m_0,n_0)=\fl{N_0\xi(\lambda)}$.
Note that 	$m_0\le\fl{N\xi_1(\lambda)}+\delta N^{2/3}=m+\delta N^{2/3}$.
This forces
	\[m_0-\delta_0 N_0^{2/3}\le m+\delta N^{2/3}-2\delta N^{2/3}=m-\delta N^{2/3}.\]
Since $n\le n_0$ we have that if $\Xmin_{n_0}<m_0-\delta_0 N_0^{2/3}$, then 
	\[\Xmax_n\le \Xmin_{n_0}<m_0-\delta_0N_0^{2/3}\le m-\delta N^{2/3}.\]
Bound \eqref{path+bd:4}  follows from the above and \eqref{bd4-prelim} with $N_0$ and $\delta_0$ playing the roles of $N$ and $\delta$, respectively.
\end{proof}

 \begin{proof}[Proof of Theorem \ref{th:kpz}]
We only argue for the case $\xi_1\in(\llnv_1,1)$, the other case being similar.

By Theorem \ref{th:Buse}\eqref{B:B=Bd} the distribution of $P^{\doob^\xi}_0$ under $\P$ is the same as that of $P^{\pi^\lambda}_0$ under $\Pplus$, provided $\lambda$ and $\xi$ are put in duality via \eqref{eq:lam-xi1}. Hence, 
$\bfP^\xi_0=\bfP^\lambda_0$.  Now, the claims of the theorem follow from \eqref{path+bd:1} and \eqref{path+bd:4}. See Figure \ref{fig:path}.  
%
%\eqref{path+bd:1} says that with $\bfP^\lambda_0$-probability at least $1-\Cref{C:path}r^{-3}$ the RWRE path  enters $\B_{\fl{N\xi}}^-$ within $rN^{2/3}$ of $N\xi$. This forces $X_N$ to be within $r\sqrt2 N^{2/3}$ of $N\xi$. See the left panel in Figure \ref{fig:rwre}.
%Bound \eqref{KPZrwre1} follows.
%
%%For \eqref{KPZrwre3}, fix $\Cgen>0$ and for $0\le a\le \Cgen \sqrt{2} N^{2/3}$ let $b=\Cgen N^{2/3}-a/\sqrt{2}$. Let $M=\fl{N-\Cgen N^{2/3}/(\xi\cdot e_2)}$.
%%Observe that if 
%%	\[\Xmax_{\fl{M\xi\cdot e_2}}-M\xi\cdot e_1\ge \Cgen \frac{\xi\cdot e_1}{\xi\cdot e_2}N^{2/3}+\frac{a}{\xi\cdot e_2\sqrt2}\]
%%then $\abs{X_N-N\xi}_2\ge a$. (See Figure \ref{fig:rwre}.)
%
%Let $M=\fl{N+\delta N^{2/3}}$. Observe that if $M\xi_1-\Xmin_{\fl{M\xi_2}}\ge\delta N^{2/3}$ then 
%$\abs{X_N-N\xi}_2\ge\sqrt2\delta\xi_2 N^{2/3}$. See Figure \ref{fig:rwre}. Hence,
%	\begin{align*}
%	\cref{c:path3}&\le\bfP_0^\lambda\big\{M\xi_1-\Xmin_{\fl{M\xi_2}}\ge\delta N^{2/3}\big\}\\
%	&\le \bfP_0^\lambda\big\{\abs{X_N-N\xi}_2\ge\sqrt2\delta\xi_2 N^{2/3}\big\}\\
%	&\le\frac{\bfE_0^\lambda[\abs{X_N-N\xi}_2]}{\sqrt2\delta\xi_2 N^{2/3}}\,.
%	\end{align*}
%\eqref{KPZrwre3} now follows.
\end{proof}

 \section{Proofs of the large deviation results}\label{Iq-pf}

%\note{Does this section need warnings about only covering one of the cases?  F: No, we either compute both or mention inside the proof that we deal with one, the other being similar.} 

\begin{proof}[Proof of Lemma \ref{lam-xi-t}]
Define the function 
	\[f(\lambda)=\frac{\psi_1(\lambda)-\psi_1(\alpha+\lambda)}{\psi_1(\lambda)-\psi_1(\alpha+\beta+\lambda)}.\]
We  prove that $f$ is strictly decreasing in $\lambda>0$.
Its derivative is
	\begin{align*}
	f'(\lambda)&=
	\frac{(\psi_2(\lambda)-\psi_2(\alpha+\lambda))}%(\psi_1(\lambda)-\psi_1(\alpha+\beta+\lambda))}
	{(\psi_1(\lambda)-\psi_1(\alpha+\beta+\lambda))}\\%^2}\\
	&\qquad\qquad-
	\frac{(\psi_1(\lambda)-\psi_1(\alpha+\lambda))(\psi_2(\lambda)-\psi_2(\alpha+\beta+\lambda))}
	{(\psi_1(\lambda)-\psi_1(\alpha+\beta+\lambda))^2}.
	\end{align*}
Since $\psi_1$ is strictly decreasing, $f'(\lambda)<0$ is equivalent to
	\begin{align}\label{nice}
	\frac{\psi_2(\lambda)-\psi_2(\alpha+\lambda)}{\psi_1(\lambda)-\psi_1(\alpha+\lambda)}<
	\frac{\psi_2(\lambda)-\psi_2(\alpha+\beta+\lambda)}{\psi_1(\lambda)-\psi_1(\alpha+\beta+\lambda)}\,.
	\end{align}
This in turn follows from $\psi_2\circ\psi_1^{-1}$ being strictly concave, which is proved in 
Lemma \ref{psi-con}.

We have so far shown that $f$ is strictly decreasing. Since $\psi_1(\lambda)\to\infty$ as $\lambda\searrow0$ we have
$f(\lambda)\to1$ as $\lambda\searrow0$.  Similarly, by Lemma \ref{lm:psi-exp} we have
$\lambda^2(\psi_1(\lambda)-\psi_1(a+\lambda))\to a$ as $\lambda\to\infty$ and thus $f(\lambda)\to\frac{\alpha}{\alpha+\beta}$ as
$\lambda\to\infty$. The claims in part \eqref{lam-xi} for $\xi_1\in[\llnv_1,1]$ now follow. The case $\xi_1\in[0,\llnv_1]$ comes by interchanging the roles of $\alpha$ and $\beta$ and those of $\xi_1$ and $\xi_2$.

Define the function
	\[g(\lambda)= \psi_0(\alpha+\beta+\lambda)-\psi_0(\lambda).\]
Since $\psi_1$ is strictly decreasing we see that
	\[g'(\lambda)=\psi_1(\alpha+\beta+\lambda)-\psi_1(\lambda)<0.\]
Hence, $g$ is strictly decreasing. As $\lambda\searrow0$ we have $\psi_0(\lambda)\to-\infty$ and $g(\lambda)\to\infty$.
Representation \eqref{psi0} gives 
	\[g(\lambda)=-\frac1{\alpha+\beta+\lambda}+\frac1\lambda-\sum_{k=1}^\infty\Big(\frac1{\alpha+\beta+\lambda+k}-\frac1{\lambda+k}\Big).\]
Then we see that as $\lambda\to\infty$, $g(\lambda)\to0$.
Part \eqref{lam-t}  follows and  Lemma \ref{lam-xi-t} is proved.
\end{proof}

\begin{proof}[Proof of Theorems \ref{thm:Iq} and \ref{thm:Iq2}]
We utilize the stationary ratios $\Bd^\lambda$ and transitions $\p^\lambda$ from Section \ref{sec:bdry}.   By  Proposition \ref{stat-pr},      $\p^\lambda$ under $\Pplus$ (defined on page \pageref{Pbar}) has the same distribution as the original environment  $\p$ under $\P$. 

By the ergodic theorem  
\begin{align}\label{aux1}
n^{-1}\log\Bd_{0,ne_2}^\lambda&=n^{-1}\sum_{i=0}^{n-1}\log\Bd_{ie_2,(i+1)e_2}^\lambda\mathop{\longrightarrow}_{n\to\infty}\Eplus[\log\Bd_{0,e_2}^\lambda] =  \psi_0(\alpha+\lambda)-\psi_0(\lambda).  
\end{align}
(Recall that the logarithm of a Gamma($\nu,1$) has expected value $\psi_0(\nu)$ and that a
Beta($a,b$) is a ratio of a Gamma($a,1$) and a Gamma($a+b,1$). By Proposition \ref{stat-pr},  $\Bd_{ie_2,(i+1)e_2}^{-1}$ are i.i.d.\  Beta($\lambda, \alpha$).)
%  (The probability space changes with $n$ in the limit above because the boundary is reconstructed for each $n$.  So it is not a.s.\ but an in probability type of limit.) 

On the other hand, harmonicity of $\Bd_{\cbullet, ne_2}$ implies, similarly  to \eqref{harm76}, 
%an argument similar to the one leading to \eqref{harm76}
%using the path $y_i=ne_2+i(e_1-e_2)$, $0\le i\le n$, and $v=ne_2$,  Lemma \ref{stat-RWRE} 
 that for any $x\in\Z_+^2$ with $\abs{x}_1\le n$
\begin{align}\label{aux3}
\Bd_{x,ne_2}^\lambda=\sum_{j=0}^n P^{\p^\lambda}_x\{X_{n-\abs{x}_1}=je_1+(n-j)e_2\} \Bd_{je_1+(n-j)e_2,ne_2}^\lambda. 
\end{align}
%where $P_x^{\p^\lambda}$ denotes the quenched RWRE starting at $x$ and using transitions $\p^\lambda$. 
The sum above  has nonzero terms exactly when $j$ is between $x\cdot e_1$ and $n-x\cdot e_2$.  
Abbreviate 
	\[R_{i,n}=\Bd^\lambda_{(i+1)e_1+(n-i-1)e_2,\, ie_1+(n-i)e_2}.\] 
For fixed $n$,  under $\Pplus$,  variables $R_{i,n}$ are i.i.d.\ and each distributed as $\Bd_{e_1,e_2}^\lambda$.   Rewrite  \eqref{aux3}  for $x=0$   as  
\begin{align}\label{aux111}
\Bd_{0,ne_2}^\lambda=\sum_{j=0}^n P^{\p^\lambda}_0\{X_n=je_1+(n-j)e_2\} \prod_{i=0}^{j-1} R_{i,n}.
\end{align}

Assuming the usual asymptotics and approximations work, 
\begin{align*}
n^{-1}\log\Bd_{0,ne_2}^\lambda&\approx \max_{0\le j\le n} \biggl\{ n^{-1}\log  P^{\p^\lambda}_0\{X_n=je_1+(n-j)e_2\}  
+  n^{-1}\sum_{i=0}^{j-1}\log R_{i,n} \biggr\} \\
&= \sup_{\xi\in\Uset} \biggl\{ n^{-1}\log  P^{\p^\lambda}_0\{X_n=[n\xi]\}  
+  n^{-1}\sum_{i=0}^{[n\xi]\cdot e_1-1}\log R_{i,n} \biggr\} \\
&\underset{n\to\infty}\longrightarrow  \; \sup_{\xi\in\Uset} \bigl\{ -I_q(\xi) + \xi\cdot e_1 \bigl(\psi_0(\alpha+\beta+\lambda)-\psi_0(\lambda)\bigr)\bigr\}. 
\end{align*}
We defer the detailed justification of this limit  to the end of the proof.

The above and \eqref{aux1} give the equation 
 \begin{align}\label{I-97} 
\psi_0(\alpha+\lambda)-\psi_0(\lambda)= 
 \sup_{\xi\in\Uset} \bigl\{  \xi_1 \bigl(\psi_0(\alpha+\beta+\lambda)-\psi_0(\lambda)\bigr) -I_q(\xi)\bigr\}. 
 \end{align} 
 
 For $t\in\R$ let 
 \[  f(t)=I_q^*(te_1)=\sup_{s\in\R} \{ ts-I_q(se_1+(1-s)e_2)\}  \]
 where of course $I_q(\xi)=\infty$ for $\xi\notin\Uset$ (i.e.\ $s\not\in[0,1]$).  
% Define a function $\lambda\mapsto t $ on $(0,\infty)$  by  
% \[  t=t(\lambda)= \psi_0(\alpha+\beta+\lambda)-\psi_0(\lambda). \]
% Since $\psi_0(0+)=-\infty$, 
% \begin{align}\label{psi_0-7} \psi_0(s)=\log s+\frac1{2s}+O(s^{-2})\quad\text{ as } \ s\to\infty, \end{align}
%\textcolor{blue}{this should be checked}   and 
% \[  \frac{dt}{d\lambda}=\psi_1(\alpha+\beta+\lambda)-\psi_1(\lambda)<0, \]
% we see that  $\lambda\mapsto t $ is a strictly decreasing bijection   on $(0,\infty)$.   In particular,  $t\to 0$ iff $\lambda\to\infty$.  
For $t\ge0$ and $\lambda(t)$ defined by Lemma \ref{lam-xi-t}\eqref{lam-t} equation \eqref{I-97} gives 
\[   f(t)=\psi_0(\alpha+\lambda(t))-\psi_0(\lambda(t)).  \]
This proves \eqref{Iq*1}.

We have
\begin{align*}
f'(t)&=\bigl( \psi_1(\alpha+\lambda(t))-\psi_1(\lambda(t))\bigr)  \lambda'(t) \\
&=\frac{\psi_1(\alpha+\lambda(t))-\psi_1(\lambda(t))}{\psi_1(\alpha+\beta+\lambda(t))-\psi_1(\lambda(t))}\ \underset{t\searrow0}\longrightarrow \   \frac{\alpha}{\alpha+\beta} 
\end{align*}
where the last limit has already been shown at the end of the proof of Lemma \ref{lam-xi-t}\eqref{lam-xi}. 
Consequently, $f'(0+)=\frac{\alpha}{\alpha+\beta}$. 
%, as one would expect since by the LLN $X_n\cdot e_1/n\to \frac{\alpha}{\alpha+\beta} $.   
Since $f$ is convex, we get that 
	\begin{align}
	f'(t\pm)\le \frac{\alpha}{\alpha+\beta}\quad\text{for $t\le 0$.}\label{aux101}
	\end{align}

Since the RWRE under the averaged measure $\int P_0^{\p^\lambda}(\cdot) \,\Pplus(d\wplus)$ is  simple random walk,   RWRE
with transitions $\p^\lambda$ satisfies an almost-sure law of large numbers with velocity given by 
	\[\Eplus[\p^\lambda_{0,e_1}e_1+\p^\lambda_{0,e_2}e_2]=(\tfrac{\alpha}{\alpha+\beta}, \tfrac{\beta}{\alpha+\beta})=\llnv.\]
This gives  $I_q(\llnv)=0$.  
%Since for $a>0$, $\psi_0(\lambda)-\psi_0(a+\lambda)\to0$ as $\lambda\to\infty$ (e.g.\ by Lemma \ref{lam-xi-t}\eqref{lam-t}),
%\eqref{Iq} holds for $\xi=\llnv$.

Let  $\xi\in\Uset$ with $\xi_1\ge\frac{\alpha}{\alpha+\beta}$. 
The second equality in the next computation comes from \eqref{aux101}. 
\begin{align}
I_q(\xi)&=\sup_{t\in\R} \{ t\xi_1-f(t)\}=  \sup_{t>0} \{ t\xi_1-f(t)\}\notag\\
&=\sup_{\lambda>0}  
  \bigl\{  \xi_1 \bigl(\psi_0(\alpha+\beta+\lambda)-\psi_0(\lambda)
 \bigr) - \psi_0(\alpha+\lambda) +\psi_0(\lambda) \bigr\}  \label{I-var}\\
 &= \xi_1\psi_0\bigl(\alpha+\beta+\lambda(\xi_1)\bigr)+(1-\xi_1)\psi_0\bigl(\lambda(\xi_1)\bigr)  -\psi_0\bigl(\alpha+\lambda(\xi_1)\bigr)\label{I-var2}
\end{align}
because condition \eqref{eq:lam-xi1} picks out the maximizer above.

To derive  $I_q(\xi)$ for $\xi_1\in[0,\frac{\alpha}{\alpha+\beta}]$, switch around $\alpha$ and $\beta$ and  the axes and then apply the first formula of \eqref{Iq} already proved.   

%To compute the rate $I_q(\xi)$ for $\xi_1\in[0,\frac{\alpha}{\alpha+\beta}]$ we use symmetry.
%To this end,
%let $\{1/\Bd^\lambda_{x,\,x+e_1}:x\in\Z_+e_1\}$ be i.i.d.\ Beta$(\lambda,\beta)$, 
%$\{\Bd^\lambda_{x,\,x+e_2}:x\in\Z_+e_2\}$  i.i.d.\ Beta$(\beta+\lambda,\alpha)$, and the two collections  independent of each other and of $\{\pch_{x,\,x-e_1}:x\in\N^2\}$.
%Define $\Bd^\lambda_{x,y}$, $x,y\in\Z_+^2$ by induction, using Lemma \ref{ind-sol} with now 
%\[U^{-1}\sim\text{\rm Beta}(\lambda,\beta), \quad V\sim\text{\rm Beta}(\beta+\lambda,\alpha), \quad\text{and}\quad \Beta\sim\text{\rm Beta}(\alpha,\beta).\]
%That is, we switched the roles of parameters $\alpha$ and $\beta$ and of the axes for the boundary weights, but not for the bulk weights $\Beta$.
%
%This gives a stationary boundary system similar to the one we defined in Section \ref{sec:bdry}, but now $\lambda$ is in bijection with $\xi\in\Uset$ such that $\xi_1\in[0,\frac{\alpha}{\alpha+\beta}]$ as in \eqref{eq:lam-xi2}.
%The second formula in \eqref{Iq} follows.

To compute $I_q^*(te_1)$ for $t<0$   write temporarily  $f_{\alpha,\beta}(t)$ and $I_{\alpha,\beta}(\xi)$ 
to make the dependence on the parameters $\alpha, \beta$ explicit.  Then  
	\begin{align*}
	f_{\alpha,\beta}(t)&=t+\sup_{0\le s\le1}\{(-t)(1-s)-I_{\alpha,\beta}(se_1+(1-s)e_2)\}\\
	&=t+\sup_{0\le s\le1}\{(-t)(1-s)-I_{\beta,\alpha}((1-s)e_1+se_2)\}
	%\\  &
	=t+f_{\beta,\alpha}(-t).
	\end{align*}
Formula \eqref{Iq*2} follows.
In particular, we have $f'_{\alpha,\beta}(0-)=1-\frac{\beta}{\alpha+\beta}=f'_{\alpha,\beta}(0+)$ and 
$f_{\alpha,\beta}$ is everywhere differentiable. Thus, $I_q=I_{\alpha,\beta}$ is strictly convex on $\Uset$.

We have now verified formula \eqref{Iq} for $I_q$ and Theorem \ref{thm:Iq2}.   
By Lemma 8.1 of \cite{Ras-Sep-Yil-17-ejp} the statement $I_q(\xi)>I_a(\xi)$ $\forall\xi\in\Uset\setminus\{\llnv\}$ is equivalent to 
	\[I^*_q(t)<I^*_a(t)\quad\text{for all $t\ne0$}.\]
(The case $t=0$ corresponds to $\xi=\llnv$ and thus leads to an equality.)

Substituting the above functions this becomes
	\[\psi_0(\alpha+\lambda(t))-\psi_0(\lambda(t))<\log(\llnv_1 e^t+\llnv_2)\]
and
	\[-t+\psi_0(\beta+\lambda(t))-\psi_0(\lambda(t))<\log(\llnv_1 e^{-t}+\llnv_2)\]
for all $t>0$.

Using \eqref{eq:lam-t} and a little bit of rearrangement the above is equivalent to showing that
%	\[\psi_0(\alpha+\lambda)-\psi_0(\lambda)<\log\Bigl(\llnv_1 e^{\psi_0(\alpha+\beta+\lambda)-\psi_0(\lambda)}+\llnv_2\Bigr)\]
%and
%	\[\psi_0(\beta+\lambda)-\psi_0(\alpha+\beta+\lambda)<\log\Bigl(\llnv_1 e^{\psi_0(\lambda)-\psi_0(\alpha+\beta+\lambda)}+\llnv_2\Bigr)\]
%for all $\lambda\ge0$. 
%
%Taking exponentials and multiplying through by $e^{\psi_0(\lambda)}$ and $e^{\psi_0(\alpha+\beta+\lambda)}$, respectively, this turns to
	\[e^{\psi_0(\alpha+\lambda)}<\llnv_1 e^{\psi_0(\alpha+\beta+\lambda)}+\llnv_2e^{\psi_0(\lambda)}\]
and
	\[e^{\psi_0(\beta+\lambda)}<\llnv_1 e^{\psi_0(\lambda)}+\llnv_2 e^{\psi_0(\alpha+\beta+\lambda)}\]
for all $\lambda\ge0$. 

Since $\llnv_1(\alpha+\beta+\lambda)+\llnv_2\lambda=\alpha+\lambda$ and $\llnv_1\lambda+\llnv_2(\alpha+\beta+\lambda)=\beta+\lambda$, 
the above inequalities would follow if  $e^{\psi_0(x)}$ were a strictly convex function. The second derivative of this function is given by 
$e^{\psi_0(x)}(\psi_2(x)+\psi_1(x)^2)$,
which is positive by Lemma \ref{psi-diff}. We have hence shown that $I_q(\xi)>I_a(\xi)$ for all $\xi\in\Uset$ with $\xi\ne\llnv$.

The proofs of Theorems \ref{thm:Iq} and \ref{thm:Iq2} are complete, except that 
it remains  to give the detailed justification of  \eqref{I-97}.  By \eqref{aux111} for any $\xi\in\Uset$,  
\[n^{-1}\log\Bd_{0,ne_2}^\lambda\ge n^{-1}\log P_0^{\p^\lambda}\{X_n=[n\xi]\}+n^{-1}\sum_{i=0}^{[n\xi]\cdot e_1-1}\log R_{i,n}.\]
%Recall that $\{\p^\lambda_{x,\,x+e_1}:x\in\Z_+^2\}$ are i.i.d.\ Beta($\alpha,\beta$) random variables. 
%In other words, $\p^\lambda$ transitions have distribution $\P$.
Taking $n\to\infty$ and applying \eqref{aux1}, \eqref{qLDP}, and the ergodic theorem we get
\begin{align}
\psi_0(\alpha+\lambda)-\psi_0(\lambda)
%\varliminf_{n\to\infty} n^{-1}\log\Bdla(0,ne_2)
&\ge -I_q(\xi)+\xi_1\E[\log\Bd_{e_1,e_2}^\lambda]\label{aux99}\\
&=-I_q(\xi)+\xi_1(\E[\log\Bd_{0,e_2}^\lambda]-\E[\log\Bd^\lambda_{0,e_1}])\notag\\
&=-I_q(\xi)+\xi_1(\psi_0(\alpha+\beta+\lambda)-\psi_0(\lambda)).\notag
\end{align}
(Since the summands $\log R_{i,n}$ shift with $n$ the limit of their average is not a.s.\ but rather an in probability limit.) 
Supremum over $\xi$ gives
	\[\psi_0(\alpha+\lambda)-\psi_0(\lambda)\ge\max_{\xi\in\Uset}\big\{-I_q(\xi)+\xi_1(\psi_0(\alpha+\beta+\lambda)-\psi_0(\lambda))\big\}.\]

For the reverse inequality go back to \eqref{aux3} and write
	\begin{align}\label{upper}
	\begin{split}
	n^{-1}\log \Bd_{0,ne_2}^\lambda
	\le n^{-1}\log(n+1)+n^{-1}\max_{0\le j\le n} \Big\{&\log P^{\p^\lambda}_0\{X_n=je_1+(n-j)e_2\}\\
	&+n^{-1}\log\Bd_{je_1+(n-j)e_2,ne_2}^\lambda\Big\}.
	\end{split}
	\end{align}

Now take %$\e\in(0,1)$, $k>3/\e$, and 
$n\ge k\ge 2$. Let $m_n$ be the integer such that 
	\[(m_n-1)(k-1)<n\le m_n(k-1).\]
Then for any $j$ with $0\le j\le n$ we have $j/m_n+1\le k$ and thus there exists a unique $i=i(j,n)$ such that $0\le i\le k$ and $j< m_n i\le j+m_n$.
Let $x_{j,n}=ie_1+(k-i)e_2$.

For $z\in\{e_1,e_2\}$ abbreviate
	\[A_{x,z}=\max\big\{\abs{\log\p^\lambda(x,x+z)},\abs{\log\Bd^\lambda_{x,\,x+z}}\big\}.\]
For $0\le j\le n$ use the Markov property to bound
	\[P_0^{\p^\lambda}\{X_n=je_1+(n-j)e_2\}P_{je_1+(n-j)e_2}^{\p^\lambda}\{X_{m_nk-n}=m_n x_{j,n}\}
	\le P_0^{\p^\lambda}\{X_{m_nk}=m_n x_{j,n}\}.\]
On the other hand, observe that one can go from $je_1+(n-j)e_2$ to $m_nx_{j,n}$ by taking at most $m_n$ steps of type $e_1$ and then at most $m_n$ steps of type $e_2$.
Since $m_nk\le nk/(k-1)+k\le 3n$ we have
	\begin{align}\label{aux2}
	\log P_{je_1+(n-j)e_2}^{\p^\lambda}\{X_{m_nk-n}=m_n x_{j,n}\}
	\ge -2\max\Big\{\sum_{0\le i\le 3n/k}A_{x+iz,z}:\abs{x}_1\le3n,z\in\{e_1,e_2\}\Big\}.
	\end{align}
(Taking a maximum over all $x$ with $\abs{x}_1\le3n$ is an overkill, but still good enough for our purposes.)
Similarly,
	\[\Bd^\lambda(je_1+(n-j)e_2,ne_2)=\Bd^\lambda(je_1+(n-j)e_2,m_n x_{j,n})\Bd^\lambda(m_n x_{j,n},m_n k e_2)\Bd^\lambda(m_n k e_2,ne_2)\]
and for the same reason as \eqref{aux2} we have
	\begin{align*}
	&\log\Bd^\lambda(je_1+(n-j)e_2,m_n x_{j,n})+\log\Bd^\lambda(m_n k e_2,ne_2)\\
	&\qquad\le 3\max\Big\{\sum_{0\le i\le 3n/k}A_{x+iz,z}:\abs{x}_1\le3n,z\in\{e_1,e_2\}\Big\}
	\end{align*}
%Lastly,
%	\[\log\Bdla(m_n k e_2,ne_2)=-\sum_{i=n}^{m_nk-1}\log\Bdla(ie_2,(i+1)e_2).\]
	
Let $D_k=\{x/k:x\in\Z_+^2,\abs{x}_1=k\}$. Collect the above bounds and continue from \eqref{upper} to write
	\begin{align}\label{aux89}
	\begin{split}
	n^{-1}\log \Bd^\lambda_{0,ne_2}
	\le&\max_{\xi\in D_k}\Big\{n^{-1}\log P_0^{\p^\lambda}\{X_{m_nk}=m_nk\xi\}+n^{-1}\log\Bd^\lambda(m_n k\xi,m_n k e_2)\Big\}\\
	&+5n^{-1}\max\Big\{\sum_{0\le i\le 3n/k}A_{x+iz,z}:\abs{x}_1\le3n,z\in\{e_1,e_2\}\Big\}.%-n^{-1}\sum_{i=n}^{m_nk-1}\log\Bdla(ie_2,(i+1)e_2).
	\end{split}
	\end{align}

%Since $(m_nk-n)/n\le 1/(k-1)+k/n$ and $\log\Bdla(0,e_2)$ is integrable, the ergodic theorem implies that in probability 
%	\[\lim_{k\to\infty}\lim_{n\to\infty}n^{-1}\abs{\log\Bdla(m_n k e_2,ne_2)}
%	\le\lim_{k\to\infty}\lim_{n\to\infty}n^{-1}\sum_{i=n}^{m_nk-1}\abs{\log\Bdla(ie_2,(i+1)e_2)}=0.\]
Since for each $x\in\Z_+^2$ and $z\in\{e_1,e_2\}$, $\{A_{x+iz,z}:i\in\Z_+\}$ are i.i.d.\ and have strictly more than two moments, Lemma A.4 of \cite{Ras-Sep-Yil-13} 
implies that 
	\[\lim_{k\to\infty}\lim_{n\to\infty}n^{-1}\max\Big\{\sum_{i=0}^{3n/k}A_{x+iz,z}:\abs{x}_1\le 3n,z\in\{e_1,e_2\}\Big\}=0.\]
%\textcolor{blue}{(Alternatively, we can keep working with the $\Bd^\lambda$ and not subsume them into $A$ (as done at the top display of this page). Then in \eqref{aux2}, $A$ would just be $A_x=\max_{j=1,2}\abs{\log\p^\lambda_{x,\,x+e_j}}$ and these are i.i.d.\ To control the $\Bd^\lambda$ terms we observe that $\Bd^\lambda$ is a stationary $L^1$ cocycle 
%and that by \eqref{aver2} we have $\log\Bd^\lambda_{x,\,x+z}\ge\log\p^\lambda_{x,\,x+z}$. Hence, the ergodic theorem for cocycles, in \cite{Geo-etal-15} controls $\Bd^\lambda$ errors.  Timo: which exposition would be better?}

Note that $m_nk/n\to k/(k-1)$ as $n\to\infty$. Applying the last display, \eqref{aux1}, \eqref{qLDP}, and the ergodic theorem 
(similarly to how the right-hand side of \eqref{aux99} was obtained) we get
	\begin{align*}
	\psi_0(\alpha+\lambda)-\psi_0(\lambda)
	&\le\lim_{k\to\infty}\lim_{n\to\infty} \max_{\xi\in D_k}\Big\{n^{-1}\log P_0^{\p^\lambda}\{X_{m_nk}=m_nk\xi\}+n^{-1}\log\Bd^\lambda(m_n k\xi,m_n k e_2)\Big\}\\
	&= \lim_{k\to\infty}\frac{k}{k-1}\max_{\xi\in D_k}\big\{-I_q(\xi)+\xi_1(\psi_0(\alpha+\beta+\lambda)-\psi_0(\lambda))\big\}\\
	&\le \max_{\xi\in\Uset}\big\{-I_q(\xi)+\xi_1(\psi_0(\alpha+\beta+\lambda)-\psi_0(\lambda))\big\}.
	\end{align*}
\eqref{I-97} is proved.
\end{proof}

\begin{proof}[Proof of Theorem \ref{thm:Iq3}]   Equation \eqref{Iq17} is the same as  Theorem \ref{th:Buse}\eqref{B:qldp} above,  where it was proved directly without recourse to the general variational formula \eqref{K-var1}.    Substitution of $B^\xi$ on the right-hand side of \eqref{K-var1} now verifies  that the infimum    is attained at  $B=B^\xi$.     After $I_q$ is extended to all of $\R_+$, formula \eqref{Iq17} remains valid.  This and  calculus verify \eqref{Iq18}.  
\end{proof}

\appendix

\section{Facts about polygamma functions}\label{psi-prop}

\begin{remark}
In what follows, some lengthy algebraic manipulations were  performed with Maple and checked with Sage.
Specifically, these were the last expansion in the proof of Lemma \ref{lm:psi-exp}, the expansions in the proofs of Lemmas \ref{lm:exp1} and \ref{lm:exp2}, and the computation of 
$I''''_q(\llnv)$ at the very end of the paper.
\end{remark}

Let us recall a few facts about polygamma functions $\psi_0(x)=\Gamma'(x)/\Gamma(x)$ and  $\psi_n(x)=\psi_{n-1}'(x)$ for $x>0$ and $n\in\N$.  

For $n\ge1$ we have the integral representation
	\begin{align}\label{psi-int}
	\psi_n(x)=-\int_0^\infty\frac{(-t)^ne^{-xt}}{1-e^{-t}}\,dt\,.
	\end{align}
See formula 6.4.1 in \cite{Abr-Ste-92}. In particular, we see that for $n\ge1$, $\psi_n(x)$ never vanishes, has sign $(-1)^{n-1}$ for all $x>0$, $\psi_0$ is strictly concave and  increasing,
while $\psi_1$ is strictly convex and decreasing with $\psi_1(x)\to\infty$ when $x\searrow0$ and $\psi_1(x)\to0$ when $x\to\infty$.
%x^n\psi_n(x)=(-1)^{n-1} x^{-1} \int_0^\infty s^n e^{-s}/(1-e^{-s/x}) ds   x^{-1}(1-e^{-s/x})\to s as x\to\infty. so limit is (-1)^{n-1}\int_0^\infty s^{n-1}e^{-s}ds=(n-1)!
%x^2(\psi_1(x)-\psi_1(x+a))=x^2\int_0^\infty te^{-xt}(1-e^{-at))/(1-e^{-t}) dt = \int_0^\infty s e^{-s} (1-e^{-as/x})/(1-e^{-s/x}) ds \to  a \int_0^\infty s e^{-s} ds = a
%We also have the expansions
%	\begin{align}\label{psi-sum}
%	\psi_n(x) = (-1)^{n+1}n! \sum_{k=0}^\infty\frac{1}{(x+k)^{n+1}}
%	\end{align}
%for $n\ge1$ and

Differentiating the relation $\Gamma(x+1)=x\Gamma(x)$, dividing by $\Gamma(x)$, then differentiating $n$ times gives 
the recurrence relation
	\begin{align}\label{psi-rec}
	\psi_n(x)=\psi_n(x+1)-(-1)^n n!\, x^{-(n+1)}\quad\text{for all }n\ge0\text{ and }x>0.
	\end{align}
In particular, this shows that $\psi_0(x)\sim\log x\to\infty$ as $x\to\infty$.

Combining formulas 6.3.5 and 6.3.16 from \cite{Abr-Ste-92} we also have the expansion
	\begin{align}\label{psi0}
	\psi_0(x)=-\gamma-\frac1x+\sum_{k=1}^\infty\Big(\frac1k-\frac1{x+k}\Big),
	\end{align}
where $\gamma=\lim_{n\to\infty}(-\log n+\sum_{k=1}^n k^{-1})$ is Euler's constant and $x>0$. In particular, $\psi_0(x)\to-\infty$ as $x\searrow0$.

\begin{lemma}\label{lm:psi-exp}
For $n\ge1$ and $a>0$ fixed we have
	\begin{align*}
	&\psi_n(x+a)-\psi_n(x)\\
	&=\tfrac{(-1)^{n}a}{x^{n+1}}\Big(n!-\tfrac{(a-1)(n+1)!}{2x}+\tfrac{(a-1)(2a-1)(n+2)!}{12x^2}-\tfrac{a(a-1)^2(n+3)!}{24x^3}+\tfrac{(6a^4-15a^3+10a^2-1)(n+4)!}{720x^4}+\cO(\tfrac1{x^5})\Big).
	\end{align*}
\end{lemma}

\begin{proof}
Using \eqref{psi-int} write 
	\begin{align*}
	&\psi_n(x+a)-\psi_n(x)
	=(-1)^{n+1}\int_0^\infty \frac{t^n e^{-xt}(e^{-at}-1)}{1-e^{-t}}\,dt\\
	&=\tfrac{(-1)^{n+1}}{x^{n+1}}\int_0^\infty s^n e^{-s}\cdot\frac{e^{-as/x}-1}{1-e^{-s/x}}\,ds\\
	&=\tfrac{(-1)^{n}}{x^{n+1}}\int_0^\infty s^n e^{-s}\cdot\frac{\frac{as}{x}-\frac{a^2s^2}{2x^2}+\frac{a^3s^3}{6x^3}+\cdots+\cO(\frac{s^6}{x^6})}{\frac{s}{x}-\frac{s^2}{2x^2}+\frac{s^3}{6x^3}+\cdots+\cO(\frac{s^6}{x^6})}\,ds\\
%	&=\frac{(-1)^{n}a}{x^{n+1}}\int_0^\infty s^n e^{-s}\cdot\frac{1-\frac{as}{2x}+\frac{a^2s^2}{6x^2}+\cO(\frac{s^3}{x^3})}{1-\frac{s}{2x}+\frac{s^2}{6x^2}+\cO(\frac{s^3}{x^3})}\,ds\\
%	&=\frac{(-1)^{n}a}{x^{n+1}}\int_0^\infty s^n e^{-s}\Big(1-\frac{as}{2x}+\frac{a^2s^2}{6x^2}+\cO(\frac{s^3}{x^3})\Big)\Big(1+\frac{s}{2x}-\frac{s^2}{6x^2}+\frac{s^2}{4x^2}+\cO(\frac{s^3}{x^3})\Big)\,ds\\
%	&=\frac{(-1)^{n}a}{x^{n+1}}\int_0^\infty s^n e^{-s}\Big(1-\frac{as}{2x}+\frac{a^2s^2}{6x^2}+\cO(\frac{s^3}{x^3})\Big)\Big(1+\frac{s}{2x}+\frac{s^2}{12x^2}+\cO(\frac{s^3}{x^3})\Big)\,ds\\
%	&=\frac{(-1)^{n}a}{x^{n+1}}\int_0^\infty s^n e^{-s}\Big(1-\frac{as}{2x}+\frac{a^2s^2}{6x^2}+\frac{s}{2x}-\frac{as^2}{4x^2}+\frac{s^2}{12x^2}+\cO(\frac{s^3}{x^3})\Big)\,ds\\
	&=\tfrac{(-1)^{n}a}{x^{n+1}}\int_0^\infty s^n e^{-s}\Big(1-\tfrac{(a-1)s}{2x}+\tfrac{(a-1)(2a-1)s^2}{12x^2}-\tfrac{a(a-1)^2s^3}{24x^3}+\tfrac{(6a^4-15a^3+10a^2-1)s^4}{720x^4}+\cO(\tfrac{s^5}{x^5})\Big)\,ds.
%	&=\tfrac{(-1)^{n}a}{x^{n+1}}\Big(n!-\tfrac{(a-1)(n+1)!}{2x}+\tfrac{(a-1)(2a-1)(n+2)!}{12x^2}-\tfrac{a(a-1)^2(n+3)!}{24x^3}+\tfrac{(6a^4-15a^3+10a^2-1)(n+4)!}{720x^4}+\cO(\tfrac1{x^5})\Big).
	\end{align*}
The claim now follows from $\int_0^\infty s^n e^{-s}\,ds=n!$. %\textcolor{blue}{(the last expansion was done by hand up to $x^{-2}$ [computation commented out] and then with maple and checked with sage.)}
\end{proof}

%In particular, we have the limit $x^2(\psi_1(x)-\psi_1(x+a))\to a$ as $x\to\infty$. 

\begin{lemma}\label{psi-con}
Function $\psi_2\circ\psi_1^{-1}$ is strictly concave; 
\end{lemma}

\begin{proof}
See for example the proof of Lemma 5.3 in 
\cite{Bar-Cor-17}. We give here the details for completeness. 

We have 
	\[(\psi_2\circ\psi_1^{-1})''=\Big(\frac{\psi_3\circ\psi_1^{-1}}{\psi_2\circ\psi_1^{-1}}\Big)'
	=\frac{\psi_4\circ\psi_1^{-1}-(\psi_3\circ\psi_1^{-1})^2/\psi_2\circ\psi_1^{-1}}{(\psi_2\circ\psi_1^{-1})^2}\,.\]
Since $\psi_2$ is negative, strict concavity of $\psi_2\circ\psi_1^{-1}$ would follow from showing that $\psi_4\psi_2>\psi_3^2$.
By the integral representation \eqref{psi-int} this is equivalent to 
	\[\int_0^\infty\!\!\int_0^\infty\frac{t^3s^3e^{-xt-xs}}{(1-e^{-t})(1-e^{-s})}\,dt\,ds<\int_0^\infty\!\!\int_0^\infty\frac{t^2s^4e^{-xt-xs}}{(1-e^{-t})(1-e^{-s})}\,dt\,ds\,.\]
Symmetrizing the right-hand side (i.e.\ adding another copy with $s$ and $t$ interchanged) the above becomes
	\[\int_0^\infty\!\!\int_0^\infty\frac{t^2s^2e^{-xt-xs}}{(1-e^{-t})(1-e^{-s})}(2ts)\,dt\,ds<\int_0^\infty\!\!\int_0^\infty\frac{t^2s^2e^{-xt-xs}}{(1-e^{-t})(1-e^{-s})}(t^2+s^2)\,dt\,ds\,,\]
which is true for all $x>0$. 
\end{proof}

\begin{lemma}
For all $x>0$ we have 
	\[\psi_1(x)>\frac1x+\frac1{2x^2}\,.\]
\end{lemma}

\begin{proof}
Write
	\[\psi_1(x)=\int_0^\infty\frac{t e^{-xt}}{1-e^{-t}}\,dt=\frac1{x^2}\int_0^\infty\frac{s e^{-s}}{1-e^{-s/x}}\,ds.\]
Expand 
\begin{align}
\frac1{1-e^{-s/x}}
&=\frac1{\frac{s}{x}-\frac{s^2}{2x^2}+\frac{s^3}{6x^3}-\frac{s^4}{24x^4}+\cO(x^{-5})}\notag\\
&=\frac{x}{s}\frac1{1-\frac{s}{2x}+\frac{s^2}{6x^2}-\frac{s^3}{24x^3}+\cO(x^{-4})}\notag\\
&=\frac{x}{s}\cdot\Bigl(1+\frac{s}{2x}-\frac{s^2}{6x^2}+\frac{s^3}{24x^3}+\bigl(\frac{s}{2x}-\frac{s^2}{6x^2}\bigr)^2+\frac{s^3}{8x^3}+\cO(x^{-4})\Bigr)\notag\\
&=\frac{x}{s}\cdot\Bigl(1+\frac{s}{2x}-\frac{s^2}{6x^2}+\frac{s^3}{24x^3}+\frac{s^2}{4x^2}-\frac{s^3}{6x^3}+\frac{s^3}{8x^3}+\cO(x^{-4})\Bigr)\notag\\
&=\frac{x}{s}\cdot\Bigl(1+\frac{s}{2x}+\frac{s^2}{12x^2}+\cO(x^{-4})\Bigr).\label{e-exp}
\end{align}
This gives
	\begin{align}\label{psi1-exp}
	\psi_1(x)=\frac1x+\frac{1}{2x^2}+\frac{1}{6x^3}+\cO(x^{-5}).
	\end{align}
We thus see that the claim of the lemma holds for all large enough $x$.

Next, assume that for some $x>0$ we have
	\[\psi_1(x+1)>\frac1{x+1}+\frac1{2(x+1)^2}\,.\]
Use  \eqref{psi-rec} to write
	\begin{align*}
	\psi_1(x)-\frac1x-\frac1{2x^2}
	&=\psi_1(x+1)-\frac1x+\frac1{2x^2}\\
	&>\frac1{x+1}+\frac1{2(x+1)^2}-\frac1x+\frac1{2x^2}\\
	&=\frac{1}{2x^2(x+1)^2}>0.
	\end{align*}
Thus, we see that if the claim holds for $x+1$ it holds also for $x$.
This and the fact that it holds for all large enough $x$ implies that it holds for all $x>0$.
\end{proof}

\begin{lemma}\label{psi-diff}
We have for all $x>0$
	\[\psi_2(x)+\psi_1(x)^2>0.\]
\end{lemma}

\begin{proof}
We proceed similarly to the above lemma. First, we compute the expansions for large $x$.
Write
	\[\psi_2(x)=-\int_0^\infty\frac{t^2 e^{-xt}}{1-e^{-t}}\,dt=-\frac1{x^3}\int_0^\infty\frac{s^2 e^{-s}}{1-e^{-s/x}}\,ds.\]
From \eqref{e-exp} we get
	\[\psi_2(x)=-\frac1{x^2}\Bigl(1+\frac1{x}+\frac{1}{2x^2}+\cO(x^{-4})\Bigr).\]
And from \eqref{psi1-exp} we have
	\[\psi_1(x)^2=\frac1{x^2}\Big(1+\frac{1}{x}+\frac{1}{3x^2}+\frac{1}{4x^2}+\frac1{6x^3}+\cO(x^{-4})\Big).\]
Hence
	\[\psi_2(x)+\psi_1(x)^2=\frac1{12x^4}+\cO(x^{-3}),\]
which says that the claim of the lemma holds for all large enough $x$.

Next, assume that for some $x>0$ we have
	\[\psi_2(x+1)+\psi_1(x+1)^2>0.\]
Use \eqref{psi-rec} (twice for the first equality and once for the last one) to write
	\begin{align*}
	\psi_2(x)+\psi_1^2(x)
	&=\psi_2(x+1)-\frac2{x^3}+\bigl(\psi_1(x+1)+\frac1{x^2}\bigr)^2\\
	&=\psi_2(x+1)+\psi_1(x+1)^2+\frac{2\psi_1(x+1)}{x^2}-\frac2{x^3}+\frac1{x^4}\\
	&>\frac{2\psi_1(x+1)}{x^2}-\frac2{x^3}+\frac1{x^4}\\
	&=\frac{2}{x^2}\Bigl(\psi_1(x+1)-\frac1{x}+\frac1{2x^2}\Bigr)\\
	&=\frac{2}{x^2}\Bigl(\psi_1(x)-\frac1{x}-\frac1{2x^2}\Bigr).
	\end{align*}
The last quantity is positive, by the above lemma. Hence, we see that if the inequality claimed in the lemma
holds for $x+1$ it holds for $x$ as well. This, and the fact the inequality holds for all large enough $x$ implies the 
inequality holds for all $x>0$.
\end{proof}

\section{Differentiability and expansion of $I_q$}\label{app:exp}

 When convenient we consider  $I_q$ and $\lambda$ from \eqref{Iq} and Lemma \ref{lam-xi-t}\eqref{lam-xi} as functions of $\xi_1$: $I_q(\xi_1)=I_q(\xi_1,1-\xi_1)$ and $\lambda(\xi_1)=\lambda(\xi_1,1-\xi_1)$.

%%%%%%%%%%%%%%   OLD DIFFERENTIABILITY LEMMA  
%\begin{lemma}\label{Cinfty}
%Both $\lambda$ and $I_q$ are infinitely differentiable at $\xi_1\ne\llnv_1$.
%\end{lemma}
%
%\begin{proof}
%We focus on the case $\xi_1\in[\llnv_1,1]$.
%An inductive argument involving repeated differentiation of \eqref{lambda-aux}  shows
%that $\lambda$ is infinitely differentiable at $\xi_1\in(0,1)$ 
%if \[\psi_2(\alpha+\lambda(\xi))-\psi_2(\lambda(\xi))  \ne \xi_1 (\psi_2(\alpha+\beta+\lambda(\xi)) -\psi_2(\lambda(\xi)).\]
%Using \eqref{eq:lam-xi1} this becomes
%	\[\frac{\psi_2(\alpha+\beta+\lambda)-\psi_2(\lambda)}{\psi_1(\alpha+\beta+\lambda)-\psi_1(\lambda)}\ne 
%	\frac{\psi_2(\alpha+\lambda)-\psi_2(\lambda)}{\psi_1(\alpha+\lambda)-\psi_1(\lambda)}\,,\]
%which holds by \eqref{nice}. 
%%(which is a result of strict concavity of $\psi_2\circ\psi_1^{-1}$, Lemma \ref{psi-con}).
%Infinite differentiability of $I_q$ follows from that of $\lambda$. 
%\end{proof} 

\begin{lemma}  \label{lm700}    There exist open sets $G_0, G_1$ in $\bC$ such that $(0,\llnv_1)\subset G_0$ and $(\llnv_1,1)\subset G_1$, and holomorphic functions $d_0$ and $f_0$ on $G_0$ and $d_1$ and $f_1$ on $G_1$ such that $\lambda(\xi)=d_0(\xi_1)$ and $I_q(\xi)=f_0(\xi_1)$ for $\xi_1\in(0,\llnv_1)$  and $\lambda(\xi)=d_1(\xi_1)$ and $I_q(\xi)=f_1(\xi_1)$ for $\xi_1\in(\llnv_1,1)$. 
\end{lemma}

\begin{proof} We present the proof for the case $G_1$.  Let 
\begin{align*} 
 \sigma_1(z)=
\frac{\psi_1(z)-\psi_1(\alpha+z)}{\psi_1(z)-\psi_1(\alpha+\beta+z)} , 
\quad z\in \bH_1=\{z: \Re z>0\}, 
 %\\  &
%=  \frac{\ddd\sum_{k=0}^\infty \frac{\alpha^2+2\alpha(k+z)}{(k+z)^2(k+\alpha+z)^2}}{\ddd\sum_{k=0}^\infty \frac{(\alpha+\beta)^2+2(\alpha+\beta)(k+z)}{(k+z)^2(k+\alpha+\beta+z)^2}}  
\end{align*} 
denote the function defined by \eqref{eq:lam-xi1}, now thought of as a function on the open right half plane.  The set $U_1=\{z\in\bH_1: \psi_1(z)\ne \psi_1(\alpha+\beta+z)\}$ is an open subset of $\bH_1$ that contains the  half-line $(0,\infty)$.   Since $\psi_1$ is holomorphic on $\bH_1$,  it follows that $\sigma_1$ is holomorphic on $U_1$.  By the proof of Lemma \ref{lam-xi-t} in Section \ref{Iq-pf},   $\sigma_1'(\lambda)<0$ for $0<\lambda<\infty$.   By Theorem 10.30 in \cite{Rud-87}, each $\lambda\in(0,\infty)$ is the center of an open  disk in   $U_1$   on which $\sigma_1$ is one-to-one.   By using the strict negativity and continuity  of $\sigma_1'$ on  $(0,\infty)$, we can take these disks small enough so that $\sigma_1$ is a one-to-one mapping of the union $V_1$ of these disks.  

 $G_1=\sigma_1(V_1)$   is an open set since holomorphic functions are open mappings.  $G_1$ contains the open real interval  $(\llnv_1,1)$ by Lemma \ref{lam-xi-t}.   By   Theorem 10.33 in \cite{Rud-87},   the inverse function  $d_1(w) $  of $\sigma_1$ is holomorphic on $G_1$.  Since $d_1$ maps $G_1$ into $\bH_1$ where $\psi_0$ is holomorphic,  
\[ f_1(w)= w\psi_0\bigl(\alpha+\beta+d_1(w)\bigr)+(1-w)\psi_0\bigl(d_1(w)\bigr)-\psi_0\bigl(\alpha+d_1(w)\bigr) \]
is holomorphic on $G_1$.  From the definitions it now follows that  for $\xi_1\in(\llnv_1,1)$,   $d_1(\xi_1)=\lambda(\xi)$ and  $f_1(\xi_1)=I_q(\xi)$.  
\end{proof} 

\begin{lemma}\label{lm:exp1}
We have the following expansion for $\xi_1$ from \eqref{eq:lam-xi1}:
\begin{align*}
\xi_1(\lambda)
&=\tfrac{\alpha}{\alpha+\beta}\Big(1+\tfrac{\beta}{\lambda}+\tfrac{\beta(1-2\alpha)}{2\lambda^2}+\tfrac{\beta\alpha(\alpha-1)}{\lambda^3}
-\tfrac{\beta(12\alpha^3  - 18\alpha^2+\alpha-\beta+3)}{12\lambda^4}+\cO(\tfrac1{\lambda^5})\Big).
\end{align*}
\end{lemma}

\begin{proof}
Applying Lemma \ref{lm:psi-exp} to \eqref{eq:lam-xi1} we have
	\begin{align*}
	\xi_1(\lambda)
	&=\frac{\psi_1(\alpha+\lambda)-\psi_1(\lambda)}{\psi_1(\alpha+\beta+\lambda)-\psi_1(\lambda)}\\
	&=\frac{\alpha}{\alpha+\beta}\cdot\frac{1-\frac{\alpha-1}{\lambda}+\frac{(\alpha-1)(2\alpha-1)}{2\lambda^2}+\cdots+\cO(\lambda^{-5})}{1-\frac{\alpha+\beta-1}{\lambda}+\frac{(\alpha+\beta-1)(2\alpha+2\beta-1)}{2\lambda^2}+\cdots+\cO(\lambda^{-5})}.
%	&=\frac{\alpha}{\alpha+\beta}\Big(1-\frac{\alpha-1}{\lambda}+\frac{(\alpha-1)(2\alpha-1)}{2\lambda^2}+\cO(\lambda^{-3})\Big)\\
%	&\qquad\times\Big(1+\frac{\alpha+\beta-1}{\lambda}-\frac{(\alpha+\beta-1)(2\alpha+2\beta-1)}{2\lambda^2}+\frac{(\alpha+\beta-1)^2}{\lambda^2}+\cO(\lambda^{-3})\Big)\\
%	&=\frac{\alpha}{\alpha+\beta}\Big(1-\frac{\alpha-1}{\lambda}+\frac{(\alpha-1)(2\alpha-1)}{2\lambda^2}+\cO(\lambda^{-3})\Big)\\
%	&\qquad\times\Big(1+\frac{\alpha+\beta-1}{\lambda}-\frac{\alpha+\beta-1}{2\lambda^2}+\cO(\lambda^{-3})\Big)\\
%	&=\frac{\alpha}{\alpha+\beta}\Big(1-\frac{\alpha-1}{\lambda}+\frac{(\alpha-1)(2\alpha-1)}{2\lambda^2} +\frac{\alpha+\beta-1}{\lambda}\\
%	&\qquad\qquad-\frac{(\alpha-1)(\alpha+\beta-1)}{\lambda^2} -\frac{\alpha+\beta-1}{2\lambda^2} +\cO(\lambda^{-3})\Big)\\
%	&=\frac{\alpha}{\alpha+\beta}\Big(1+\frac{\beta}{\lambda}+\frac{\beta(1-2\alpha)}{2\lambda^2}+...+\cO(\lambda^{-5})\Big).
	\end{align*}
The claim follows. 
%\textcolor{blue}{(expansion was done by hand up to $\lambda^{-2}$ [computation commented out] and then with maple and checked with sage.)}
\end{proof}

\begin{lemma}\label{lm:exp2}
With $\xi_1$ from \eqref{eq:lam-xi1} we have the following expansions as $\mu\to\infty$:
	\begin{align*}
	&\lambda'(\xi_1(\mu))=-\tfrac{\alpha+\beta}{\alpha\beta}\Bigl(\mu^2+(2\alpha-1)\mu+(\alpha^2-\alpha+1)+\tfrac{\alpha-\beta}{3\mu}\Bigr)+\cO(\tfrac1{\mu^2}),\\
	&\lambda''(\xi_1(\mu))=\tfrac{(\alpha+\beta)^2}{\alpha^2\beta^2}\Bigl(2\mu^3+3(2\alpha-1)\mu^2+3(2\alpha^2-2\alpha+1)\mu
	+\tfrac{6\alpha^3-9\alpha^2+10\alpha-\beta-3}3\Bigr)+\cO(\mu^{-1}),\\
	&\lambda'''(\xi(\mu))=-\tfrac{(\alpha+\beta)^3}{\alpha^3\beta^3}\Bigl(6\mu^4+12(2\alpha-1)\mu^3+3(12\alpha^2-12\alpha+5)\mu^2\\
	&\qquad\qquad\qquad\qquad\qquad\quad+(24\alpha^3-36\alpha^2+32\alpha-2\beta-9)\mu\Bigr)+\cO(1).
	\end{align*}
\end{lemma}

\begin{proof}
Use \eqref{eq:lam-xi1} to write
	\begin{align*}%\label{lambda-aux}
	\xi_1\bigl[\psi_1\bigl(\alpha+\beta+\lambda(\xi_1)\bigr)-\psi_1\bigl(\lambda(\xi_1)\bigr)\bigr]=\psi_1\bigl(\alpha+\lambda(\xi_1)\bigr)-\psi_1\bigl(\lambda(\xi_1)\bigr).
	\end{align*}
Differentiating this %\eqref{lambda-aux} 
in $\xi_1$ and solving for $\lambda'(\xi_1)$ we get
	\[\lambda'(\xi_1)
	=\frac{\psi_1\bigl(\alpha+\beta+\lambda(\xi_1)\bigr)-\psi_1\bigl(\lambda(\xi_1)\bigr)}
	{\Bigl(\psi_2\bigl(\alpha+\lambda(\xi_1)\bigr)-\psi_2\bigl(\lambda(\xi_1)\bigr)\Bigr)-\xi_1\Bigl(\psi_2\bigl(\alpha+\beta+\lambda(\xi_1)\bigr)-\psi_2\bigl(\lambda(\xi_1)\bigr)\Bigr)}\,.\]
%Applying the above expansions we get
%	\begin{align*}
%	\lambda'(\xi(\mu))
%	&=-\mu^{-2}(\alpha+\beta)\Bigl(1-(\alpha+\beta-1)\mu^{-1}+\cO(\mu^{-2})\Bigr)\\
%	&\quad\cdot\mu^3\Bigl(\alpha\bigl(2-3(\alpha-1)\mu^{-1}+2(\alpha-1)(2\alpha-1)\mu^{-2}\bigr)\\
%	&\qquad\quad-\frac{\alpha}{\alpha+\beta}\bigl(1+\beta\mu^{-1}+\beta(1-2\alpha)\mu^{-2}/2\bigr)\\
%	&\qquad\qquad\cdot(\alpha+\beta)\bigl(2-3(\alpha+\beta-1)\mu^{-1}+2(\alpha+\beta-1)(2\alpha+2\beta-1)\mu^{-2}\bigr)+\cO(\mu^{-3})\Bigr)^{-1}\\
%	&=-\frac{\alpha+\beta}{\alpha}\mu\Bigl(1-(\alpha+\beta-1)\mu^{-1}+\cO(\mu^{-2})\Bigr)\\
%	&\quad\cdot\Bigl(2-(3\alpha-3)\mu^{-1}+(4\alpha^2-6\alpha+2)\mu^{-2}\\
%	&\qquad-2+(3\alpha+\beta-3)\mu^{-1}-(4\alpha^2+\beta^2+3\alpha\beta-6\alpha-2\beta+2)\mu^{-2}
%          +\cO(\mu^{-3})\Bigr)^{-1}\\
%	%&=-\frac{\alpha+\beta}{\alpha}\mu\Bigl(1-(\alpha+\beta-1)\mu^{-1}+\cO(\mu^{-2})\Bigr)\Bigl(\beta\mu^{-1}-(\beta^2+3\alpha\beta-2\beta)\mu^{-2}+\cO(\mu^{-3})\Bigr)^{-1}\\
%	&=-\frac{\alpha+\beta}{\alpha\beta}\mu^2\Bigl(1-(\alpha+\beta-1)\mu^{-1}+\cO(\mu^{-2})\Bigr)\Bigl(1-(\beta+3\alpha-2)\mu^{-1}+\cO(\mu^{-2})\Bigr)^{-1}\\
%	&=-\frac{\alpha+\beta}{\alpha\beta}\mu^2\Bigl(1+(2\alpha-1)\mu^{-1}+\cO(\mu^{-2})\Bigr).
%         \end{align*}
%
Differentiating  a second time we get
	\begin{align*}
	\lambda''(\xi_1)
	&=\frac{2\Bigl(\psi_2\bigl(\alpha+\beta+\lambda(\xi_1)\bigr)-\psi_2\bigl(\lambda(\xi_1)\bigr)\Bigr)\lambda'(\xi_1)}{\psi_2(\alpha+\lambda(\xi_1))-\psi_2(\lambda(\xi_1))-\xi_1\bigl(\psi_2(\alpha+\beta+\lambda(\xi_1))-\psi_2(\lambda(\xi_1))\bigr)}\\
	&\quad+\frac{\Bigl[\xi_1\bigl(\psi_3(\alpha+\beta+\lambda(\xi_1))-\psi_3(\lambda(\xi_1))\bigr)-\bigl(\psi_3(\alpha+\lambda(\xi_1))-\psi_3(\lambda(\xi_1))\bigr)\Bigr](\lambda'(\xi_1))^2}
	{\psi_2(\alpha+\lambda(\xi_1))-\psi_2(\lambda(\xi_1))-\xi_1\bigl(\psi_2(\alpha+\beta+\lambda(\xi_1))-\psi_2(\lambda(\xi_1))\bigr)}\,.
	\end{align*}
	
A third round of differentiation gives	
	\begin{align*}
	&\lambda'''(\xi_1)\\
	&=\frac{3\Bigl(\psi_3\bigl(\alpha+\beta+\lambda(\xi_1)\bigr)-\psi_3\bigl(\lambda(\xi_1)\bigr)\Bigr)\bigl(\lambda'(\xi_1)\bigr)^2+3\Bigl(\psi_2\bigl(\alpha+\beta+\lambda(\xi_1)\bigr)-\psi_2\bigl(\lambda(\xi_1)\bigr)\Bigr)\lambda''(\xi_1)}{\psi_2(\alpha+\lambda(\xi_1))-\psi_2(\lambda(\xi_1))-\xi_1\bigl(\psi_2(\alpha+\beta+\lambda(\xi_1))-\psi_2(\lambda(\xi_1))\bigr)}\\
	&\quad+\frac{\Bigl[\xi_1\bigl(\psi_4(\alpha+\beta+\lambda(\xi_1))-\psi_4(\lambda(\xi_1))\bigr)-\bigl(\psi_4(\alpha+\lambda(\xi_1))-\psi_4(\lambda(\xi_1))\bigr)\Bigr](\lambda'(\xi_1))^3}
	{\psi_2(\alpha+\lambda(\xi_1))-\psi_2(\lambda(\xi_1))-\xi_1\bigl(\psi_2(\alpha+\beta+\lambda(\xi_1))-\psi_2(\lambda(\xi_1))\bigr)}\\
	&\quad+\frac{3\Bigl[\xi_1\bigl(\psi_3(\alpha+\beta+\lambda(\xi_1))-\psi_3(\lambda(\xi_1))\bigr)-\bigl(\psi_3(\alpha+\lambda(\xi_1))-\psi_3(\lambda(\xi_1))\bigr)\Bigr]\lambda'(\xi_1)\lambda''(\xi_1)}
	{\psi_2(\alpha+\lambda(\xi_1))-\psi_2(\lambda(\xi_1))-\xi_1\bigl(\psi_2(\alpha+\beta+\lambda(\xi_1))-\psi_2(\lambda(\xi_1))\bigr)}\,.
	\end{align*}

The claims now come by applying  the  expansions from  Lemma \ref{lm:psi-exp}. %\textcolor{blue}{(Did $\lambda'$ by hand to order $\mu$ [computation commented out]. Did the rest with maple and checked with sage.)}
\end{proof}

Now we can prove expansion \eqref{Iq-exp} of $I_q$, claimed in Remark \ref{rk:I-reg}. We focus only on the case $\xi_1\in[\frac{\alpha}{\alpha+\beta},1]$, the other case being similar.
Consider  $I_q$ and $\lambda$ as functions of $\xi_1$: $I_q(\xi_1)=I_q(\xi_1,1-\xi_1)$ and $\lambda(\xi_1)=\lambda(\xi_1,1-\xi_1)$.

%Lemma \ref{lm700} shows that $\lambda$ and $I_q$ are infinitely differentiable at $\xi_1\ne\llnv_1$.
%Taking $\xi_1\to\llnv_1$ makes $\lambda(\xi_1)$ go to infinity and  since $\psi_0(s)\sim\log s$ as $s\to\infty$ we get that $I_q(\llnv)=0$. 
By \eqref{eq:lam-xi1} 
\begin{align*}
I_q'(\xi_1)
&=\psi_0\bigl(\alpha+\beta+\lambda(\xi_1)\bigr)-\psi_0\bigl(\lambda(\xi_1)\bigr)\\
&\qquad+\bigl[\xi_1\psi_1\bigl(\alpha+\beta+\lambda(\xi_1)\bigr)+(1-\xi_1)\psi_1\bigl(\lambda(\xi_1)\bigr)-\psi_1\bigl(\alpha+\lambda(\xi_1)\bigr)\bigr]\lambda'(\xi_1)\\
&=\psi_0\bigl(\alpha+\beta+\lambda(\xi_1)\bigr)-\psi_0\bigl(\lambda(\xi_1)\bigr).
\end{align*}
%Then \eqref{eq:lam-xi1} implies that the second line in the above display vanishes and we simply have
%\[I_q'(\xi_1)=\psi_0\bigl(\alpha+\beta+\lambda(\xi_1)\bigr)-\psi_0\bigl(\lambda(\xi_1)\bigr).\]
Thus, $I_q'(\xi_1)\to0$ as $\xi_1\searrow\llnv_1$. A similar argument works for $\xi_1\nearrow\llnv_1$ and thus $I_q'(\llnv)=0$. Taking a second derivative we have
\begin{align}\label{I''}
I_q''(\xi_1)=\bigl[\psi_1\bigl(\alpha+\beta+\lambda(\xi_1)\bigr)-\psi_1\bigl(\lambda(\xi_1)\bigr)\bigr]\lambda'(\xi_1).
\end{align}
Since $\lambda$ is strictly decreasing on $[\llnv_1,1]$ and $\psi_1$ is strictly decreasing on $[0,\infty)$, the above is strictly positive.  This  gives another verification of the strict convexity of $I_q$.

Further differentiations give
\begin{align}\label{I'''}
I_q'''(\xi_1)
&=\bigl[\psi_2\bigl(\alpha+\beta+\lambda(\xi_1)\bigr)-\psi_2\bigl(\lambda(\xi_1)\bigr)\bigr](\lambda'(\xi_1))^2\\
&\quad+\bigl[\psi_1\bigl(\alpha+\beta+\lambda(\xi_1)\bigr)-\psi_1\bigl(\lambda(\xi_1)\bigr)\bigr]\lambda''(\xi_1)
\end{align}
and
	\begin{align}\label{I''''}
	\begin{split}
	I''''_q(\xi_1)
	&=\bigl[\psi_3\bigl(\alpha+\beta+\lambda(\xi_1)\bigr)-\psi_3\bigl(\lambda(\xi_1)\bigr)\bigr](\lambda'(\xi_1))^3\\
	&\quad+3\bigl[\psi_2\bigl(\alpha+\beta+\lambda(\xi_1)\bigr)-\psi_2\bigl(\lambda(\xi_1)\bigr)\bigr]\lambda'(\xi_1)\lambda''(\xi_1)\\
	&\quad+\bigl[\psi_1\bigl(\alpha+\beta+\lambda(\xi_1)\bigr)-\psi_1\bigl(\lambda(\xi_1)\bigr)\bigr]\lambda'''(\xi_1).
	\end{split}
	\end{align}
From \eqref{I''} and the  expansions in Lemmas \ref{lm:psi-exp}, \ref{lm:exp1}, and \ref{lm:exp2} we have
	\begin{align*}
	I''_q(\xi_1(\lambda))=-(\alpha+\beta)\lambda^{-2}\bigl(1+\cO(\lambda^{-1})\bigr)\Bigl(-\frac{\alpha+\beta}{\alpha\beta}\lambda^2+\cO(\lambda)\Bigr)\mathop{\longrightarrow}_{\lambda\to\infty}\frac{(\alpha+\beta)^2}{\alpha\beta}.
	\end{align*}	
In other words,	\[\lim_{\xi_1\searrow\llnv_1}I''_q(\xi_1)=\frac{(\alpha+\beta)^2}{\alpha\beta}>0.\]

%Use \eqref{eq:lam-xi1} to write
%	\begin{align}\label{lambda-aux}
%	\xi_1\bigl[\psi_1\bigl(\alpha+\beta+\lambda(\xi_1)\bigr)-\psi_1\bigl(\lambda(\xi_1)\bigr)\bigr]=\psi_1\bigl(\alpha+\lambda(\xi_1)\bigr)-\psi_1\bigl(\lambda(\xi_1)\bigr).
%	\end{align}
%and differentiate three times  in $\xi_1$ to get formulas  for $\lambda'(\xi_1)$, $\lambda''(\xi_1)$, and $\lambda'''(\xi_1)$. This gives $I''_q$, $I'''_q$, and $I''''_q$ in terms of 
%differences of the type $\psi_n(a+\lambda)-\psi_n(\lambda)$ with $n\ge1$ and $a>0$.  Using the integral representation \eqref{psi-int} 
%one can get expansions for these differences at large $\lambda$, which leads to expansion \eqref{Iq-exp}.  The details appear in Appendix \ref{app:exp}. 

Similarly,
\begin{align*}
I_q'''(\xi_1(\lambda))
&=\frac{(\alpha+\beta)^3}{\alpha^2\beta^2\lambda^3}\bigl(2-3(\alpha+\beta-1)\lambda^{-1}+\cO(\lambda^{-2})\bigr)\bigl(\lambda^2+(2\alpha-1)\lambda+\cO(1)\bigr)^2\\
&\quad-\frac{(\alpha+\beta)^3}{\alpha^2\beta^2\lambda^2}\bigl(1-(\alpha+\beta-1)\lambda^{-1}+\cO(\lambda^{-2})\bigr)\bigl(2\lambda^3+3(2\alpha-1)\lambda^2+\cO(\lambda)\bigr)\\
%=&\frac{(\alpha+\beta)^3\lambda}{\alpha^2\beta^2}\bigl(2-3(\alpha+\beta-1)\lambda^{-1}+\cO(\lambda^{-2}\bigr)\bigl(1+2(2\alpha-1)\lambda^{-1}+\cO(\lambda^{-2})\bigr)\\
%&\quad-\frac{(\alpha+\beta)^3\lambda}{\alpha^2\beta^2}\bigl(1-(\alpha+\beta-1)\lambda^{-1}+\cO(\lambda^{-2})\bigr)\bigl(2+3(2\alpha-1)\lambda^{-1}+\cO(\lambda^{-2})\bigr)\\
%=&\frac{(\alpha+\beta)^3\lambda}{\alpha^2\beta^2}\bigl(2+(5\alpha-3\beta-1)\lambda^{-1}+\cO(\lambda^{-2})\bigr)\\
%&\quad-\frac{(\alpha+\beta)^3\lambda}{\alpha^2\beta^2}\bigl(2+(4\alpha-2\beta-1)\lambda^{-1}+\cO(\lambda^{-2})\bigr)\\
&=\frac{(\alpha+\beta)^3(\alpha-\beta)}{\alpha^2\beta^2}\bigl(1+\cO(\lambda^{-1})\bigr)\mathop{\longrightarrow}_{\lambda\to\infty}\frac{(\alpha+\beta)^3(\alpha-\beta)}{\alpha^2\beta^2}
\end{align*}
%\textcolor{blue}{(Computation done by hand commented out.)}
and also %\textcolor{blue}{(used maple this time!)}
	\[\lim_{\xi_1\searrow\llnv_1}I''''_q(\xi_1)=\frac{(\alpha+\beta)^4(2\alpha^2-2\alpha\beta+2\beta^2+1)}{\alpha^3\beta^3}.\]

Write $I_q(\xi_1(\alpha,\beta),\alpha,\beta)$ to make the dependence on $\alpha$ and $\beta$ clear.
Then symmetry gives $\llnv_1(\alpha,\beta)=1-\llnv_1(\beta,\alpha)$ and $I_q(s,\alpha,\beta)=I_q(1-s,\beta,\alpha)$. Thus
	\[\lim_{s\nearrow\llnv_1(\alpha,\beta)}I^{(n)}_q(s,\alpha,\beta)=(-1)^n\lim_{t\searrow\llnv_1(\beta,\alpha)}I^{(n)}_q(t,\beta,\alpha).\]
From this we see that $I_q$ is $n$ times  continuously differentiable at $\llnv$ if and only if  functions 
	\[F_m(\alpha,\beta)=\lim_{\xi_1\searrow\llnv_1(\alpha,\beta)}I^{(m)}_q(\xi_1,\alpha,\beta)\]
are symmetric in $\alpha,\beta$ for all even $m\le n$ and antisymmetric for all odd $m\le n$.  The formulas we found above satisfy this up to $n=4$, therefore  $I_q$ is at least
four times continuously differentiable at $\llnv$ and expansion \eqref{Iq-exp} holds.\qed
%\bigskip

\footnotesize

\bibliographystyle{plain}

\bibliography{firasbib2010} %,growthrefs}
%\bibliography{firasbib2010}

\end{document}